\documentclass[12pt]{amsart}

\usepackage[margin = 1in]{geometry}
\usepackage{amsfonts, amsmath,amscd, amssymb, cite, color, latexsym, mathrsfs, mathtools,
slashed, stmaryrd, verbatim, wasysym }
\usepackage[all]{xy}
\usepackage[pdftex]{graphicx}
\usepackage{hyperref}
\usepackage[applemac]{inputenc}

\newtheorem{theorem}{Theorem}
\newtheorem{corollary}[theorem]{Corollary}
\newtheorem{definition}{Definition}
\newtheorem{example}{Example}
\newtheorem{lemma}[theorem]{Lemma}
\newtheorem{proposition}[theorem]{Proposition}

\numberwithin{equation}{section}
\numberwithin{theorem}{section}

\let\oldsqrt\sqrt
\def\sqrt{\mathpalette\DHLhksqrt}
\def\DHLhksqrt#1#2{%
\setbox0=\hbox{$#1\oldsqrt{#2\,}$}\dimen0=\ht0
\advance\dimen0-0.2\ht0
\setbox2=\hbox{\vrule height\ht0 depth -\dimen0}%
{\box0\lower0.4pt\box2}}

\newcommand{\df}[1]{\mathfrak{#1}}

\renewcommand{\bar}{\overline}
\renewcommand{\Re}{\operatorname{Re}}
\renewcommand{\hat}[1]{\widehat{#1}}

\newcommand{\wt}[1]{\widetilde{#1}}
\newcommand{\rest}[1]{\big\rvert_{#1}} 

\newcommand\lra{\longrightarrow}
\newcommand\xlra[1]{\xrightarrow{\phantom{x} #1 \phantom{x}}}

\newcommand\pa{\partial}

\newcommand\eps\varepsilon
\renewcommand\epsilon\varepsilon

\newcommand{\IC}[1]{\operatorname{\mathbf{IC}}_{\overline{#1}}^{\bullet}}

\renewcommand{\sc}[1]{\mathbf{ #1 }^{\bullet}} 

\newcommand\CI{{\mathcal{C}}^{\infty}}
\newcommand\CIc{{\mathcal{C}}^{\infty}_c}
\newcommand\CmI{{\mathcal{C}}^{-\infty}}

\newcommand{\lrpar}[1]{\left( #1 \right)}

\newcommand\ang[1]{\langle #1 \rangle}

\newcommand{\norm}[1]{\lVert #1 \rVert}

\newcommand\fib{\operatorname{---}} 

\newcommand\Diff{\operatorname{Diff}}

\newcommand\depth{\operatorname{depth}}
\newcommand\dR{\operatorname{dR}}
\newcommand\dvol{\operatorname{dvol}}
\newcommand\End{\operatorname{End}}
\newcommand\ev{\operatorname{even}}

\newcommand\id{\operatorname{id}}
\newcommand\Id{\operatorname{Id}}

\newcommand\odd{\operatorname{odd}}

\newcommand\pt{\operatorname{pt}}

\renewcommand\Re{\operatorname{Re}}
\newcommand{\reg}{ \mathrm{reg} }

\newcommand{\Sh}{\operatorname{Sh}}
\newcommand\sign{\operatorname{sign}}

\newcommand\spec{\operatorname{spec}}

\newcommand\supp{\operatorname{supp}}

\newcommand\Mand{\text{ and }}
\newcommand\Mas{\text{ as }}

\newcommand\Mforall{\text{ for all }}

\newcommand\Mforevery{\text{ for every }}

\newcommand\Mif{\text{ if }}

\newcommand\Min{\text{ in }}

\newcommand\Mon{\text{ on }}
\newcommand\Mor{\text{ or }}

\newcommand\Mst{\text{ s.t. }}

\newcommand\Mwhere{\text{ where }}
\newcommand\Mwith{\text{ with }}

\newcommand\paperintro%
        {%
         }
\newcommand\paperbody%
        {%
         }

\newcommand\bB{\mathbf{B}}

\newcommand\bG{\mathbf{G}}
\newcommand\bH{\mathbf{H}}

\newcommand\bL{\mathbf{L}}
\newcommand\bM{\mathbf{M}}
\newcommand\bN{\mathbf{N}}

\newcommand\bP{\mathbf{P}}
\newcommand\bQ{\mathbf{Q}}
\newcommand\bR{\mathbf{R}}
\newcommand\bS{\mathbf{S}}
\newcommand\bT{\mathbf{T}}
\newcommand\bU{\mathbf{U}}

\newcommand\bW{\mathbf{W}}

\newcommand\bOm{\mathbf{\Omega}}

\newcommand\bbB{\mathbb{B}}
\newcommand\bbC{\mathbb{C}}
\newcommand\bbD{\mathbb{D}}

\newcommand\bbF{\mathbb{F}}

\newcommand\bbH{\mathbb{H}}

\newcommand\bbN{\mathbb{N}}

\newcommand\bbQ{\mathbb{Q}}
\newcommand\bbR{\mathbb{R}}
\newcommand\bbS{\mathbb{S}}

\newcommand\bbZ{\mathbb{Z}}

\newcommand\cC{\mathcal{C}}
\newcommand\cD{\mathcal{D}}
\newcommand\cE{\mathcal{E}}

\newcommand\cH{\mathcal{H}}

\newcommand\cJ{\mathcal{J}}

\newcommand\cL{\mathcal{L}}
\newcommand\cM{\mathcal{M}}

\newcommand\cO{\mathcal{O}}
\newcommand\cP{\mathcal{P}}

\newcommand\cS{\mathcal{S}}
\newcommand\cT{\mathcal{T}}
\newcommand\cU{\mathcal{U}}
\newcommand\cV{\mathcal{V}}
\newcommand\cW{\mathcal{W}}

\newcommand\cY{\mathcal{Y}}

\newcommand\sC{\mathscr{C}}

\newcommand\sF{\mathscr{F}}

\newcommand\sM{\mathscr{M}}

\newcommand\sS{\mathscr{S}}
\newcommand\sT{\mathscr{T}}

\newcommand\sW{\mathscr{W}}

\newcommand\tH{\operatorname{H}}

\DeclareMathAlphabet{\mathpzc}{OT1}{pzc}{m}{it}

\begin{document}

\title{On the Hodge theory of stratified spaces}

\author{Pierre Albin}
\address{University of Illinois at Urbana-Champaign}
\email{palbin@illinois.edu}

\begin{abstract}
This article is a survey of recent work of the author, together with Markus Banagl, Eric Leichtnam, Rafe Mazzeo, and Paolo Piazza, on the Hodge theory of stratified spaces. We discuss how to resolve a Thom-Mather stratified space to a manifold with corners with an iterated fibration structure and the generalization of a perversity in the sense of Goresky-MacPherson to a mezzoperversity. We define Cheeger spaces and their signatures and describe how to carry out the analytic proof of the Novikov conjecture on these spaces. Finally we review the reductive Borel-Serre compactification of a locally symmetric space to a stratified space and describe its resolution to a manifold with corners.
\end{abstract}

\dedicatory{Dedicated to Steven Zucker on the occasion of his 65$^{\text{th}}$ birthday}

\maketitle

\section{Introduction}

The purpose of this paper is to describe some recent advances in the Hodge theory of pseudomanifolds.\\

Algebraic varieties and $L^2$-metrics have been integral aspects of Hodge theory from the very beginning. 
Hodge appealed to the Dirichlet principle \cite{Hodge:Dirichlet}  and the Levi parametrix method \cite{Hodge:Existence,Hodge:Harmonic} to find harmonic forms in each de Rham cohomology class and used these forms to  study the cohomology of complex projective algebraic varieties. Weyl \cite{Weyl:Address} referred to \cite{Hodge:Harmonic} as ``one of the great landmarks in the history of our science in the present century" (and corrected Hodge's existence proof in \cite{Weyl:Hodge}). An existence proof using purely Hilbert space methods was carried out by Gaffney \cite{Gaffney:Harmonic, Gaffney:Hodge} who, for instance, considered the exterior derivative $d$ on a smooth manifold $X$ as an unbounded operator on square integrable differential forms with domain
\begin{equation*}
	\cD_{\max}(d) = \{ \omega \in L^2(X;\Lambda^\bullet T^*X): d\omega \in L^2(X;\Lambda^\bullet T^*X) \},
\end{equation*}
where $d\omega$ is computed distributionally. This domain forms a complex whose cohomology, known as the $L^2$-cohomology of $X,$ was introduced independently in 
\cite{Zucker:Theorie} on a complete manifold and in \cite{Cheeger:Conic} on an incomplete manifold.

Algebraic varieties are often singular but Whitney \cite{Whitney, Kaloshin} showed that they are stratified spaces. This means that they can be decomposed into smooth manifold pieces of various dimensions, called {\em strata}, in such a way that different points on the same stratum have similar neighborhoods in the variety. In particular one of these pieces is an open dense smooth manifold. 
Endowing this `regular part' with a conic-type Riemannian metric, Cheeger studied the de Rham complex of differential forms in $L^2.$
He singled out a topological condition that guaranteed that the exterior derivative had an unambiguous definition as an unbounded operator,
every cohomology class had a harmonic representative, and the resulting cohomology groups satisfy Poincar\'e duality.
In the case of conic singularities he explained how, if this topological condition failed, one can choose ideal boundary conditions and still establish these results.
In this note I will review recent work \cite{ALMP11, ABLMP, ALMP13, ALMP13.2} extending Cheeger's results to general stratified {\em `pseudomanifolds'}.

Specifically, let $\hat X$ denote a pseudomanifold (a stratified space `without boundary') and endow its regular part $\hat X^{\reg}$ with a `wedge metric' $g.$
The exterior derivative $d$ and its formal adjoint $\delta$ make up the de Rham operator
\begin{equation*}
	\eth_{\dR} = d+ \delta: \CIc(\hat X^{\reg}; \Lambda^\bullet T^*\hat X^{\reg}) \lra  \CIc(\hat X^{\reg}; \Lambda^\bullet T^*\hat X^{\reg}).
\end{equation*}
There are two canonical extensions of $\eth_{\dR}$ to a closed operator on $L^2(\hat X^{\reg}; \Lambda^\bullet T^*\hat X^{\reg}).$ The first has domain
\begin{equation*}
	\cD_{\max}(\eth_{\dR}) = \{ \omega \in L^2(\hat X^{\reg}; \Lambda^\bullet T^*\hat X^{\reg}) : 
	\eth_{\dR}\omega \in L^2(\hat X^{\reg}; \Lambda^\bullet T^*\hat X^{\reg}) \}
\end{equation*}
where $\eth_{\dR}\omega$ is computed distributionally, and the second has domain
\begin{multline*}
	\cD_{\min}(\eth_{\dR}) = \{ \omega \in L^2(\hat X^{\reg}; \Lambda^\bullet T^*\hat X^{\reg}) : 
	\exists (\omega_n) \subseteq \CIc(\hat X^{\reg}; \Lambda^\bullet T^*\hat X^{\reg}) \Mst
	\omega_n \to \omega \Min L^2 \\
	\Mand (\eth_{\dR}\omega_n) \text{ is $L^2$-Cauchy} \}
\end{multline*}
and $\eth_{\dR}\omega$ can be computed both distributionally and as $\lim \eth_{\dR}\omega_n.$

Cheeger \cite{Cheeger:Conic} singled out a class of spaces, now known as {\em Witt spaces} \cite{Siegel}, where for a suitably chosen metric we have
\begin{equation*}
	\cD_{\min}(\eth_{\dR}) = \cD_{\max}(\eth_{\dR}).
\end{equation*}
This implies that many of the classical Hodge theory results will hold on $\hat X.$ For example, the $L^2$-cohomology is finite dimensional, represented by harmonic forms, and satisfies Poincar\'e duality. Moreover, Cheeger showed that this cohomology coincides with the (middle perversity) intersection homology of Goresky-MacPherson and so in particular one can study the signature of intersection homology using analytic tools.

The signature is an important invariant of smooth manifolds, e.g., in \cite{Gromov:Macroscopic}, Gromov says that it is ``{\em the }invariant, which can be matched in the beauty and power only by the Euler characteristic". The bordism invariance of the signature implies that it is a polynomial in the rational Pontryagin classes, and Hirzebruch's signature theorem identified the polynomial with the `$L$-class'. The signature is also invariant under (oriented) homotopies and Novikov proved that, for a simply connected oriented manifold, the only rational Pontryagin characteristic class that is (oriented) homotopy invariant is the `top order part' of the $L$-class. He defined `higher signatures' on non-simply connected manifolds and the celebrated, still open, Novikov conjecture claims that these are homotopy invariant and they are the only ones.

Due to the importance of the conjecture and the ubiquity of singular spaces, it is natural to study the Novikov conjecture on pseudomanifolds. A fruitful approach due to Lusztig \cite{Lusztig}, Mishchenko \cite{Mishchenko}, and Kasparov \cite{Kasparov:Novikov} involves the Atiyah-Singer index theorem for families of operators but with $C^*$-algebra bundles (see \cite{Rosenberg:Survey} for a recent survey of the Novikov conjecture). In \cite{ALMP11} we carried out this approach to the Novikov conjecture on Witt spaces. In the process we established Cheeger's results on Witt spaces by an approach that lent itself to allowing $C^*$-algebra bundles.

In \cite{ALMP13, ALMP13.2} we carried out the analytic approach to the Novikov conjecture on a larger class of stratified spaces. In \cite{Cheeger:Conic}, Cheeger considered non-Witt spaces with isolated conic singularities and explained that one could still define an $L^2$ cohomology theory satisfying Poincar\'e duality by choosing `ideal boundary conditions' at the cone points. In \cite{ALMP13} we explain how to use ideal boundary conditions on a general pseudomanifold to define a self-adjoint Fredholm domain for the de Rham operator, and an associated $L^2$-de Rham cohomology complex. We call the spaces that admit ideal boundary conditions for the signature operator {\em Cheeger spaces} in recognition of Cheeger's contributions. In \cite{ALMP13.2} we carry out the analytic approach to the Novikov conjecture on Cheeger spaces.

Our first step, described in \S\ref{sec:StratSpaces}, is to resolve the geometry in the sense of Melrose (e.g., \cite{Melrose:Fibs}) by replacing a stratified space $\hat X$ with a smooth manifold with corners $\wt X,$ so that the regular part of $\hat X$ becomes the interior of $\wt X.$ The idea of resolving a stratified space goes back at least as far as Thom \cite{Thom:Ensembles} and continues to be an object of study \cite{Ayala-Francis-Tanaka:Local}. In our formulation the stratification is replaced by a {\em boundary fibration structure}: each boundary hypersurface of $\wt X$ is the total space of a fiber bundle and these fibrations are compatible over corners. Collapsing the fibers of these fibrations in the appropriate order recovers the singular space from the smooth space, while `blowing-up' the singular strata constructs the smooth space from the singular one.  Having replaced one compactification $\hat X$ of the regular part by another one $\wt X,$ we next replace the cotangent bundle $T^*\wt X$ by a `rescaled bundle', the wedge cotangent bundle ${}^wT^*\wt X,$ that reflects the original degeneration of $\hat X.$ This replacement brings out the iterative structure of the de Rham operator as we describe fully in \S\ref{sec:ResolGeo}. A key fact is that these replacements do not change the regular part of the stratified space nor the cotangent bundle over the regular part; since analysis on a stratified space means analysis on the regular part, these convenient replacements have not changed the problem we wish to solve.

Once these changes are made, $\eth_{\dR}$ is amenable to study using the microanalytic tools developed in, e.g., \cite{Melrose-Mendoza, Melrose:APS, Mazzeo:Edge, Mazzeo-Vertman, Krainer-Mendoza:1stSurvey,Schulze:ItStrSing}. 
We explain how to use these tools first in the case of a single singular stratum in \S\ref{sec:DepthOne} and then for an arbitrary pseudomanifold in \S\ref{sec:DepthEll}. One of the features of our approach is to take advantage of the iterative structure of the space and the operator which means that most of the analytic difficulties are already present for a single singular stratum. 

If $\hat X$ has a single singular stratum $Y,$ then $\wt X$ is a smooth manifold with boundary and its boundary participates in a fiber bundle of smooth closed manifolds
\begin{equation*}
	Z - \pa \wt X \xlra{\phi} Y
\end{equation*}
with base $Y$ and typical fiber $Z.$
Using the microlocal tools mentioned above, we see in \S\ref{sec:MaxMinDom} that elements in $\cD_{\max}(\eth_{\dR})$ have (distributional) asymptotic expansions as they approach the boundary of $\wt X,$
\begin{equation*}
	\omega \in \cD_{\max}(\eth_{\dR}) \implies
	\omega \sim r^{-v/2}(\alpha(\omega) + dr \wedge \beta(\omega)) + \wt \omega \Mas r\to0
\end{equation*}
where $\wt \omega$ is a `remainder' and $\alpha(\omega),$ $\beta(\omega)$ are distributional forms on $Y,$
\begin{equation*}
	\alpha(\omega), \beta(\omega) \in \CmI(Y; \Lambda^\bullet T^*Y \otimes \cH^{\dim Z/2}(\pa \wt X/Y)),
\end{equation*}
with coefficients in $\cH^{\dim Z/2}(\pa\wt X/Y)$ the vertical Hodge cohomology of the boundary fibration.
These forms serve as `Cauchy data' for $\omega.$ For example we show that the minimal domain is determined by the vanishing of the Cauchy data.
Thus if $v$ is odd or $\tH^{\dim Z/2}_{\dR}(Z)=0$ (which is the `Witt condition' in this context) then the maximal domain is the same as the minimal domain.
Otherwise $\cH^{\dim Z/2}(\pa\wt X/Y)$ is a flat bundle and for every flat subbundle $W,$
\begin{equation*}
	\xymatrix{
	W \ar[rr]\ar[rd] & & \cH^{\dim Z/2}(\pa\wt X/Y) \ar[ld] \\
	& Y& },
\end{equation*}
we define a domain for $\eth_{\dR}$ by
\begin{multline*}
	\cD_{W}(\eth_{\dR}) = \{ \omega \in \cD_{\max}(\eth_{\dR}) : \\
	\alpha(\omega) \in \CmI(Y; \Lambda^\bullet T^*Y \otimes W), \quad
	\beta(\omega) \in \CmI(Y; \Lambda^\bullet T^*Y \otimes W^{\perp}) \}.
\end{multline*}
We refer to $W$ as a {\em mezzoperversity} and to $\cD_{W}(\eth_{\dR})$ as imposing {\em Cheeger ideal boundary conditions}.
With this domain, we see in \S\ref{sec:CIBC} that $\eth_{\dR}$ is a self-adjoint Fredholm operator with compact resolvent and hence the Hodge cohomology groups $\cH_{W}^*(\hat X)$ are finite dimensional.
In section \ref{sec:DepthOneDR} we explain how $W$ determines a domain for the exterior derivative, $\cD_W(d),$ and a Fredholm Hilbert complex in the sense of \cite{Bruning-Lesch:Hilbert}. The corresponding de Rham cohomology is denoted $\tH_{W,\dR}^*(\hat X)$ and it satisfies the Hodge theorem
\begin{equation*}
	\tH_{W, \dR}^*(\hat X) \cong \cH_W^*(\hat X).
\end{equation*}

In \S\ref{sec:DepthEll} we explain how these constructions can be carried out on a general pseudomanifold. This involves an inductive procedure: assuming we have carried out this analysis of Hodge and de Rham cohomology on pseudomanifolds of depth less than $k,$ we can carry out it out on spaces of depth $k.$ A mezzoperversity in this context consists of system of flat bundles, $\cW=(W_1, \ldots, W_{\ell}),$ one over each non-Witt stratum, satisfying a compatibility condition at the corners of $\wt X.$ To each mezzoperversity we assign finite dimensional de Rham and Hodge cohomology groups and we prove their equality,
\begin{equation*}
	\tH_{\cW,\dR}^*(\hat X) \cong \cH_{\cW}^*(\hat X).
\end{equation*}
If $\wt X$ is oriented we associate to each mezzoperversity $\cW$ a dual mezzoperversity $\bbD\cW$
and we show that the intersection pairing on differential forms descends to a non-degenerate pairing
\begin{equation}\label{eq:PDPairing}
	\tH_{\cW, \dR}^*(\hat X) \times \tH_{\bbD \cW, \dR}^*(\hat X) \lra \bbR
\end{equation}
yielding a generalized Poincar\'e duality in the sense of \cite{Goresky-MacPherson:IH}. 
This will be a `true' Poincar\'e duality if $\cW = \bbD \cW$ but there are topological obstructions to the existence of such self-dual mezzoperversities, notably the signatures of the fibers of the boundary fibrations. We refer to a pseudomanifold with a self-dual mezzoperversity as a {\em Cheeger space.}

Section \ref{sec:PropsCoho} is devoted to explaining that these cohomology theories share many interesting properties with the usual cohomology theories of smooth manifolds. For example we show that the de Rham cohomology is independent of the metric used to define it. The key fact is that the boundary conditions can be defined using  de Rham cohomology instead of Hodge cohomology, allowing us (inductively) to define the boundary conditions without reference to a metric. 

We also show that the cohomology is invariant under stratified homotopy equivalences. This is a bit tricky since pull-back by a homotopy equivalence need not define a bounded operator on $L^2.$ We make use of a clever construction of Hilsum and Skandalis \cite{Hilsum-Skandalis} to define a bounded pull-back map, and then explain how it acts on mezzoperversities and on cohomology.

On Cheeger spaces self-dual mezzoperversities define, through the quadratic form \eqref{eq:PDPairing}, a signature. In \S\ref{sec:PropSign} we explain how a self-dual mezzoperversity determines a domain for the signature operator on which it is Fredholm with index the signature of \eqref{eq:PDPairing}. If two Cheeger spaces, endowed with self-dual mezzoperversities, are bordant (with a Cheeger space bordism and an extension of the self-dual mezzoperversities) then the signature does not change. Together with a theorem of Banagl \cite{Banagl:LClass}, this implies that the signature of \eqref{eq:PDPairing} depends only on $\hat X$ and not on the self-dual mezzoperversity! Thus every Cheeger space has a signature that is invariant under Cheeger space bordism and stratified homotopy equivalences.

A beautiful approach to characteristic classes of Thom \cite{Thom:Classes} allows us to use this signature to define an $L$-class in homology. (This same approach was used by Goresky-MacPherson \cite{Goresky-MacPherson:IH} and Banagl \cite{Banagl:LClass}; Cheeger \cite{Cheeger:Spec} used $\eta$-invariants of the links to give a combinatorial formula for the $L$-class of a Witt space.) In \S\ref{sec:LClass} we define the $L$-class of a Cheeger space and use it to define the higher signatures.
In \cite{ALMP13.2} we carry out the analytic approach to the Novikov conjecture and prove that these higher signatures are stratified homotopy invariant for a large class of Cheeger spaces.

Cheeger was able to relate $L^2$-cohomology to the intersection cohomology of Goresky-MacPherson \cite{Goresky-MacPherson:IH, GM2}. Specifically, the complex $\cD_{\min}(d)$ has cohomology isomorphic to $\mathrm{IH}_{\bar m}^*(\hat X),$ the `lower middle perversity' intersection cohomology of $\hat X,$ and the complex $\cD_{\max}(d)$ has cohomology isomorphic to $\mathrm{IH}_{\bar n}^*(\hat X),$ the `upper middle perversity' intersection cohomology of $\hat X.$ These can be described as the hypercohomology groups of two sheaf complexes
\begin{equation*}
	\IC m, \quad \IC n
\end{equation*}
respectively.
Banagl \cite{Banagl:ExtendingIH} took the sheaf-theoretic description of intersection cohomology from \cite{GM2} and studied the category of all self-dual sheaf complexes consistent with $\IC m$ and $\IC n.$ In \cite{ABLMP} we considered the category of sheaf complexes consistent with $\IC m$ and $\IC n$ but not necessarily self-dual. As explained in \S\ref{sec:RefIH}, we showed that the cohomology groups $\tH^\bullet _{\cW, \dR}(\hat X)$ corresponding to a mezzoperversity are the hypercohomology groups of one of these sheaf complexes. Also self-dual mezzoperversities precisely correspond to the self-dual sheaves studied by Banagl. This yields the `de Rham theorem' to complement our `Hodge theorem' and completes the analogy between our new cohomology theories and the classical cohomology of closed manifolds.

The final two sections discuss the geometry of stratified spaces. Most of \S\ref{sec:StratSpaces} is devoted to showing the equivalence of Thom-Mather stratified spaces and manifolds with corners and boundary fibration structures. To exemplify the usefulness of this equivalence we include a description of a smooth de Rham complex that computes the intersection cohomology of any (Goresky-MacPherson) perversity function. One of the most interesting sources of stratified spaces is the reductive Borel-Serre compactification of a locally symmetric space, introduced by Zucker in \cite{Zucker:ReductiveLp}, which we discuss in the final section, \S\ref{sec:BorelSerre}. By resolving this stratified space we obtain a new compactification, the resolved Borel-Serre compactification, of a locally symmetric space to a manifold with corners and a boundary fibration structure.

$ $\\
{\bf Acknowledgements.} I thank Lizhen Ji for the invitation to write this article. I am grateful to my collaborators on the work being reported on, Markus Banagl, Eric Leichtnam, Rafe Mazzeo, Richard Melrose, and Paolo Piazza. I am also grateful for conversations on these topics with Francesco Bei and Jeff Cheeger. Melinda Lanius made suggestions that improved the exposition, and Steven Zucker was kind enough to read the manuscript closely and suggest many improvements. Part of this project was completed while the author was in residence at the Mathematical Sciences Research Institute in Berkeley, CA, made possible by the N.S.F. grant 0932078000. This work was partially supported by Simons Foundation grant \#317883.

\tableofcontents

\section{Background and context}\label{sec:Background}

As described in \cite[\S1.10]{Cheeger-Goresky-MacPherson}, in the mid 70's three independent sources gave rise to a new (co)homology theory for singular spaces:
\begin{itemize}
\item Goresky and MacPherson while working to develop a theory of characteristic classes of algebraic varieties developed intersection homology.
\item Cheeger's proof of the Ray-Singer conjecture on the equality of analytic torsion and Reidemeister torsion inspired his study of the $L^2$-cohomology of differential forms on the regular part of a variety.
\item A generalization by Deligne of work of Zucker on extending a variation of Hodge structure from the complement of a divisor to the entirety of a nonsingular complex algebraic variety.
\end{itemize}
The very fruitful history of this development is discussed in detail through 1989 in \cite{Kleiman:Development} and has continued apace. Our focus in this expository paper is on the $L^2$ cohomology of singular spaces, continuing the work of Cheeger \cite{Cheeger:Conic, Cheeger:Hodge, Cheeger-Goresky-MacPherson, Cheeger:Spec}.

\subsection{Stratified spaces}

Let us start by considering a particular example of the general situation we will be studying.
{\em Whitney's umbrella} is given by
\begin{equation*}
	\sW = \{ (x,y,z) \in \bbR^3 : x^2 -zy^2 = 0 \}. 
\end{equation*}
The gradient of the defining polynomial is $(2x, -2yz, -y^2),$ so the critical points make up the $\{ z\text{-axis} \}.$
At every point $p \in \sW\setminus \{ z\text{-axis} \},$ the inverse function theorem guarantees that there is a neighborhood of $p$ in $\sW$ that is diffeomorphic to an open ball in $\bbR^2.$ Thus, this `{\em regular part}' of $\sW,$ $\sW^{\reg},$ forms a two dimensional smooth manifold and we have a decomposition of $\sW$ into two smooth manifolds
\begin{equation*}
	\sW = \{ z\text{-axis} \} \cup \sW^{\reg}.
\end{equation*}
This is the first feature we will ask of a stratified space: it is a union of smooth manifolds, often of different dimensions.
Each of these manifolds is called a {\em stratum} and we talk about the collection of {\em strata}. 

Manifolds have a local homogeneity: every point has a neighborhood diffeomorphic to the same Euclidean ball.
Local homogeneity is a second feature we will require of each stratum of a stratified space.
Of course every point on the $z$-axis has a neighborhood on the $z$-axis diffeomorphic to an interval, but what we want is to look at neighborhoods of points on the $z$-axis as subsets of $\sW.$
The intersections of $\sW$ with the planes $z = c$ consist of
\begin{equation*}
	\begin{cases}
	\text{ two lines: } x= cy, x=-cy & \Mif c>0 \\
	\text{ one line: } x = 0 & \Mif c=0 \\
	\text{ one point: } x=y=0 & \Mif c<0
	\end{cases}
\end{equation*}
so we see that a point on the negative $z$-axis has a neighborhood in $\sW$ homeomorphic to an interval, while a point on the
positive $z$-axis has a neighborhood in $\sW$ homeomorphic to an interval (on the $z$-axis) times a small cross, which we think of as a truncated cone on four points, 
$(0,1) \times C(\{\text{ four points } \}).$
Thus to have local homogeneity we have to refine the decomposition of $\sW$ to 
\begin{equation*}
	\sW = \{ \text{ origin } \} \cup \{ \text{ negative $z$-axis } \} 
	\cup \{ \text{ positive $z$-axis } \} \cup (\sW \setminus  \{ z\text{-axis} \} ).
\end{equation*}
This exhibits $\sW$ as a disjoint union of four strata: one $0$-dimensional, two $1$-dimensional, and one $2$-dimensional, and each stratum is locally homogeneous.\\

We defer the formal definition of a stratified space to \S\ref{sec:StratSpaces} and meanwhile content ourselves with an informal description and some examples.
For us a stratified space will always mean a ``Thom-Mather stratification"; for other notions of stratification we refer the interested reader to the excellent surveys \cite{Hughes-Weinberger, Kloeckner}.
Roughly speaking this means that we have a topological space $\hat X$ with a decomposition into smooth manifold pieces called strata.
The stratum of largest dimension, $\hat X^{\reg},$ is called the regular part and is assumed to be dense; the other strata are called singular.
Each stratum $Y$ is `locally homogeneous' in that 
each point has a neighborhood in $\hat X$ homeomorphic to $\bbB^h \times C(Z),$ for some stratified space $Z$ called the {\em link} of $Y$ in $\hat X.$
Moreover each stratum $Y$ has a `tubular neighborhood' $\cT$ in $\hat X$ that is the total space of a fiber bundle
\begin{equation*}
	C(Z) \fib \cT \lra Y.
\end{equation*}
A stratified space is called a {\em pseudomanifold} if all of the singular strata have codimension at least two.
(The Whitney umbrella is a stratified space but not a pseudomanifold.)

A key feature of stratified spaces is that they show up even when our main object of interest is smooth.

\begin{example}
[Algebraic varieties]
Suppose $\hat X$ is the zero set of a finite set of polynomials $\{ f_1, \ldots, f_k \}.$
The Jacobian matrix of $F = (f_1, \ldots, f_k)$ has rank in $\{0,\ldots, k\}$ at each point in $\hat X.$ 
The sets 
\begin{equation*}
	Y_k = \{ \zeta \in \hat X: \mathrm{rank}(DF)(\zeta) = k \}
\end{equation*}
partition $\hat X$ into smooth manifolds, but do not always form a stratification.
Nevertheless, Whitney \cite{Whitney} showed that this partition can always be refined to yield a stratification (see \cite{Kaloshin}).
\end{example}

\begin{example} [Group Actions]
Another common example of a stratified set is the orbit space of a smooth manifold under the action of a compact Lie group.
Suppose that $M$ is a smooth manifold with an action of the compact Lie group $G,$ meaning a homomorphism $G \lra \Diff(M).$
Recall that to each point $\zeta \in M$ we associate an orbit $G\cdot\zeta$ and a stabilizer $G_{\zeta},$
\begin{equation*}
	G\cdot \zeta = \{ g(\zeta) \in M: g \in G \} \subseteq M, \quad
	G_{\zeta} = \{ g \in G : g(\zeta) = \zeta \} \subseteq G.
\end{equation*}
Note that the stabilizer group changes along an orbit since
\begin{equation*}
	G_{g(\zeta)} = g G_{\zeta} g^{-1},
\end{equation*}
suggesting that a better invariant is the conjugacy class of the stabilizer.
This is borne out by a classical result of Borel: 
the orbit space $M/G = \{ G \cdot \zeta : \zeta \in M\}$ inherits a smooth manifold structure along the projection $M \lra M/G$ if all of the stabilizer groups $G_{\zeta}$ are conjugate to each other in $G.$

In general the space $M/G$ is a stratified space. Given a subgroup $H$ of $G,$ let $[H]$ denote its conjugacy class in $G,$ and let
\begin{equation*}
	M_{[H]} = \{ \zeta \in M : G_{\zeta} \in [H]\}.
\end{equation*}
It's clear from the above that $M_{[H]}$ is closed under the action of $G$ and Borel's theorem implies that $M_{[H]}/G$ is a smooth manifold. 
This exhibits $M/G$ as a union of smooth manifolds. The fact that this decomposition is a stratification follows from the fact that the $M_{[H]}$ have equivariant tubular neighborhoods in $M.$

(In \cite{Albin-Melrose:Resolution} we discuss the resolution of a group action. The minimal isotropy type $X$ can be identified with the interior of a smooth manifold with corners $Y(M)$ endowed with a $G$-action and an equivariant map
\begin{equation*}
	\beta: Y(M) \lra M.
\end{equation*}
The manifold $Y(M)$ has an iterated fibration structure in the sense of \S\ref{sec:TotalRes} where all spaces have $G$-actions and all maps are equivariant. Moreover $Y(M)$ and the bases of its boundary fibrations have unique isotropy types and hence manifold orbit spaces. The orbit space of $Y(M)$ thus inherits an iterated fibration structure and can be identified with the resolution of the stratified space $M/G.$)

As explained in \cite{Hughes-Weinberger}, 
If $M$ and $G$ are not smooth, or the action of $G$ on $M$ is not smooth, then the orbit space $M/G$ need not admit a Thom-Mather stratification. This is the case even for a finite group and was one of the main motivations for introducing more general notions of stratification.
\end{example}

\begin{example}
[Moduli spaces]
A moduli space is a space, usually an open manifold, whose points represent (isomorphism classes of) geometric objects of some fixed kind.
Moduli spaces often involve singular spaces in two ways. On the one hand natural compactifications of moduli spaces are often singular. On the other hand, points in the boundary of the compactification can often be interpreted as moduli of singular versions of the original objects.

For example, let $\Sigma_g$ be a smooth oriented surface of genus $g,$ say $g>1,$ and let $C_g$ denote the space of complex structures on $\Sigma_g.$
The orientation preserving diffeomorphisms of $\Sigma_g,$ $\Diff^+(\Sigma_g),$ act naturally on $C_g$ and Teichm\"uller space is the quotient by the subgroup $\Diff^+_0(\Sigma_g)$ of diffeomorphisms isotopic to the identity,
\begin{equation*}
	\cT_g = C_g / \Diff^+_0(\Sigma_g).
\end{equation*}
This has a natural manifold structure, for which $\cT_g \cong \bbR^{6g-6}$ as real manifolds.
The {\em augmented Teichm\"uller space, $\hat\cT_g,$} is a {\em partial} compactification of $\cT_g$ obtained by attaching a stratified boundary to $\cT_g,$ with boundary hypersurfaces corresponding to allowing geodesic lengths to converge to zero. Thus $\hat \cT_g$ is singular and points on the singular strata correspond to Riemann surfaces with geodesics pinched to points.
The action of the `mapping class group', $\Diff^+(\Sigma_G)/\Diff^+_0(\Sigma_g),$ extends from $\cT_g$ to $\hat \cT_g$ and the quotient yields the Deligne-Mumford (stratified) compactification of $\cM_g,$ see e.g., \cite{Hubbard-Koch,Vakil:Notices}.
\end{example}

\begin{example}
[Unstable manifolds]
Following \cite{Nicolaescu:Morse},
suppose $M$ is a smooth manifold of dimension $m,$ $f:M\lra \bbR$ a Morse function, and $p\in M$ is a critical point of $f.$ Let $V$ be a gradient-like function on $M$ (e.g., the gradient of $f$ with respect to a Riemannian metric) and let $\Phi_t$ denote the flow generated by $-V.$
The stable/unstable manifolds associated to this data are respectively
\begin{equation*}
	W^{\pm}_p = \{ \zeta \in M : \lim_{t \to \pm\infty}\Phi_t(\zeta) = p \}.
\end{equation*}
These are both smooth manifolds and, up to homeomorphism,
\begin{equation*}
	W_p^- \cong \bbR^{\lambda}, \quad W_p^+ \cong \bbR^{m-\lambda},
\end{equation*}
where $\lambda = \lambda(f, p)$ is the Morse index of $f$ at $p.$
The flow satisfies the Morse-Smale transversality condition if the stable and unstable manifolds associated to all of the critical points of $f$ intersect each other transversely.

Let 
\begin{equation*}
	Y_{k} = \bigcup \{ W_p^- : \lambda(f,p) = k \}.
\end{equation*}
Nicolaescu shows in \cite[Chapter 4]{Nicolaescu:Morse} that the decomposition of $M$ into $\{ Y_k \}$ is a stratification if and only if the flow satisfies the Morse-Smale transversality condition.
\end{example}

\begin{example}
[Homology classes]
If $X$ is a topological space, we say that a homology class $h\in H_*(X)$ is realizable by a smooth manifold if there is a closed orientable manifold $M$ and a continuous map $f:M \lra X$ such that
\begin{equation*}
	h= f_*[M].
\end{equation*}
Steenrod's problem consists of describing the realizable homology classes when $X$ is a smooth manifold.
Thom \cite{Thom:Quelques} proved that all classes are realizable with $\bbZ_2$ or $\bbQ$ coefficients.
He showed that all elements of $H_j(X;\bbZ)$ are realizable for $j \leq 6$ and for $\dim X-2 \leq j < \dim X$ and provided an explicit example of a non-realizable class.
Goresky showed that you can represent every homology class of a smooth manifold by a stratified subset \cite{Goresky:Whitney, Murolo:Whitney}.

(Thom's results do not hold if we impose conditions on $f$ beyond being continuous, e.g., \cite{Grant-Szucs} shows that there are homology classes with $\bbZ_2$ coefficients that can not be realized by immersed submanifolds. See also \cite{Rudyak, Brumfiel, Rees-Thomas, Bohr-Hanke-Kotschick}.)
\end{example}

\begin{example}[Locally symmetric spaces]
Let $G$ be a connected semisimple algebraic group defined over $\bbQ,$ $G(\bbR)$ its real points, regarded as a Lie group, $K$ a maximal compact subgroup of $G,$ $\Gamma$ a torsion-free arithmetic subgroup, and
\begin{equation*}
	X = \Gamma \diagdown G(\bbR) \diagup K
\end{equation*}
the corresponding locally symmetric space. This is generally non-compact but there are many natural compactifications of $X,$ see \cite{Borel-Ji:Book}.
The reductive Borel-Serre compactification (see \cite{Zucker:ReductiveLp}) is a stratified space obtained by adding one stratum per $\Gamma$-conjugacy class of rational parabolic subgroups $P$ of $G.$
We will review this and related compactifications in \S\ref{sec:BorelSerre}.
\end{example}

\subsection{Intersection homology and Witt spaces}

One of the primary motivations for \cite{Goresky-MacPherson:IH} was a question of Sullivan: what is the largest class of spaces with singularities to which the signature of manifolds extends as a cobordism invariant?

Recall that on a smooth oriented manifold Poincar\'e duality defines a non-degenerate pairing on (co)homology whose signature is the signature of the manifold.
The singular homology of a stratified space does not satisfy Poincar\'e duality.

\begin{example}
If $M$ is a smooth manifold, its (unreduced) suspension is the quotient space
\begin{equation*}
	SM = (M \times [0,1])/ \{ (m_1, 0) \sim (m_2, 0) \Mand (m_1, 1) \sim (m_2, 1) \Mforall m_i \in M\}
\end{equation*}
which is naturally a pseudomanifold with singular stratum the two cone points.
A Mayer-Vietoris argument shows that 
\begin{equation*}
	\wt H_{k}(SM) \cong \wt H_{k-1}(M), k\geq 0.
\end{equation*}
In particular if $M$ is a two dimensional torus,
\begin{equation*}
	H_0(SM) =\bbZ, \quad H_1(SM) = \bbZ, \quad H_2(SM) = \bbZ \oplus \bbZ, \quad H_3(SM) = \bbZ
\end{equation*}
and we see that Poincar\'e duality can not hold.
\end{example}

Goresky-MacPherson came up with the brave idea of redefining homology to recover Poincar\'e duality. The key is to keep track of the failure of transversality of the chains with the singular set. We keep track of this with a perversity function $\bar p,$ a function from $\{ 2, 3, \ldots\}$ to $\bbN_0$ such that $p(c)$ and $q(c) = 2-c-p(c)$ are both non-negative and non-decreasing. The perversity $\bar q$ is known as the dual perversity to $\bar p.$ (For more general perversity functions, see \cite{Friedman:Intro}.)

Intersection homology can be defined analogously to singular homology but maintaining control of the dimensions of the intersections of chains and singularities \cite{King}.
Let $\hat X$ be a stratified space of dimension $n$ and let $Y_{k}$ be the union of the strata of dimension $k.$
Given a perversity function $\bar p$ 
let the admissible chains, $AC_i^{\bar p}(\hat X),$ be the free group with generators the singular $i$-simplices
\begin{equation*}
	f:\Delta^i \lra \hat X \Mst f^{-1}(Y_{n-j}) \subseteq \text{ $i-j+p(j)$-skeleton of }\Delta^i \Mforall j,
\end{equation*}
and let
\begin{equation*}
	SC_i^{\bar p}(\hat X) = \{ f \in AC_i^{\bar p}(\hat X): \pa f \in AC_{i-1}^{\bar p}(\hat X) \}.
\end{equation*}
The intersection homology groups of perversity $\bar p$  are the homology groups of the complex $SC_*^{\bar p}(\hat X),$
\begin{equation*}
	\mathrm{IH}_i^{\bar p}(\hat X) = H_i(SC_*^{\bar p}(\hat X)).
\end{equation*}
(For the definition of intersection homology through complexes of sheaves see our \S\ref{sec:RefIH} below.)

\begin{example}
Consider again the suspension of a smooth $m$-dimensional manifold $M,$ $SM.$ Here $Y_0$ consists of two points and the regular part is 
\begin{equation*}
	X = SM\setminus Y_0 = M \times (0,1).
\end{equation*}
A perversity consists of a single integer $p(m+1)\in [0,m-1],$ and the intersection homology is given by
\begin{equation*}
	\mathrm{IH}^{\bar p}_k(SM) = 
	\begin{cases}
	H_{k}(M) & k<m-p(m+1) \\
	0 & k=m-p(m+1) \\
	H_{k-1}(M) & k>m-p(m+1)
	\end{cases}
\end{equation*}
The intuition is that allowable chains avoid the suspension points in low dimensions, and so really map into $X,$ but not in high dimensions (cf. \cite[proof of Proposition 4.7.2]{Kirwan-Woolf}).
\end{example}

Goresky-MacPherson proved that the intersection pairing induces a non-degenerate pairing
\begin{equation*}
	\mathrm{IH}_i^{\bar p}(\hat X) \times 
	\mathrm{IH}_i^{\bar q}(\hat X) \lra \bbZ.
\end{equation*}
In general different perversities give rise to different groups. 
The closest dual perversities,
\begin{equation*}
	\bar m = (0,0,1,1,2,2,\ldots), \quad
	\bar n = (0,1, 1, 2, 2, \ldots),
\end{equation*}
are known respectively as the upper and lower middle perversities and differ only with respect to strata of odd codimension.
Thus if $\hat X$ has only even codimensional strata (e.g., a complex variety) then $\mathrm{IH}^{\bar m}_{*}(\hat X)$ is dually paired with itself.

More generally, Siegel \cite[Theorem I.3.4]{Siegel} showed that if every even dimensional link $Z$ satisfies $\mathrm{IH}^{\dim Z/2}_{\bar m}(Z)=0,$ then 
$\mathrm{IH}^{\bar m}_{*}(\hat X)$ is dually paired with itself. He called these spaces {\em Witt spaces} and proved that the signature is a Witt-bordism invariant, i.e., if two Witt spaces are cobordant via a Witt space (with boundary) then their signatures coincide.
Banagl \cite{Banagl:ExtendingIH} managed to extend Poincar\'e duality to stratified spaces that are not necessarily Witt by studying the category of sheaves that are self-dual and compatible with the sheaf-theoretic approach to intersection homology (see \S\ref{sec:RefIH}). Spaces that carry a self-dual sheaf of this type (not every space does) are known as {\em L-spaces}. Banagl mentions in \cite{Banagl:ExtendingIH} that the `idea of employing Lagrangian subspaces in order to obtain self-duality' is invoked in unpublished work of Morgan and is also present in Cheeger's ideal boundary conditions. Indeed in \cite{ABLMP} we show that a Thom-Mather stratified space is an L-space if and only if it is a Cheeger space.

\subsection{$L^2$ cohomology}

If $(M,g)$ is a Riemannian manifold, the metric defines a volume form and a bundle metric on the exterior powers of the cotangent bundle, $\Lambda^\bullet T^*M.$
The space of $L^2$ differential forms can be defined to be the closure of $\CIc(M; \Lambda^\bullet T^*M)$ with respect to the norm
\begin{equation*}
	\norm{\omega}_{L^2}^2 = \int_M |\omega(\zeta)|_g^2 \; \dvol(\zeta)
\end{equation*}
and is then a Hilbert space. The exterior derivative is an unbounded operator 
\begin{equation*}
	d: \CIc(M;\Lambda^\bullet T^*M) \lra  \CIc(M;\Lambda^\bullet T^*M)
\end{equation*}
and its formal adjoint is the  operator
\begin{equation*}
	\delta: \CIc(M;\Lambda^\bullet T^*M) \lra  \CIc(M;\Lambda^\bullet T^*M)
\end{equation*}
characterized by
\begin{equation*}
	\ang{d\omega, \eta}_{L^2} = \ang{\omega, \delta \eta}_{L^2}.
\end{equation*}

If $\omega$ and $\eta$ are $L^2$ differential forms we say that the $d\omega =\eta$ holds weakly (or distributionally) if
\begin{equation*}
	\ang{\omega, \delta \alpha}_{L^2} = \ang{\eta, \alpha}_{L^2} \Mforall \alpha \in \CIc(M; \Lambda^\bullet T^*M).
\end{equation*}
The maximal domain of $d$ is 
\begin{equation*}
	\cD_{\max}(d) = \{ \omega \in L^2(M;\Lambda^\bullet T^*M) : d\omega \in L^2(M;\Lambda^\bullet T^*M)
\end{equation*}
where $d\omega$ is computed distributionally. 
If we denote the $k$-forms in the maximal domain of $d$ by
\begin{equation*}
	\cD_{\max}(d)_{k} = \cD_{\max}(d) \cap L^2(M;\Lambda^kT^*M)
\end{equation*}
then we obtain a complex
\begin{equation*}
	0 \lra \cD_{\max}(d)_0 \xlra{d}
	\cD_{\max}(d)_1 \xlra{d} \ldots \xlra{d}
	\cD_{\max}(d)_{\dim M} \lra 0.
\end{equation*}
The cohomology of this complex is the $L^2$-cohomology of $M$ (see, e.g., \cite{Carron:L2HarmForms} for an introduction to $L^2$-cohomology).

If $(M,g)$ is not complete then typically there are many closed domains for $d$ other than $\cD_{\max}(d).$ A {\em Hilbert complex for $d$} consists of a choice of domains $\cD$ such that $(d, \cD)$ is a closed operator and 
\begin{equation*}
	0 \lra \cD(d)_0 \xlra{d}
	\cD(d)_1 \xlra{d} \ldots \xlra{d}
	\cD(d)_{\dim M} \lra 0
\end{equation*}
forms a complex. An excellent study of Hilbert complexes is carried out in \cite{Bruning-Lesch:Hilbert}.\\

One of the motivations behind \cite{ALMP13, ALMP13.2, ABLMP} was to find an analytic version of Banagl's cohomology groups by defining appropriate Hilbert complexes.
As mentioned above, Cheeger showed in \cite[\S 6]{Cheeger:Hodge} that the $L^2$-cohomology groups coincide with intersection cohomology with perversity $\bar m.$ 
He also came across the Witt condition independently in \cite{Cheeger:Conic} and described ideal boundary conditions for isolated conic singularities as a way of recovering Poincar\'e duality. We explain below (\S\ref{sec:CIBC}) how to use the analytic properties of the de Rham operator $\eth_{\dR}$ to impose ideal boundary conditions, and then how to use the Kodaira decomposition and these ideal boundary conditions to define Hilbert complexes (\S\ref{sec:DepthOneDR}).

Let us mention some related but different approaches. Signatures for general pseudomanifolds are defined in \cite{Hunsicker:Hodge, Friedman-Hunsicker, Bei:PD}. These can not be bordism invariant (since the cone over a smooth manifold would make the manifold bordant to a point)
and it would be interesting to understand their relation with Cheeger spaces. 
There are also analytic studies of other perversity functions in, e.g., \cite{Hunsicker:Hodge, Bei:GralPerv}.\\

We mention a few instances of the vast literature on the $L^2$ cohomology of complete metrics.
Already in \cite{APS0}, Atiyah, Patodi, and Singer explain that the `correct analytical context' for their signature theorem is the signature of the $L^2$ harmonic forms on a manifold with a cylindrical end.
Zucker \cite{Zucker:DegCoeff}, in order to analyze the variation of Hodge structure on a degenerating family of curves, worked out the $L^2$ cohomology on a manifold with ends modeled on the Poincar\'e metric of a punctured disk; this {\em is} the intersection cohomology with coefficients.
Still in the setting of complete metrics, Zucker \cite{Zucker:L2Coho} conjectured that the $L^2$ cohomology of a Hermitian locally symmetric space arising from the real points of a semisimple connected algebraic group defined over $\bbQ$ would equal the intersection homology of its  Baily-Borel-Satake compactification (see \cite{Borel-Ji:Book}), which is a singular space. This celebrated conjecture was verified independently by Looijenga and Saper-Stern \cite{Looijenga:Zucker, Saper-Stern:Zucker}.
Among many others, let us list some of the results which are closest in spirit to the subject and methods here
\begin{itemize}
\item the $L^2$ cohomology of spaces with asymptotically cylindrical ends was determined by Melrose \cite{Melrose:APS}, 
\item cusp ends by M\"uller \cite{Muller:SpecThy}, 
\item asymptotically geometrically finite hyperbolic quotients in \cite{Mazzeo:Hodge, Mazzeo-Phillips} by Mazzeo and Phillips, 
\item flat ends by Carron \cite{Carron:FlatEnds}, 
\item asymptotically fibered cusp ends and asymptotically fibered boundary ends by Hausel, Hunsicker, and Mazzeo \cite{Hausel-Hunsicker-Mazzeo}, 
\item foliated analogues of the latter by Gell-Redman and Rochon \cite{GellRedman-Rochon}, and 
\item iterated fibered cusp metrics on stratified spaces by Hunsicker and Rochon \cite{Hunsicker-Rochon}.
\end{itemize}

Finally we point out that closely related to the $L^2$-cohomology is the study of index theorems. Cheeger's study of $L^2$-cohomology was part of his wider study of the spectral geometry of singular spaces in \cite{Cheeger:Spec} and found application to index theory problems even on smooth spaces cf. \cite{Cheeger:APSW, Bismut-Cheeger:Ad, Bismut-Cheeger:FamI, Bismut-Cheeger:FamII, Bismut-Cheeger:FamIII}. There has been a lot of work on index theory on spaces with isolated conic singularities e.g., \cite{Chou:Dirac, Chou:Criteria, Bruning-Seeley:IndThm, Lesch:Fuchs} and recently also on index theory on spaces with a single singular stratum \cite{Atiyah-LeBrun, Bruning:SignOp, Cheeger-Dai, Albin-GellRedman}.

\section{Depth one stratifications}\label{sec:DepthOne}

A pseudomanifold $\hat X$ is said to have {\em depth one} if there is a single singular stratum. Thus there is one regular stratum, $\hat X^{\reg},$ and one singular stratum, $Y.$
We denote the link of $\hat X$ at a point in $Y$ by $Z,$ and we will denote dimensions by
\begin{equation*}
	n = \dim \hat X^{\reg}, \quad h = \dim Y, \quad v = \dim Z
\end{equation*}
related by $n = h+v+1.$
We assume that $\hat X$ is compact and, for notational convenience, also that $\hat X^{\reg},$  $Y,$ and $Z$ are connected. 
This situation presents most of the analytic complications of the general case but has the advantage that the geometry is much simpler than the general case. 
There are many treatments of analysis in this setting, for example \cite{Mazzeo-Vertman, Gil-Krainer-Mendoza:Wedge, Krainer-Mendoza:Friedrichs, 
Krainer-Mendoza:BVP1st, Krainer-Mendoza:EllSys,Krainer-Mendoza:Kernel,Schulze:PsiDO}.

Our main object of study is the operator $d+\delta$ on differential forms on $\hat X^{\reg}.$ We will see that the analysis of this operator is facilitated by first resolving the geometry.

\subsection{Resolving the geometry} \label{sec:ResolGeo}

Let us describe our geometric setup in detail. We have a smooth open manifold $\hat X^{\reg}$ and a compactification of it to a singular space 
\begin{equation*}
	\hat X = \hat X^{\reg} \cup Y.
\end{equation*}
The smooth manifold $Y$ has a tubular neighborhood $\cT\subseteq \hat X$ that fibers over $Y$ with fiber the cone over a smooth manifold $Z,$
\begin{equation*}
	C(Z) - \cT \lra Y.
\end{equation*}
Let $r$ be a radial variable for the cone and let 
\begin{equation*}
	\wt X = \hat X \setminus \{ r \leq \eps \}.
\end{equation*}
For sufficiently small $\eps>0$ this is a smooth manifold with boundary $\pa \wt X = \{ r=\eps\}$ and, up to diffeomorphism, is independent of the choice of $\eps.$
The boundary of $\wt X$ inherits a smooth fiber bundle structure
\begin{equation*}
	Z \fib \pa\wt X \xlra{\phi} Y.
\end{equation*}
The interior of $\wt X$ is naturally identified with $\hat X^{\reg}$ and so $\wt X$ is a compactification of $\hat X^{\reg}$ to a smooth manifold with boundary.
A key fact for us is that analysis on $\hat X$ refers to analysis on $\hat X^{\reg}$ and hence can equally well be carried out on $\wt X,$ with the advantage that this is a smooth manifold.

A {\em rigid wedge metric} is a Riemannian metric on $\hat X^{\reg} = \wt X^{\circ}$ that, in a collar neighborhood of $\partial\wt X,$ takes the form
\begin{equation*}
	g = dr^2 + r^2g_Z + \phi^*g_Y
\end{equation*}
with $g_Z$ and $g_Y$ independent of $r.$
(Precisely, it has this form after we choose a diffeomorphism of a collar neighborhood with $[0,\eps)_r \times \pa\wt X$ and then we extend $\phi$ to this product in the obvious way.)
This metric is ill-behaved as we approach the boundary of $\wt X,$ but there is a natural way to resolve this.
Let us denote the induced bundle metric on the cotangent bundle still by $g,$ and define
\begin{equation*}
	\cE = \{ \omega \in \CI(\wt X; T^*\wt X) : |\omega|_g \text{ is bounded } \}.
\end{equation*}
It turns out that the space $\cE$ is the space of sections of a new bundle and $g$ defines a non-degenerate bundle metric on this bundle.

Precisely, there is a bundle which we call the {\em wedge cotangent bundle} and denote ${}^w T^*\wt X,$ together with a map
\begin{equation*}
	j: {}^w T^*\wt X \lra T^*\wt X,
\end{equation*}
such that
\begin{equation*}
	j\circ \CI(\wt X; {}^w T^*\wt X) = \cE \subseteq \CI(\wt X; T^*\wt X).
\end{equation*}
We can show this either abstractly, through the Serre-Swan theorem, or concretely as in \cite[Proposition 8.1]{Melrose:APS} by showing that the vector spaces
\begin{equation*}
	{}^w T^*_\zeta \wt X = \cE/ (\cJ_\zeta \cdot \cE),
\end{equation*}
where $\cJ_{\zeta} = \{ f \in \CI(\wt X): f(\zeta) =0 \},$ inherit a vector bundle structure from $T^*\wt X$ and a map $j$ from the inclusion of $\cE$ into $\CI(\wt X; T^*\wt X).$
The map $j$ is an isomorphism over $\wt X^{\circ},$ but over $\pa \wt X$ fits into a short exact sequence
\begin{equation*}
	0 \lra T^*Z \lra {}^w T^*\wt X\rest{\pa \wt X} \xlra{j}  {}^w N^*\pa \wt X \cong N^*\pa \wt X \oplus T^*Y \lra 0
\end{equation*}
of bundles over $\pa \wt X.$ 

In local coordinates trivializing the fiber bundle $\phi$ in a neighborhood of a point $q \in \partial \wt X,$
\begin{equation}\label{eq:FibCoords}
	r, y_1, \ldots, y_h, z_1, \ldots, z_v,
\end{equation}
a section of $T^*\wt X$ is locally spanned by $dr,$ $dy_i,$ $dz_j,$ while a section of ${}^w T^*\wt X$ is locally spanned by
\begin{equation*}
	dr, dy_1, \ldots, dy_h, r\; dz_1, \ldots, r\; dz_v.
\end{equation*}
Crucially, $r\; dz_j$ does not vanish at $\pa\wt X$ as a section of ${}^w T^*\wt X$ since it is not equal to $r$ times a section of ${}^w T^*\wt X.$

The replacement of $T^*\wt X$ by ${}^w T^*\wt X$ is very convenient. On the one hand, as anticipated above, the definition of $\cE$ implies that the metric $g$ extends from $\wt X^{\circ}$ to a non-degenerate metric over all of $\wt X.$ On the other hand, let us consider the effect of this change on the behavior of $d+\delta$ near $\wt X.$
For simplicity of comparing $d+\delta$ in these two settings, let us briefly assume that $Y$ is a single point, i.e., that $\hat X$ has an isolated conic singularity.
After a choice of collar neighborhood of the boundary $c(\pa \wt X),$ we have
\begin{equation*}
\begin{gathered}
	\CI(c(\pa \wt X); \Lambda^i (T^*\wt X)) = \CI(\pa \wt X; \Lambda^i T^*Z) \oplus dr \wedge \CI(\pa \wt X; \Lambda^{i-1} T^*Z) \\
	\CI(c(\pa \wt X); \Lambda^i ({}^wT^*\wt X)) = \CI(\pa \wt X; r^i \Lambda^i T^*Z) \oplus dr \wedge \CI( \pa \wt X; r^{i-1} \Lambda^{i-1} T^*Z)
\end{gathered}
\end{equation*}
and correspondingly
\begin{equation*}
\begin{gathered}
	\Mon \CI(c(\pa \wt X); \Lambda^i (T^*\wt X)), \quad
	d+ \delta = 
	\begin{pmatrix}
	d^Z + \tfrac1{r^2}\delta^Z & -\pa_r +\tfrac{2i - v-1}r \\
	\pa_r & -d^Z - \tfrac1{r^2}\delta^Z
	\end{pmatrix},
	\\
	\Mon 
	\CI(c(\pa \wt X); \Lambda^i ({}^wT^*\wt X)), \quad
	d+ \delta = 
	\begin{pmatrix}
	\tfrac1r(d^Z + \delta^Z) & -\pa_r +\tfrac{i - v}r \\
	\pa_r + \tfrac ir & -\tfrac1r(d^Z +\delta^Z)
	\end{pmatrix}.
\end{gathered}
\end{equation*}
There are two things to note about these expressions. First the form of the metric is going to lead to operators with coefficients that blow-up at the boundary no matter how we look at it. Secondly, when we work with ${}^wT^*\wt X,$ the link $Z$ enters the description of the de Rham operator $d+\delta$ only through its de Rham operator $d^Z + \delta^Z.$ This is very convenient for inductive arguments!

(A related approach, espoused in, e.g., \cite{Bruning-Seeley:IndThm, Bruning:SignOp} is to conjugate the de Rham operator on exterior powers of $T^*\wt X$ with a suitable scaling operator; we realize this rescaling at the level of bundles.)

\subsection{Model operators of the de Rham operator}

Having resolved the geometry by replacing the singular space $\hat X$ with the smooth manifold with boundary $\wt X,$ and resolved the degeneracy of the metric 
by replacing the cotangent bundle $T^*\wt X$ with the wedge cotangent bundle ${}^w T^*\wt X,$ we now proceed to study the de Rham operator $d+ \delta.$

\subsubsection{Cohomology of fiber bundles}\label{sec:CohoFiberBdles}
Since the boundary of $\wt X$ is the total space of a fiber bundle, let us briefly review the behavior of $d$ and $\delta$ on fiber bundles (cf. \cite{Bismut-Lott}, \cite[\S10.1]{Berline-Getzler-Vergne} \cite[\S5.3]{Hausel-Hunsicker-Mazzeo}). 
Let
\begin{equation*}
	F \fib H \xlra{\varphi} B
\end{equation*}
be a fiber bundle of smooth closed manifolds. 
The functions $\varphi^*\CI(B) \subseteq \CI(H)$ are called horizontal, and a vector field on $H$ is called vertical if it annihilates horizontal functions. 
Vertical vector fields are the sections of a vector bundle $TH/B \lra H,$ which participates in a short exact sequence of bundles
\begin{equation*}
	0 \lra 
	TH/B \lra TH \lra \varphi^*TB \lra 0.
\end{equation*}
The Lie bracket of two vertical vector fields is again a vertical vector field and hence the Cartan formula defines a vertical exterior derivative on the exterior powers of the bundle $T^*H/B$ dual to $TH/B,$
\begin{equation*}
	d^Z: \CI(H;\Lambda^k T^*H/B) \lra \CI(H;\Lambda^{k+1}T^*H/B).
\end{equation*}
The cohomology of the resulting complex, $\tH_{\dR}^*(H/B),$ is called the vertical de Rham cohomology and forms a vector bundle over $B$ with typical fiber $\tH_{\dR}^*(F).$
Indeed, given a transition function $G: \cU \lra \Diff(F)$ for $\varphi,$ defined on a neighborhood $\cU \subseteq B,$ there is an induced map on cohomology 
\begin{equation*}
	G^*: \cU \lra \End(\tH_{\dR}^*(F))
\end{equation*}
giving a transition function for $\tH_{\dR}^*(H/B).$ Moreover if $\cU$ is connected then homotopy invariance of de Rham cohomology implies that this transition function is constant, which exhibits $\tH_{\dR}^*(H/B)$ as a flat vector bundle.

Now, let $g_H = g_{H/B} + \varphi^*g_B$ be a Riemannian submersion metric on $H$ with corresponding splittings
\begin{equation*}
\begin{gathered}
	TH = TH/B \oplus \varphi^*TB, \quad T^*H = T^*H/B \oplus \varphi^*T^*B, \\
	\Lambda^k T^*H = \bigoplus_{i+j=k} \Lambda^{i,j}T^*H, \Mwhere
	\Lambda^{i,j} T^*H = \Lambda^i T^*H/B \otimes \Lambda^j \varphi^*T^*B.
\end{gathered}
\end{equation*}
We refer to $i$ in $\Lambda^{i,j} T^*H$ as the {\em vertical form degree}.
The exterior derivative interacts with this splitting via
\begin{equation*}
\begin{gathered}
	d:\CI(H;\Lambda^{i,j}T^*H) \lra 
	\CI(H;\Lambda^{i+1,j}T^*H) \oplus \CI(H;\Lambda^{i, j+1}T^*H) \oplus \CI(H;\Lambda^{i-1,j+2}T^*H) \\
	d = d^Z + \wt d^B + R
\end{gathered}
\end{equation*}
where the first summand is the vertical exterior derivative, the second summand extends the exterior derivative of $B$ from sections of $\Lambda^{0,*}T^*H,$ and the third is a tensorial map built up from the curvature of $\varphi.$ 
In the special case where $H = F \times B$ and $\varphi$ is the projection onto $B,$ this decomposition is just $d = d^Z + d^B.$
There is a corresponding splitting of the adjoint of $d,$
\begin{equation*}
\begin{gathered}
	\delta:\CI(H;\Lambda^{i,j}T^*H) \lra 
	\CI(H;\Lambda^{i-1,j}T^*H) \oplus \CI(H;\Lambda^{i, j-1}T^*H) \oplus \CI(H;\Lambda^{i+1,j-2}T^*H) \\
	\delta = \delta^Z + \wt \delta^B + R^*
\end{gathered}
\end{equation*}
which reduces to $\delta^Z + \delta^B$ when $H = F \times B.$

Next let us point out that $\wt d^B$ induces a flat connection, $\nabla^{\tH},$ on $\tH^\bullet _{\dR}(H/B)\lra B$ compatible with the flat structure described above.
Indeed, we have various relationships among $d^Z,$ $\wt d^B,$ and $R$ stemming from the fact that their sum squares to zero, including
\begin{equation*}
	d^Z \wt d^B + \wt d^B d^Z = 0, \quad
	(\wt d^B)^2 + d^Z R + Rd^Z =0.
\end{equation*}
If $\omega \in \CI(H;\Lambda^k T^*H/Z)$ then
\begin{equation*}
	d^Z \wt d^B \omega = - \wt d^B d^Z\omega
\end{equation*}
shows that $\wt d^B$ preserves both the image and the null space of $d^Z,$ and hence induces a map on $\tH^\bullet _{\dR}(H/B).$
The Leibnitz rule of $d$ implies that
\begin{equation*}
	\wt d^B( (\varphi^*f) \omega) = (\varphi^* df) \wedge \omega + \varphi^*f  \wt d^B \omega
\end{equation*}
which shows that the induced map on $\tH^\bullet _{\dR}(H/B)$ is a connection.
Finally, if $d^Z\omega=0,$ then 
\begin{equation*}
	(\wt d^B)^2\omega = d^Z (-R\omega)
\end{equation*}
which shows that the induced connection on $\tH^\bullet _{\dR}(H/B)$ is flat.
Since this connection coincides with $d^B$ over trivializations of the fiber bundle $\varphi,$ it follows that this flat structure is the same as the one above.

Finally note that the Hodge isomorphism on the fibers of $\varphi,$ $\tH^\bullet _{\dR}(F) \cong \cH^\bullet (F),$ can be used to import the flat bundle structure to the vertical Hodge cohomology
\begin{equation*}
	\cH^\bullet (H/B) \lra B.
\end{equation*}
We denote the flat connection on this bundle by $\nabla^{\cH}.$

\subsubsection{Model operators} \label{sec:ModelOps}

Now let us return to $\wt X,$ with boundary fibration
\begin{equation*}
	Z \fib \pa\wt X \xlra{\phi} Y.
\end{equation*}
Let
\begin{equation*}
	c: [0,\eps)_r \times \pa\wt X \lra \wt X
\end{equation*}
be a collar neighborhood of $\pa \wt X,$ extend $\phi$ trivially to $[0,\eps)_r \times \pa\wt X,$ and assume that $g$ is a wedge metric that is rigid on $c(\wt \pa X)$ so that
$c^*g = dr^2 + r^2g_Z + \phi^*g_Y.$
We abbreviate $c([0,\eps)_r \times \pa \wt X)$ by $c(\pa\wt X).$

Near $\pa \wt X$ we have a splitting
\begin{equation*}
	{}^w T^*\wt X\rest{c(\pa\wt X)} = N^*\pa\wt X \oplus r T^*\pa\wt X/Y \oplus T^*Y
\end{equation*}
and we will make frequent use of the decomposition
\begin{multline*}
	\CI(c(\pa\wt X); \Lambda^\bullet  ({}^w T^*\wt X)) \\ = 
	\CI(c(\pa\wt X); \Lambda^\bullet (r T^*\pa\wt X/Y \oplus T^*Y) )
	\oplus dr \wedge
	\CI(c(\pa\wt X); \Lambda^\bullet (r T^*\pa\wt X/Y \oplus T^*Y) ).
\end{multline*}
With respect to this decomposition, the de Rham operator $\eth_{\dR} = d+\delta$ of $g$ takes the form
\begin{equation}\label{eq:DNearBdy}
	\eth_{\dR} = 
	\begin{pmatrix}
	\tfrac1r\eth_{\dR}^Z + \wt \eth_{\dR}^Y  + r \bR& -\pa_r +\tfrac1r(\bN -v) \\
	\pa_r + \tfrac1r \bN & - \tfrac1r\eth_{\dR}^Z - \wt \eth_{\dR}^Y - r\bR
	\end{pmatrix}
\end{equation}
where, in the notation of \S\ref{sec:CohoFiberBdles}, $\wt \eth_{\dR}^Y = \wt d^Y + \wt \delta^Y,$ $\bR = R + R^*,$ and $\bN$ is the `vertical number operator', i.e., the integer valued map with the property that
\begin{equation*}
	\bN\rest{\Lambda^i N^*\pa \wt X \wedge \Lambda^j (r T^*\pa\wt X/Y) \wedge \Lambda^k(T^*Y) } = j.
\end{equation*}
$ $

The local coordinate expression exhibits $\eth_{\dR}$ as an example of a {\em wedge differential operator} (see, e.g., \cite{ALMP11, Krainer-Mendoza:1stSurvey}).
A vector field on $\wt X$ is said to be an {\em edge vector field} if its restriction to $\pa\wt X$ is tangent to the fibers of $\phi.$
The space of edge vector fields is denoted $\cV_e$ and its enveloping algebra is the space of {\em edge differential operators} \cite{Mazzeo:Edge}.
Put another way, a differential operator of order $m$ is an edge differential operator if in every choice of local coordinates like \eqref{eq:FibCoords} it has the form
\begin{equation*}
	\sum_{j+|\alpha|+|\beta|\leq m} a_{j,\alpha,\beta}(r, y, z) (r\pa_r)^j(r\pa_y)^{\alpha}\pa_z^{\beta}
\end{equation*}
where $\alpha$ and $\beta$ are multi-indices and the coefficients $a_{j,\alpha,\beta}$ are smooth on all of $\wt X.$
A differential operator $L$ of order $m$ is a {\em wedge differential operator}  if $r^m L$ is an edge differential operator.
Thus, since the above expression shows that $r\eth_{\dR}$ can be written in terms of the vector fields $r\pa_r,$ $r\pa_y,$ $\pa_z,$ we can conclude that $\eth_{\dR}$ is first order wedge differential operator.
As such, its analysis proceeds by understanding three model operators known as the (wedge) principal symbol, the normal operator, and the indicial operator.

Just as the principal symbol of a  differential operator on a closed manifold is a section defined on the cotangent bundle, the principal symbol of $\eth_{\dR}$ as a wedge differential operator will be a section of $\End(\Lambda^\bullet  ({}^w T^*\wt X))$ over ${}^w T^*\wt X.$
We can find this symbol in one of the standard ways, given $\xi \in {}^w T^*_p \wt X$ choose a function $f \in \CI(\wt X)$ such that $df(p)=\xi$ and define
\begin{equation*}
	\sigma(\eth_{\dR})(\xi) = [\eth_{\dR}, f](p) \in \End(\Lambda^\bullet  ({}^w T^*\wt X)).
\end{equation*}
It is easy to see that
\begin{equation*}
	(\sigma(\eth_{\dR})(\xi))^2 = |\xi|^2_g \Id \in \End(\Lambda^\bullet  ({}^w T^*\wt X))
\end{equation*}
so we see that $\eth_{\dR}$ is {\em uniformly elliptic} on $\wt X$ as a wedge differential operator.

The normal operator of $\eth_{\dR}$ is a family of operators parametrized by $q \in Y.$
At each point $q\in Y$ we look at then model wedge $\bbR^+_r \times Z_q \times T_qY$ with the metric $dr^2 + r^2 g_{Z_q} + g_{T_qY}$ obtained from $g$ by freezing coefficients at $q.$ The de Rham operator of this metric is given by
\begin{equation}\label{eq:NormalOp}
	N_q(\eth_{\dR}) = 
	\begin{pmatrix}
	\tfrac1r\eth_{\dR}^{Z_q} + \eth_{\dR}^{\bbR^h} & -\pa_r +\tfrac1r(\bN -v) \\
	\pa_r + \tfrac1r \bN & - \tfrac1r\eth_{\dR}^{Z_q} -  \eth_{\dR}^{\bbR^h}
	\end{pmatrix},
\end{equation}
known as the {\em normal operator of $\eth_{\dR}$ at $q\in Y$},
and it is these operators that model the behavior of $\eth_{\dR}$ near $\pa \wt X.$
A key fact, already mentioned above in the case of conic singularities, is that {\em the link $Z_q$ only enters through its de Rham operator, $\eth_{\dR}^{Z_q}.$}

The final model operator, the indicial family, determines the leading order behavior of elements in the null space of $\eth_{\dR}.$
Assume that we have an element of the null space that in local coordinates has the form $r^{\zeta}a(y,z) + \cO(r^{\zeta+1})$ for some $\zeta \in \bbC.$
Applying $\eth_{\dR}$ yields
\begin{equation*}
	\eth_{\dR}(r^{\zeta}a(y,z) + \cO(r^{\zeta+1}))
	= r^{\zeta-1}
	\begin{pmatrix}
	\eth_{\dR}^Z & -\zeta + \bN -v \\
	\zeta + \bN & - \eth_{\dR}^Z 
	\end{pmatrix}
	a(y,z) + \cO(r^{\zeta})
\end{equation*}
and in order for this to vanish we need to have $a(q, \cdot)$ in the null space of the {\em indicial family of $\eth_{\dR}$ at $q \in Y,$ $\zeta \in \bbC$},
\begin{equation}\label{eq:IndicialOp}
	I_q(\eth_{\dR};\zeta) = 
	\begin{pmatrix}
	\eth_{\dR}^Z & -\zeta + \bN -v \\
	\zeta + \bN & - \eth_{\dR}^Z 
	\end{pmatrix}.
\end{equation}
For each $q$ the set of $\zeta$ for which $I_q(\eth_{\dR};\zeta)$ has a non-trivial null space are called the {\em indicial roots} or {\em boundary spectrum} of $\eth_{\dR}$ and denoted
\begin{equation*}
	\spec_b(\eth_{\dR};q).
\end{equation*}
Ellipticity implies that, for a fixed $q,$ this is a discrete subset of $\bbC.$

\subsection{The maximal and minimal domain} \label{sec:MaxMinDom}

The model operators described above can be used to understand the behavior of $\eth_{\dR}$ on various different function spaces. For our purposes the most important will be (a) the $L^2$ space of differential forms determined by the Riemannian metric and (b) the edge Sobolev spaces. The latter are defined by
\begin{multline*}
	H_e^m(\wt X; \Lambda^\bullet  ({}^w T^*\wt X)) = \{ \omega \in L^2(\wt X; \Lambda^\bullet  ({}^w T^*\wt X)): \\
	P\omega \in L^2(\wt X; \Lambda^\bullet  ({}^w T^*\wt X)) 
	\text{ for every edge differential operator $P$ of order } m \}.
\end{multline*}
In fact an elliptic edge differential operator has a unique (closed) domain as an operator on $L^2,$ namely the appropriate edge Sobolev space.
As we will now explore, this is usually not true for wedge differential operators such as $\eth_{\dR}.$

The de Rham operator unambiguously defines an operator
\begin{equation*}
	\eth_{\dR}: \CIc(\wt X; \Lambda^\bullet  ({}^w T^*\wt X)) \lra \CIc(\wt X; \Lambda^\bullet  ({}^w T^*\wt X)),
\end{equation*}
but if we wish to study $\eth_{\dR}$ as an operator on $L^2$-differential forms we need to be careful about its domain.
Recall that an unbounded operator is said to be {\em closed} if its graph is a closed set, and that a well-behaved operator (e.g., Fredholm or even non-empty resolvent set) is necessarily closed.
There are two canonical closed extensions of $\eth_{\dR}$ from smooth compactly supported forms.
First there is the {\em minimal domain},
\begin{multline*}
	\cD_{\min}(\eth_{\dR}) = \{ \omega \in L^2(\wt X; \Lambda^\bullet  ({}^w T^*\wt X)) : \\
	\exists (\omega_n) \subseteq \CIc(\wt X; \Lambda^\bullet  ({}^w T^*\wt X)) \Mst
	\omega_n \to \omega \Min L^2 \Mand \eth_{\dR}\omega_n \text{ is $L^2$-Cauchy} \},
\end{multline*}
on which we define $\eth_{\dR}\omega = \lim \eth_{\dR}\omega_n.$ This domain is characterized by the fact that the graph of $(\eth_{\dR}, \cD_{\min}(\eth_{\dR}))$ is the closure of the graph of $(\eth_{\dR}, \CIc(\wt X; \Lambda^\bullet  ({}^w T^*\wt X))).$
Secondly, there is the {\em maximal domain},
\begin{equation*}
	\cD_{\max}(\eth_{\dR}) = \{ \omega \in L^2(\wt X; \Lambda^\bullet  ({}^w T^*\wt X)) : \\
	\eth_{\dR}\omega \in L^2(\wt X; \Lambda^\bullet  ({}^w T^*\wt X)) \}
\end{equation*}
where $\eth_{\dR}\omega$ is computed distributionally. 
Clearly, if $\cD \subseteq L^2(\wt X; \Lambda^\bullet  ({}^w T^*\wt X))$ is a closed domain for $\eth_{\dR}$ containing $\CIc(\wt X; \Lambda^\bullet  ({}^w T^*\wt X))$ then it must satisfy
\begin{equation*}
	\cD_{\min}(\eth_{\dR}) \subseteq \cD \subseteq \cD_{\max}(\eth_{\dR})
\end{equation*}
justifying the labels `minimal' and `maximal'.

These domain are related to the edge Sobolev space of order one by the fact that $r\eth_{\dR}$ is an elliptic edge operator.
In fact this yields that
\begin{equation*}
	rH^1_e(\wt X; \Lambda^\bullet  ({}^w T^*\wt X)) \subseteq \cD_{\min}(\eth_{\dR})
	\subseteq \cD_{\max}(\eth_{\dR}) \subseteq H^1_e(\wt X; \Lambda^\bullet  ({}^w T^*\wt X)).
\end{equation*}

It turns out that the key to distinguishing between closed extensions of $\eth_{\dR}$ lies in the indicial operator.
To see intuitively why this happens, suppose we have a differential form supported near the boundary of the form $r^{\zeta}a(y,z).$
The volume form of a wedge metric in local coordinates is essentially of the form $r^v \; drdydz$ so that
\begin{equation*}
	r^{\zeta}a(y,z) \in L^2 \iff \Re\zeta +\tfrac12(v+1)>0.
\end{equation*}
Now if we also want to have 
\begin{equation*}
	\eth_{\dR}(r^{\zeta}a(y,z)) = r^{\zeta-1} I_y(\eth_{\dR};\zeta) a(y,z) + \cO(r^{\zeta})
\end{equation*}
in $L^2$ we will need to have
\begin{equation*}
	{ \Re \zeta + \tfrac12(v+1) > 1 } \quad
	\Mor \quad
	{
	\Re \zeta + \tfrac12(v+1) > 0 \Mand
	I_y(\eth_{\dR};\zeta)a(y,z) = 0 }.
\end{equation*}
This suggests that a key r\^ole will be played by
\begin{equation*}
	\spec_b(\eth_{\dR};q) \cap \{ \zeta: 0 < \Re \zeta + \tfrac12(v+1) \leq 1 \},
\end{equation*}
and we now make an extra assumption on the metric in order to make this set as simple as possible.
Directly from \eqref{eq:IndicialOp} it is easy to see that the set $\spec_b(\eth_{\dR};q)$ can be determined from the spectrum of $\eth_{\dR}^{Z_q},$ the de Rham operator on $Z_q$ with respect to the metric $g_{Z_q}.$ By rescaling this metric, we can assume that there is an interval around zero with no non-zero eigenvalues of $\eth_{\dR}^{Z_q}$ for any $q \in Y.$
We will say that a wedge metric is {\em suitably scaled} if
\begin{equation*}
	\spec_b(\eth_{\dR};q) \cap \{ \zeta: 0 \leq \Re \zeta + \tfrac12(v+1) \leq 1 \} \subseteq 
	\begin{cases}
	\{ -\tfrac12(v) \} & \Mif v \text{ is even} \\
	\{ -\tfrac12(v\pm 1) \} & \Mif v \text{ is odd}
	\end{cases}
\end{equation*}
These are the indicial roots that come from the null space of $\eth_{\dR}^Z$ and can not be removed by scaling the metric.
Note that the null space of the indicial operator at these roots is cohomological, e.g., 
\begin{equation*}
	\ker( I_q(\eth_{\dR};-\tfrac v2) ) = \cH^{v/2}(Z_q) \oplus \cH^{v/2}(Z_q) = \{ \text{harmonic forms on $Z_q$ of degree $v/2$} \}^2,
\end{equation*}
and so these null spaces fit together to form a bundle over $Y,$ isomorphic to two copies of the flat bundle described in \S\ref{sec:CohoFiberBdles}.

The intuitive analysis about the leading terms of elements in the maximal domain from the last paragraph can be made precise, but it requires distributional asymptotic expansions \cite[Lemma 2.1]{ALMP13}, cf. \cite[Theorem 7.3]{Mazzeo:Edge}, \cite[\S5]{Krainer-Mendoza:BVP1st}, \cite{Melrose-Mendoza}:
\begin{lemma}\label{lem:AsympExp}
If $g$ is a suitably scaled wedge metric on $\wt X,$ then any $\omega \in \cD_{\max}(\eth_{\dR})$ has a partial distributional asymptotic expansion
\begin{multline*}
	\omega \sim r^{-v/2}(\alpha(\omega) + dr \wedge \beta(\omega)) + \wt \omega \\
	\Mwhere
	\alpha, \beta \in H^{-1/2}(Y; \Lambda^\bullet T^*Y \otimes \cH^{v/2}(\pa\wt X/Y)), 
	\quad \wt \omega \in \bigcap_{\eps>0} r^{1-\eps}H^{-1}_e(\wt X; \Lambda^\bullet  ({}^w T^*\wt X))
\end{multline*}
in the sense that the $L^2$-pairing of $\omega$ with a sufficiently regular form is equal to the distributional pairing with this expansion.
If $v$ is odd or $\tH_{\dR}^{v/2}(Z) =0$ (i.e., if $\hat X$ is Witt) then this expansion is just $\omega \sim \wt\omega$ and our convention is that $\alpha(\omega) = \beta(\omega)=0.$
\end{lemma}

As anticipated, we can use this lemma to distinguish between the domains of closed extensions of $\eth_{\dR}$ \cite[Proposition 2.2]{ALMP13}.
\begin{proposition}
If $g$ is a suitably scaled wedge metric on $\wt X$ and $\omega \in \cD_{\max}(\eth_{\dR})$ then
\begin{equation*}
	\omega \in \cD_{\min}(\eth_{\dR}) \iff \alpha(\omega) = \beta(\omega) =0.
\end{equation*}
\end{proposition}

In particular, if $\hat X$ is a Witt space then the minimal and maximal domains of $\eth_{\dR}$ coincide so we say that $\eth_{\dR}$ is {\em essentially closed} (in fact {\em essentially self-adjoint} since this closed extension will necessarily be self-adjoint).
On a non-Witt space, this proposition shows that domains between the minimal and the maximal are determined by the coefficients $\alpha$ and $\beta.$ 
We can think of these as being Cauchy data on which we can impose ideal boundary conditions.

\subsection{Cheeger ideal boundary conditions} \label{sec:CIBC}

Given an element $\omega \in \cD_{\max}(\eth_{\dR})$ of the maximal domain,  its `Cauchy data' consists of two distributional sections $\alpha(\omega),$ $\beta(\omega)$ of the vector bundle
\begin{equation*}
	\cH^{v/2}(\pa \wt X/Y) \lra Y.
\end{equation*}
This vector bundle has a canonical flat structure reviewed in \S\ref{sec:CohoFiberBdles}. We say that its subbundle $W,$
\begin{equation*}
	\xymatrix{ W \ar[rr] \ar[rd] & & \cH^{v/2}(\pa \wt X/Y) \ar[ld] \\ & Y & }
\end{equation*}
is a {\em mezzoperversity} if it is parallel with respect to the flat connection of $\cH^{v/2}(\pa\wt X/Y).$ 
Note that the bundle $\cH^{v/2}(\pa \wt X/Y)$ inherits a bundle metric (from the $L^2$ metric on harmonic forms); we denote the orthogonal complement of $W$ by $W^{\perp}.$
To each mezzoperversity we associate {\em Cheeger ideal boundary conditions} to define the domain
\begin{multline*}
	\cD_{W}(\eth_{\dR}) = \{ \omega \in \cD_{\max}(\eth_{\dR}) : \\
	\alpha(\omega) \in H^{-1/2}(Y; \Lambda^\bullet T^*Y \otimes W), \; \beta(\omega) \in H^{-1/2}(Y;\Lambda^\bullet T^*Y \otimes W^{\perp}) \}.
\end{multline*}
Note that mezzoperversities always exist: two canonical choices are $W = \{0 \}$ and $W = \cH^{v/2}(\pa\wt X/Y).$

To avoid having to continually distinguish between the Witt and non-Witt cases, let us adopt the convention of writing the unique closed domain of $d+\delta$ on a Witt space by
\begin{equation*}
	\cD_{W}(\eth_{\dR})
\end{equation*}
with $W = \emptyset,$ and refer to this as the empty mezzoperversity.

\begin{theorem}\label{thm:CmptResolv}
Let $\hat X$ be a stratified space of depth one endowed with a suitably scaled wedge metric $g$
and let $\cD_W(\eth_{\dR}) \subseteq \cD_{\max}(\eth_{\dR})$ be the domain corresponding to a (possibly empty) mezzoperversity.
The operator $(\eth_{\dR}, \cD_W(\eth_{\dR}))$ is closed, self-adjoint, and Fredholm with compact resolvent, and
\begin{equation*}
	\cD_W(\eth_{\dR}) \subseteq \bigcap_{\eps>0}r^{1/2-\eps}H^1_e(\wt X; \Lambda^\bullet ({}^w T^*\wt X)).
\end{equation*}
\end{theorem}

This theorem is proved in \cite[\S4]{ALMP13}. The fact that the domain is closed follows easily from Lemma \ref{lem:AsympExp} and continuity of the Cauchy data map.
The proof that the operator $(\eth_{\dR}, \cD_W)$ is Fredholm and has compact resolvent uses the normal operators $N_q(\eth_{\dR})$ defined in \S\ref{sec:ModelOps}.
The domain $\cD_W$ determines a domain for $N_q(\eth_{\dR})$ as an unbounded operator on the appropriate $L^2$ space.
With this domain $N_q(\eth_{\dR})$ is invertible and we can put these inverses together to construct a parametrix for $(\eth_{\dR}, \cD_W).$
This parametrix shows that $(\eth_{\dR}, \cD_W)$ is Fredholm and also that $\cD_W$ is compactly contained in $L^2(\wt X; \Lambda^\bullet ({}^w T^*\wt X))$ which implies that the resolvent will be compact.
Finally the proof of that this operator is self-adjoint in the non-Witt case follows from facts about domains of $d$ that we will discuss in the next subsection (Theorem \ref{thm:HodgeThmSimpleEdge}(iii)).

An important consequence of the theorem is that the null space of $(\eth_{\dR}, \cD_W(\eth_{\dR}))$ is finite dimensional. We refer to this null space as the {\em Hodge cohomology with Cheeger ideal boundary conditions} and we denote it by
\begin{equation*}
	\ker (\eth_{\dR}, \cD_W(\eth_{\dR})) = \cH_{W}^*(\hat X).
\end{equation*}

With this theorem we have established the most important analytic properties of $\eth_{\dR}$ and Hodge cohomology. 
We next look for a corresponding de Rham theory.

\subsection{De Rham cohomology} \label{sec:DepthOneDR}

The operators $d$ and $\delta$ individually define unbounded operators on $L^2(\wt X; \Lambda^\bullet  ({}^w T^*\wt X))$ and have minimal and maximal domains,
\begin{equation*}
	\cD_{\min}(d) \subseteq \cD_{\max}(d), \quad \cD_{\min}(\delta) \subseteq \cD_{\max}(\delta).
\end{equation*}
Since $d$ and $\delta$ are formally adjoint, these domains are related by
\begin{equation*}
	(d, \cD_{\min}(d))^* = (\delta, \cD_{\max}(\delta)), \quad
	(d, \cD_{\max}(d))^* = (\delta, \cD_{\min}(\delta))
\end{equation*}
as follows directly from the definition of adjoint and the distributional action of a differential operator.

Given a domain $\cD$ for $d,$ let us denote
\begin{equation*}
	\cD_p = \cD \cap L^2(\wt X; \Lambda^p ({}^w T^*\wt X)).
\end{equation*}
It is easy to see that
\begin{equation}\label{eq:DmaxCmpx}
	0 \lra
	\cD_{\max}(d)_0 
	\xlra{d}
	\cD_{\max}(d)_1 
	\xlra{d}
	\ldots
	\xlra{d}
	\cD_{\max}(d)_n
	\lra 0
\end{equation}
forms a complex, indeed a {\em Hilbert complex} in the sense of \cite{Bruning-Lesch:Hilbert}, and similarly for $\cD_{\min}(d).$

If $\hat X$ is a Witt space, then $\eth_{\dR}$ is essentially self-adjoint and \cite[Lemma 3.8]{Bruning-Lesch:Hilbert} shows that
\begin{equation*}
	\cD_{\max}(d) = \cD_{\min}(d).
\end{equation*}
If we denote the cohomology of the complex \eqref{eq:DmaxCmpx} by $\tH_{\dR}^*(\hat X)$ then Theorem \ref{thm:CmptResolv} and \cite[Corollary 2.5]{Bruning-Lesch:Hilbert} show that
\begin{equation*}
	\tH_{\dR}^p(\hat X) \cong \cH^p(\hat X).
\end{equation*}

Consider next the situation when $\hat X$ is not a Witt space.
Since $d$ is not elliptic, we can not expect arbitrary elements in the maximal domain of $d$ to have a distributional asymptotic expansion at $\wt X.$
However there is a way around this using the Kodaira decomposition (see, e.g., \cite[Lemma 2.1]{Bruning-Lesch:Hilbert}) corresponding to $(d, \cD_{\min}(d)),$
\begin{equation*}
	L^2(\wt X; \Lambda^\bullet  ({}^w T^*\wt X))
	= \bar{ d( \cD_{\min}(d)) } \oplus \bar{ \delta(\cD_{\max}(\delta)) } \oplus
	\ker (d, \cD_{\min}(d)) \cap \ker(\delta, \cD_{\max}(\delta)).
\end{equation*}
Given $\omega \in \cD_{\max}(d),$ let $\omega_{\delta}$ be the projection onto $\bar{ \delta(\cD_{\max}(\delta)) },$ and notice that
\begin{equation*}
	d\omega_{\delta} = d\omega, \quad \delta \omega_{\delta} = 0,
\end{equation*}
so $(d+\delta)\omega_{\delta} \in L^2(\wt X; \Lambda^\bullet  ({}^w T^*\wt X))$ and hence
\begin{equation*}
	\omega_\delta \in \cD_{\max}(\eth_{\dR}).
\end{equation*}
In particular Lemma \ref{lem:AsympExp} determines Cauchy data $\alpha(\omega_{\delta})$ and $\beta(\omega_{\delta}),$ and given a mezzoperversity $W$ we define
\begin{equation*}
	\cD_W(d) = \{ \omega \in \cD_{\max}(d) : \alpha(\omega_\delta) \in H^{-1/2}(Y; \Lambda^\bullet T^*Y \otimes W) \}.
\end{equation*}
It turns out that \cite[Lemma 5.3]{ALMP13}
\begin{equation*}
	\alpha( (d\omega)_{\delta} ) = \nabla^{\cH} \alpha( \omega_{\delta} )
\end{equation*}
where $\nabla^{\cH}$ is the flat connection on $\cH^{v/2}(\pa \wt X/Y),$ and hence the fact that $W$ is a flat bundle implies  
\begin{equation*}
	d(\cD_{W(d)}) \subseteq \cD_{W}(d).
\end{equation*}
Thus each mezzoperversity yields a complex
\begin{equation*}
	0 \lra
	\cD_{W}(d)_0 
	\xlra{d}
	\cD_{W}(d)_1 
	\xlra{d}
	\ldots
	\xlra{d}
	\cD_{W}(d)_n
	\lra 0
\end{equation*}
whose cohomology we denote
\begin{equation*}
	\tH_{W, \dR}^*(\hat X).
\end{equation*}
This is the de Rham cohomology associated to a mezzoperversity, and we next establish a Hodge theorem showing that it is equal to the Hodge cohomology.

\begin{theorem}\label{thm:HodgeThmSimpleEdge}
$ $
\begin{itemize}
\item [i)] $\cD_W(d)$ is the closure of $\cD_{W}(\eth_{\dR})$ with respect to the graph norm of $d.$
\item [ii)] If $\omega \in \cD_{\max}(\delta)$ and $\omega_d$ is its projection onto $\bar{d(\cD_{\max}(d))}$ then $\omega_d$ is an element of $\cD_{\max}(\eth_{\dR})$ and hence has Cauchy data. The adjoint domain of $\cD_W(d)$ is
\begin{equation*}
	\cD_W(\delta) = \{ \omega \in \cD_{\max}(\delta) : \beta(\omega_d) \in H^{-1/2}(Y; \Lambda^\bullet T^*Y \otimes W^{\perp}).
\end{equation*}
\item [iii)] $\cD_{W}(\eth_{\dR}) = \cD_W(d) \cap \cD_{W}(\delta)$ and hence
\begin{equation*}
	\tH^\bullet _{W, \dR}(\hat X) \cong \cH^\bullet _W(\hat X).
\end{equation*}
\end{itemize}
\end{theorem}

\subsection{Generalized Poincar\'e duality} \label{sec:DepthOnePD}

Recall that on any smooth oriented closed manifold there is a natural intersection pairing on differential forms
\begin{equation*}
	(\eta, \omega) \mapsto \int \eta \wedge \omega
\end{equation*}
that descends to a non-degenerate pairing on cohomology. The signature of this pairing is the signature of the manifold.

Assume that $\hat X$ is an oriented pseudomanifold of depth one, and let
\begin{equation*}
	\xymatrix @R=1pt 
	{L^2(\wt X; \Lambda^\bullet ({}^w T^*\wt X)) \times L^2(\wt X; \Lambda^\bullet ({}^w T^*\wt X)) \ar[r]^-B & \bbR \\
	(\eta, \omega) \ar@{|->}[r] & \displaystyle \int_{\wt X} \eta\wedge\omega }.
\end{equation*}
Note that the Hodge star of a wedge metric is an isometric involution of $L^2(\wt X; \Lambda^\bullet ({}^w T^*\wt X))$ and hence defines an involution
\begin{equation*}
	*: \cD_{\max}(\eth_{\dR})\lra \cD_{\max}(\eth_{\dR}).
\end{equation*}
We have $B(\eta, *\eta) = \norm{\eta}_{L^2}^2,$ and hence $B$ is non-degenerate on $\cD_{\max}(\eth_{\dR}).$

For an oriented Witt space, $B$ descends to a non-degenerate pairing on cohomology
\begin{equation*}
	B: \tH^\bullet _{\dR}(\hat X) \times \tH^\bullet _{\dR}(\hat X) \lra \bbR,
\end{equation*}
and the signature of this form is the {\em signature of the Witt space $\hat X.$}

If $\hat X$ is oriented but not a Witt space, then $*$ will generally not preserve $\cD_W(\eth_{\dR})$ and we need to bring in a different mezzoperversity. 
Define
\begin{equation*}
	Q: \CI(Y; \cH^{v/2}(\pa\wt X/Y)) \times \CI(Y; \cH^{v/2}(\pa\wt X/Y)) \lra \CI(Y), \quad
	Q(\eta, \omega) = \int_{\pa\wt X/Y} \eta \wedge\omega.
\end{equation*}
Directly from the definition of the flat connection $\nabla^{\cH}$ we see that
\begin{equation*}
	dQ(\eta, \omega) = Q(\nabla^{\cH}\eta, \omega) \pm Q(\eta, \nabla^{\cH}\omega).
\end{equation*}
This shows that if $W \lra Y$ is a flat subbundle of $\cH^{v/2}(\pa\wt X/Y)$ and $\bbD W \lra Y$ is the $Q$-orthogonal complement of $W,$
then $\bbD W$ is also a flat subbundle. We call $\bbD W$ the {\em dual mezzoperversity} to $W.$
We can rewrite $Q$ using the Hodge star (on $Z$) and see that
\begin{equation*}
	\bbD W = *W^{\perp}.
\end{equation*}
In turn this makes it easy to see that the Hodge star defines a map 
\begin{equation*}
	*: \cD_{W}(\eth_{\dR}) \lra \cD_{\bbD W}(\eth_{\dR})
\end{equation*}
and that $B$ descends to non-degenerate pairing on cohomology,
\begin{equation*}
	B: \tH^\bullet _{W, \dR}(\hat X) \times \tH^\bullet _{\bbD W, \dR}(\hat X) \lra \bbR,
\end{equation*}
which is a generalization of Poincar\'e duality to this setting.

If $W = \bbD W$ we say that $W$ is a {\em self-dual mezzoperversity}. 
There are topological obstructions to the existence of a self-dual mezzoperversity, for example the existence of a self-dual mezzoperversity implies that the signature of the link $Z$ vanishes. 
If a space admits a self-dual mezzoperversity (or is a Witt space) then we say that it is a {\em Cheeger space}.
If $W$ is a self-dual mezzoperversity then $B$ descends to a non degenerate form on $\tH^\bullet _{W, \dR}(\hat X)$ and so these groups satisfy Poincar\'e duality and define a signature.

In \S\ref{sec:PropSign} we will explain why this signature is independent of the choice of self-dual mezzoperversity and hence only depends on the Cheeger space.
Also in this section we describe this signature as the Fredholm index of the signature operator with an appropriate domain.

\section{General stratified spaces} \label{sec:DepthEll}

Now we consider the general case where $\hat X$ is a union of several smooth manifolds,
\begin{equation*}
	\hat X = Y_0 \cup Y_1 \cup \ldots \cup Y_\ell.
\end{equation*}
For simplicity let us {\em assume that}, for each $k,$ the closure $\bar Y_k$ of $Y_k$ in $\hat X$ is $\hat Y_k = Y_k \cup Y_{k+1} \cup \ldots \cup Y_\ell.$
It is always true that the closure of a stratum is a union of strata, so our assumption is that there is a single chain of strata.
The points of depth zero, $Y_0,$ form $\hat X^{\reg},$ the regular part of $\hat X,$ while the other strata form the singular part of $\hat X.$
We denote the link of $\hat X$ at points in $Y_k$ by $\hat Z_k,$ these are themselves stratified spaces (of depth $k-1$).
We denote the dimensions of these spaces by
\begin{equation*}
	n = \dim \hat X^{\reg}, \quad h_k =\dim Y_k, \quad v_k = \dim \hat Z_k^{\reg}.
\end{equation*}
Note that $h_k$ is a decreasing function of $k.$
We assume throughout that all strata and links are connected.
(Neither of these assumptions is significant and they mainly serve to simplify the notation.)

In Section \ref{sec:StratSpaces} we will recall the precise definition of a stratified space and, more importantly for our purposes, explain how to resolve $\hat X$ to a smooth manifold with corners $\wt X.$ For now we content ourselves with describing the structure of the resolved space where our analysis is carried out.
The resolved structure is called a {\em boundary fibration structure}. The boundary hypersurfaces of $\wt X,$ which we denote $M_1, \ldots, M_{\ell},$ are embedded and in one to one correspondence with the singular strata. Indeed, each $M_k$ is the total space of a fiber bundle,
\begin{equation*}
	\wt Z_k \fib M_k \xlra{\phi_k} \wt Y_k,
\end{equation*}
with base $\wt Y_k$ the resolution of $\bar Y_k,$ and typical fiber $\wt Z_k,$ the resolution of $\hat Z_k.$
If $M_j \cap M_k \neq \emptyset,$ and $j<k,$ 
the fiber bundles are compatible in that there is a boundary hypersurface $K_{jk}$ of $Y_j$ and a fiber bundle map $\phi_{jk}: K_{jk}\lra Y_j$ such that the diagram
\begin{equation}\label{eq:CornerFibration}
	\xymatrix{
	M_j \cap M_k \ar[dr]_{\phi_k} \ar[rr]^{\phi_j} & & K_{jk} \ar[ld]^{\phi_{jk}} \subseteq Y_j \\
	& Y_k & }
\end{equation}
commutes.

A rigid wedge metric $g$ on $\hat X^{\reg} = \wt X^{\circ}$ is a metric that in some collar neighborhood of $M_k$ (for each $k$) has the form
\begin{equation*}
	dr_k^2 + r_k^2 g_{M_k/Y_k} + \phi_k^*g_{Y_k} 
\end{equation*}
where $r_k$ is a smooth function vanishing linearly on $M_k$ and positive elsewhere,
$g_{Y_k}$ is a rigid wedge metric on $\wt Y_k$ and $g_{M_k/Y_k}$ is a family of rigid wedge metrics on the fibers $\wt Z_k.$
The covectors of bounded pointwise length with respect to $g$ are the sections of a vector bundle, the {\em (iterated) wedge cotangent bundle,}
\begin{equation*}
	{}^{w} T^*\wt X \lra \wt X,
\end{equation*}
that is canonically isomorphic to $T^*\wt X$ over $\wt X^\circ.$
A wedge metric extends to a non-degenerate bundle metric on the iterated wedge cotangent bundle.
We denote the de Rham operator of a rigid wedge metric by $\eth_{\dR} = d + \delta.$ Its minimal and maximal domains $\cD_{\min}(\eth_{\dR}),$ $\cD_{\max}(\eth_{\dR})$ are defined just as in the case of a single singular stratum. 

Our plan will be to study $\eth_{\dR}$ one stratum at a time.
For this purpose,
a useful class of functions on $\wt X$ is given by
\begin{equation*}
	\CI_{\Phi}(\wt X) = \{ f \in \CI(\wt X): f\rest{M_k} \in \phi_k^* \CI(\wt Y_k) \}.
\end{equation*}
Note that any such function descends to a continuous function on $\hat X,$ so we can think of $\CI_{\Phi}(\wt X)$ as providing a notion of smooth function on $\hat X.$
If $f \in \CI_{\Phi}(\wt X)$ then $df$ extends from $\wt X^{\circ}$ to a smooth section of ${}^w T^*\wt X$ and this can be used to show
\begin{equation*}
	\omega \in \cD_{\max}(\eth_{\dR}), f\in \CI_{\Phi}(\wt X) \implies f\omega \in \cD_{\max}(\eth_{\dR}).
\end{equation*}
It is easy to see that there is a partition of unity of $\wt X$ consisting entirely of functions in $\CI_{\Phi}(\wt X)$ (\cite[Lemma 5.2]{ABLMP}).
This allows us to focus on elements of $\cD_{\max}(\eth_{\dR})$ supported on a {\em distinguished neighborhood} of a point in a singular stratum $q \in Y_k = \wt Y_k^{\circ},$
by which we mean a subset $\cU_q$ of a collar neighborhood $[0,1)_x \times M_k$ on which the fiber bundle is trivial and diffeomorphic to 
$[0,1)_x \times \wt Z_k \times \bbB^h$ where $\bbB^h$ is a ball in $Y_k$ centered around $q.$
Let
\begin{multline}
	A(Y_k) = \{ \chi \in \CI_{\Phi}(\wt X) : \\
	\chi \text{ is supported in a distinguished neighborhood of a point in } Y_k \}.
\end{multline}
%

\subsection{Cheeger ideal boundary conditions}

The stratum $Y_1$ has depth one, meaning that the link at each point is a smooth manifold.
If $q \in Y_1$ and $\cU_q \subseteq \wt X$ is a distinguished neighborhood, then any element $u \in \cD_{\max}(\eth_{\dR})$ supported in $\cU_q$ is subject to the analysis carried out in \S\ref{sec:DepthOne}. Thus if $g$ is suitably scaled then $u$ has a partial distributional asymptotic expansion as in Lemma \ref{lem:AsympExp}.
If $Y_1$ is not a Witt stratum then we choose $W_1,$ a flat subbundle of $\cH^{v/2}(M_1/Y_1),$ and we define
\begin{multline*}
	\cD_{\max, W_1}(\eth_{\dR}) = \{ u \in \cD_{\max}(\eth_{\dR}): \Mforall \chi \in A(Y_1), \\
	\alpha(\chi u) \in H^{-1/2}(Y_1; \Lambda^\bullet ({}^wT^*Y_1)\otimes W_1), \;
	\beta(\chi u) \in H^{-1/2}(Y_1; \Lambda^\bullet ({}^wT^*Y_1)\otimes W_1^{\perp}) \}.
\end{multline*}
If $Y_1$ is a Witt stratum, then there is no need to choose boundary conditions at $Y_1.$
For notational consistency, when $Y_1$ is a Witt stratum we set $W_1 = \emptyset$ and define $\cD_{\max, W_1}(\eth_{\dR}) =\cD_{\max}(\eth_{\dR}).$

Having chosen $W_1,$ we can apply Theorem \ref{thm:CmptResolv} (or rather its proof) and see that
\begin{equation*}
	u \in \cD_{\max, W_1}(\eth_{\dR}) \implies  
	\chi u \in r_1^{a}H^1_e(\wt X; \Lambda^\bullet ({}^wT^*\wt X)) \Mforall a<1/2, \chi \in A(Y_1).
\end{equation*}
Note that $\cD_{\max, W_1}(\eth_{\dR})$ is closed under multiplication by functions in $\CI_{\Phi}(\wt X),$ so we may continue to localize sections to consider one stratum at a time.\\

Next consider $Y_2.$ This is a stratum of depth two, meaning that its link $\hat Z_2$ is a stratified space of depth one. We can see, e.g., from the compatibility condition \eqref{eq:CornerFibration}, that the link of a singular point in $\hat Z_2$ must be $Z_1.$ Thus in our choice of boundary conditions at $Y_1,$ we have determined a domain for the de Rham operator on $\hat Z_2.$ Now recall the fundamental observation from \eqref{eq:NormalOp} that the link only enters into the normal operator through its de Rham operator. This allows us to carry out the analysis from \S\ref{sec:DepthOne} keeping in mind that $\eth_{\dR}^{Z_2}$ comes with its domain $\cD_{W_1}(\eth_{\dR}^{Z_2}).$
Thus (the analogue of) Lemma \ref{lem:AsympExp} applies to show that
\begin{multline*}
	u \in \cD_{\max, W_1}(\eth_{\dR}), \quad \chi \in A(Y_2) \\ \implies
	\chi u \sim r_2^{-v_2/2}(\alpha_2(\chi u) + dr_2 \wedge \beta_2(\chi u)) + \chi \wt u, \Mwhere \\
	\alpha_2(\chi u), \beta_2(\chi u) \in H^{-1/2}(Y_2; \Lambda^\bullet ({}^wT^*Y_2) \otimes \cH^{v_2/2}_{W_1}(M_2/Y_2)), \quad
	\wt u \in \bigcap_{\eps>0} r_2^{1-\eps}H^{-1}_e(\wt X; \Lambda^\bullet  ({}^w T^*\wt X))
\end{multline*}
and $\cH^{v_2/2}_{W_1}(M_2/Y_2)$ denotes the bundle over $Y_2$ whose typical fiber is $\ker(\eth_{\dR}^{Z_2}, \cD_{W_1}(\eth_{\dR}^{Z_2})),$ \cite[Lemma 2.1]{ALMP13}.
It is not obvious that these null spaces fit together to form a bundle but this follows from Corollary \ref{cor:DiffeoInv} below, and this bundle comes with a natural flat connection.
If $v_2$ is odd or this bundle has rank zero, let us set $W_2=\emptyset.$ Otherwise, let $W_2$ be a flat subbundle of $\cH^{v_2/2}_{W_1}(M_2/Y_2)$ and set
\begin{multline*}
	\cD_{\max, (W_1, W_2)}(\eth_{\dR}) = \{ u \in \cD_{\max,W_1}(\eth_{\dR}): \Mforall \chi \in A(Y_2), \\
	\alpha_2(\chi u) \in H^{-1/2}(Y_2; \Lambda^\bullet ({}^wT^*Y_2)\otimes W_2), \;
	\beta_2(\chi u) \in H^{-1/2}(Y_2; \Lambda^\bullet ({}^wT^*Y_2)\otimes W_2^{\perp}) \}.
\end{multline*}
Another application of Theorem \ref{thm:CmptResolv} shows that
\begin{equation*}
	u \in \cD_{\max, (W_1, W_2)}(\eth_{\dR})
	\implies
	\chi u \in r_2^{a}H^1_e(\wt X; \Lambda^\bullet ({}^wT^*\wt X)) \Mforall a<1/2, \chi \in A(Y_2).
\end{equation*}

We proceed in this way inductively stratum by stratum. At each step, imposing boundary conditions at the first $k-1$ strata allows us to establish a distributional asymptotic expansion at the $k^{\text{th}}$ stratum. Using this expansion we can impose boundary conditions at the $k^{\text{th}}$ stratum, invert the normal operators, construct a local parametrix, and show that elements satisfying the boundary conditions vanish to order almost $1/2$ at the first $k$-strata. Once we reach the $\ell^{\text{th}}$ stratum we can conclude with the analogue of Theorem \ref{thm:CmptResolv}, \cite[Theorem 4.3]{ALMP13}.

\begin{theorem}
Let $\hat X$ be a stratified space endowed with a suitably scaled rigid iterated wedge metric. Let $\bW = (W_1, W_2, \ldots, W_{\ell})$ be Cheeger ideal boundary conditions at the singular strata of $\hat X,$ and let
\begin{equation*}
	\cD_{\bW}(\eth_{\dR})
\end{equation*}
be the corresponding domain of the de Rham operator. Then $(\eth_{\dR}, \cD_{\bW}(\eth_{\dR}))$ is a closed, self-adjoint, Fredholm operator with compact resolvent and 
\begin{equation*}
	\cD_{\bW}(\eth_{\dR}) \subseteq \bigcap_{\eps>0} (r_1\cdots r_{\ell})^{1/2-\eps} H^1_e(\wt X; \Lambda^\bullet ({}^w T^*\wt X)).
\end{equation*}
\end{theorem}

We denote the resulting Hodge cohomology groups by 
\begin{equation*}
	\cH^\bullet _{\bW}(\hat X).
\end{equation*}
%

\subsection{De Rham cohomology}

Given Cheeger ideal boundary conditions $\bW,$ we can define domains for $d$ and $\delta$ by
\begin{equation*}
	\cD_{\bW}(d) = d\text{-graph closure of }\cD_{\bW}(\eth_{\dR}), \quad
	\cD_{\bW}(\delta) = \delta\text{-graph closure of }\cD_{\bW}(\eth_{\dR}).
\end{equation*}
To see that these form complexes and that their intersection is $\cD_{\bW}(\eth_{\dR}),$ it is important to characterize these domains in terms of distributional asymptotic expansions. 

Let $Y$ be a stratum of $\hat X$ of depth $k,$ and assume inductively that we have carried out the analysis of Hodge and de Rham cohomology on all space of depth less than $k.$
The domain
\begin{equation*}
	\cD_{\max, (W_1, \ldots, W_{k-1})}(d)
\end{equation*}
forms a Hilbert complex and so we have a Kodaira decomposition
\begin{multline*}
	L^2(\wt X; \Lambda^\bullet ({}^w T^*\wt X)) = 
	\bar{d( \cD_{\max, (W_1, \ldots, W_{k-1})}(d) )}
	\oplus
	\bar{\delta( \cD_{\max, (W_1, \ldots, W_{k-1})}(d)^* )} \\
	\oplus
	\ker( d,  \cD_{\max, (W_1, \ldots, W_{k-1})}(d) ) \cap
	\ker(\delta,  \cD_{\max, (W_1, \ldots, W_{k-1})}(d)^* ).
\end{multline*}
Given $\omega \in \cD_{\max, (W_1, \ldots, W_{k-1})}(d),$ let $\omega_\delta$ denote the orthogonal projection onto the second summand, and note that
\begin{equation*}
	\omega_{\delta} \in \cD_{\max, (W_1, \ldots, W_{k-1})}(\eth_{\dR}).
\end{equation*}
This allows us to define
\begin{multline*}
	\cD_{\max, (W_1, \ldots, W_k)}(d) = \{ \omega \in \cD_{\max, (W_1, \ldots, W_{k-1})}(d) : \\
	\alpha_k((\chi \omega)_\delta) \in H^{-1/2}(Y_k; \Lambda^\bullet ({}^w T^*Y_k) \otimes W_k) \Mforall \chi \in A(Y_k) \}.
\end{multline*}
The fact that $W_k$ is a flat subbundle of $\cH^{v_k/2}_{(W_1, \ldots, W_{k-1})}(M_k/Y_k)$ implies  that $\cD_{\max, (W_1, \ldots, W_k)}(d)$ forms a complex.

Similar reasoning allows us to define
\begin{multline*}
	\cD_{\max, (W_1, \ldots, W_k)}(\delta) = \{ \omega \in \cD_{\max, (W_1, \ldots, W_{k-1})}(\delta) : \\
	\beta_k((\chi \omega)_d) \in H^{-1/2}(Y_k; \Lambda^\bullet ({}^w T^*Y_k) \otimes W_k^{\perp}) \Mforall \chi \in A(Y_k) \}
\end{multline*}
where $(\chi\omega)_d$ denotes the orthogonal projection of $\chi\omega$ onto $\bar{d(\cD_{\max, (W_1, \ldots, W_{k-1})}(\delta)^*)}.$
Together these expressions show that
\begin{equation*}
	\cD_{\max, (W_1, \ldots, W_k)}(\eth_{\dR}) = \cD_{\max, (W_1, \ldots, W_k)}(d) \cap \cD_{\max, (W_1, \ldots, W_k)}(\delta)
\end{equation*}
and establish the analogue of Theorem \ref{thm:HodgeThmSimpleEdge}, \cite[Theorem 5.1]{ALMP13}.

\begin{theorem}
Let $\hat X$ be a stratified space endowed with a suitably scaled rigid iterated wedge metric. For any choice of Cheeger ideal boundary conditions, $\bW = (W_1, W_2, \ldots, W_{\ell}),$ the complex
\begin{equation*}
	0 \lra 
	\cD_{\bW}(d)_0 \xlra{d} 
	\cD_{\bW}(d)_1 \xlra{d}
	\ldots \xlra{d}
	\cD_{\bW}(d)_n \lra 0
\end{equation*}
is a Fredholm Hilbert complex whose cohomology $\tH^\bullet _{\bW, \dR}(\hat X)$ is canonically isomorphic to the Hodge cohomology of $\bW,$
\begin{equation*}
	\tH^\bullet _{\bW, \dR}(\hat X) \cong \cH^\bullet _{\bW}(\hat X).
\end{equation*}
\end{theorem}

\subsection{Cheeger spaces}

If $\hat X$ is oriented then to each mezzoperversity we can associate a dual mezzoperversity.
Inductively, assume that we have chosen compatible flat bundles $W_1, \ldots, W_k$ over $Y_1, \ldots, Y_k$ respectively, and their dual bundles $\bbD W_1, \ldots, \bbD W_k.$
Then, over $Y_k,$ we let
\begin{equation}\label{eq:Qk}
\begin{gathered}
	Q_k: \CI(Y_k; \cH^{v_k/2}_{(W_1, \ldots, W_{k-1})}(M_k/Y_k)) \times  \CI(Y_k; \cH^{v_k/2}_{(\bbD W_1, \ldots, \bbD W_{k-1})}(M_k/Y_k))
	\lra \CI(Y_k), \\
	Q_k(\eta, \omega) = \int_{M_k/Y_k} \eta \wedge\omega
\end{gathered}
\end{equation}
and we define the dual bundle of a flat subbundle $W_k$ of $\cH^{v_k/2}_{(W_1, \ldots, W_k)}(M_k/Y_k)$ to be its $Q$-orthogonal complement.
As before, this is again a flat bundle and satisfies
\begin{equation*}
	*W_k^{\perp} = \bbD W_k.
\end{equation*}

In this way, given a mezzoperversity $\bW = (W_1, \ldots, W_{\ell})$ we have defined a dual mezzoperversity
\begin{equation*}
	\bbD\bW = (\bbD W_1, \ldots, \bbD W_\ell)
\end{equation*}
and shown that the Hodge start intertwines the respective domains,
\begin{equation*}
	*: \cD_{\bW}(\eth_{\dR}) \lra \cD_{\bbD \bW}(\eth_{\dR}).
\end{equation*}

We say that a stratified space is a {\em Cheeger space} if it admits a mezzoperversity that coincides with its dual mezzoperversity (by convention this includes Witt spaces).

\begin{theorem}[Generalized Poincar\'e duality]
Let $\hat X$ be an oriented pseudomanifold endowed with a rigid suitably scaled iterated wedge metric.
The pairing
\begin{equation*}
	B: L^2(\wt X; \Lambda^\bullet ({}^wT^*\wt X)) \times L^2(\wt X; \Lambda^\bullet ({}^wT^*\wt X)) \lra \bbR, \quad
	B(\eta, \omega) = \int \eta\wedge\omega
\end{equation*}
descends, for any choice of mezzoperversity $\bW,$ to a non-degenerate pairing
\begin{equation*}
	\tH^\bullet _{\bW, \dR}(\hat X) \times \tH^\bullet _{\bbD\bW, \dR}(\hat X) \lra \bbR.
\end{equation*}
\end{theorem}

Thus on a Cheeger space the groups $\tH^\bullet _{\bW, \dR}(\hat X)$ satisfy Poincar\'e duality and define a signature. We will discuss the properties of this signature in section \ref{sec:PropSign}.

\section{Properties of cohomology} \label{sec:PropsCoho}

\subsection{Metric independence of de Rham cohomology} \label{sec:MetInd}

Our definition of a mezzoperversity uses Hodge cohomology groups and hence is metric dependent.
In this section we will explain how to define mezzoperversities in a metric-independent way, and why this leads to metric independent de Rham cohomology groups.
For this section we will incorporate the metric into the notation of the domain, e.g., 
\begin{equation*}
	\cD_{\max}(d;g).
\end{equation*}

To start with, note that if $g$ and $g'$ are two iterated wedge metrics on the same pseudomanifold $\hat X$ then they both induce bundle metrics on ${}^w T^*\wt X$ over $\wt X$ and, since $\wt X$ is compact, these must be quasi-isometric. That is, there is a positive constant $C$ such that
\begin{equation*}
	C g(\cdot, \cdot) \leq g'(\cdot, \cdot) \leq \frac1C g(\cdot, \cdot).
\end{equation*}
It follows that $L^2$ norms defined by $g$ and $g'$ on $\CIc(\wt X; \Lambda^\bullet ({}^w T^*\wt X))$ are equivalent and hence
\begin{equation*}
	L^2(\wt X; \Lambda^\bullet ({}^w T^*\wt X))
\end{equation*}
as a set is independent of the choice of metric and inherits equivalent norms from $g$ and $g'.$
Then directly from the definitions of the minimal and maximal domains we see that
\begin{equation*}
	\cD_{\min}(d;g) = \cD_{\min}(d;g'), \quad \cD_{\max}(d;g)= \cD_{\max}(d;g')
\end{equation*}
so we can denote these unambiguously as $\cD_{\min}(d)$ and $\cD_{\max}(d).$
If $\hat X$ is a Witt space, then this shows that its de Rham cohomology is independent of the choice of metric.

For domains defined by Cheeger ideal boundary conditions, we proceed inductively.
So first {\em assume that $\hat X$ is a pseudomanifold of depth one} and adopt the notation of Section \ref{sec:DepthOne}.
A mezzoperversity is a flat subbundle $W$ of $\cH^{v/2}(M/Y) \lra Y$ and hence depends on the metric.
Let us emphasize this by referring to $W$ as a {\em Hodge mezzoperversity} or even a {\em $g$-Hodge mezzoperversity}.
Analogously a {\em de Rham mezzoperversity} will refer to a flat subbundle of $\tH^{v/2}_{\dR}(M/Y) \lra Y.$
Using the Hodge isomorphism on the smooth manifold $Z,$ we can associate a de Rham mezzoperversity to each Hodge mezzoperversity.
We denote the de Rham mezzoperversity corresponding to $W$ by $[W].$
The advantage is that the bundle $\tH^{v/2}_{\dR}(M/Y) \lra Y$ and, from \S\ref{sec:CohoFiberBdles} its flat connection, are defined independently of the metric.
Thus it makes sense to compare the two domains,
\begin{equation*}
	\cD_{[W]}(d;g) \Mand \cD_{[W]}(d;g').
\end{equation*}

Consider the subspace $\cD_{\max}^{\reg}(d) \subseteq \cD_{\max}(d)$ consisting of wedge differential forms with an asymptotic expansion near $Y,$
\begin{equation*}
	\chi \omega = r^{-v/2}(\alpha(\omega)+ dr\wedge \beta(\omega)) + \wt \omega, \Mforall \chi \in A(Y),
\end{equation*}
where $\wt\omega \in \cD_{\min}(d)$ and $\alpha(\omega),$ $\beta(\omega)$ are smooth differential forms defined near $\pa\wt X,$ of vertical degree $v/2,$ in the null space of $d^Z.$ 
This decomposition when it exists is not unique, since elements of $\cD_{\min}(d)$ may themselves have asymptotic expansions, but one can show that the de Rham cohomology class of $\alpha(\omega)$ is well-defined.
It turns out \cite[\S5]{ALMP13} that $\cD_{\max}^{\reg}(d)$ is dense in $\cD_{\max}(d)$ with respect to the graph norm of $d,$
and 
\begin{equation*}
	\cD_{[W]}^{\reg}(d;g) = \cD_{[W]}(d, g) \cap \cD_{\max}^{\reg}(d)
\end{equation*}
is dense in $\cD_{[W]}(d;g)$ with respect to the graph norm of $d.$
Moreover, if $\omega$ is in $\cD_{\max}^{\reg}(d)$ then $\omega \in \cD_{[W]}(d)$ precisely when the vertical de Rham cohomology class of $\alpha(\omega)$ as a section over $Y$ is in $[W].$
Thus $\cD_{[W]}^{\reg}(d;g) = \cD_{[W]}^{\reg}(d;g')$ and, since taking the closure in $L^2$ with respect to $g$ is the same as taking the closure with respect to $g',$ we have
\begin{equation*}
	\cD_{[W]}(d;g) = \cD_{[W]}(d;g').
\end{equation*}
It follows that on a pseudomanifold of depth one, the domain of the exterior derivative corresponding to a de Rham mezzoperversity is independent of the choice of rigid suitably scaled wedge metric and hence 
\begin{equation*}
	\tH^\bullet _{[W], \dR}(\hat X) \text{ is independent of the choice of }g.
\end{equation*}
The same thing can be shown on a general pseudomanifold proceeding inductively one stratum at a time.

\begin{theorem}[Metric independence]
Let $\hat X$ be a pseudomanifold and let $[\bW]$ be a de Rham mezzoperversity, i.e., a sequential choice of flat subbundles of the vertical de Rham cohomology with Cheeger ideal boundary conditions at each non-Witt stratum. If $g$ and $g'$ are two rigid, suitably scaled, iterated wedge metrics on $\hat X$ then
\begin{equation*}
	\cD_{[\bW]}(d;g)
	= \cD_{[\bW]}(d;g')
\end{equation*}
and, in particular, the de Rham cohomology groups $\tH^\bullet _{[\bW], \dR}(\hat X)$ are independent of the choice of metric.
\end{theorem}

An easy corollary is that de Rham cohomology is invariant under pull-back by a diffeomorphism of manifolds with corners $\wt X \lra \wt N$ that intertwines iterated boundary fibration structures on $\wt X$ and $\wt N$ (i.e., an isomorphism in the sense of \S\ref{sec:TotalRes}).

\begin{corollary}\label{cor:DiffeoInv}
Let $\hat X$ and $\hat N$ be pseudomanifolds and $F: \wt X \lra \wt N$ a diffeomorphism intertwining the boundary fibration structures. Given a de Rham mezzoperversity $\bW_{\hat N}$ on $\hat N,$ pull-back by $F$ determines a de Rham mezzoperversity $F^*\bW_{\hat N}$ on $\hat X,$ and an isomorphism
\begin{equation*}
	F^*: \tH^\bullet _{[\bW_{\hat N}], \dR}(\hat N) \lra \tH^\bullet _{[F^*\bW_{\hat N}], \dR}(\hat X).
\end{equation*}
\end{corollary}

This is proved in \cite[Theorem 5.4]{ALMP13}, we do not go into the details as in the next section we discuss the more general case of stratified homotopy equivalences.

\subsection{Stratified homotopy invariance of de Rham cohomology}

If $X$ and $N$ are smooth closed manifolds and $F: X \lra N$ is a smooth homotopy equivalence between them, then
$F$ defines a pull-back map of differential forms that commutes with the exterior derivative. Thus $F^*$ descends to a map on de Rham cohomology which is necessarily an isomorphism,
\begin{equation*}
	F^*: \tH^\bullet _{\dR}(N) \xlra{\cong} \tH^\bullet _{\dR}(X).
\end{equation*}
We want to prove an analogous statement for the $L^2$-de Rham cohomology of a pseudomanifold corresponding to a mezzoperversity.

One problem is that, even on smooth Riemannian manifolds, the pull-back by a smooth homotopy equivalence need not induce a bounded map on $L^2$-differential forms.
For example, if $F$ is an embedding and $u_{\eps} \in \CI(N)$ is the indicator function of an $\eps$-ball around $F(X),$ then the $L^2$-norm of $u_{\eps}$ goes to zero with $\eps,$ while $F^*u_{\eps}=1$ has constant $L^2$ norm. Thus 
\begin{equation*}
	F^*: \{\text{ step functions }\} \subseteq L^2(N) \lra L^2(X)
\end{equation*}
is not even closable.
There is a clever way around this, developed by Hilsum-Skandalis \cite{Hilsum-Skandalis} for smooth manifolds, and extended in \cite{ALMP11, ALMP13.2} to pseudomanifolds. We replace $F^*$ by a map on $L^2$-differential forms called the {\em Hilsum-Skandalis replacement} and denoted $HS(F).$\\

First let us explain our assumptions on maps between stratified spaces.
By a {\em smooth map between stratified spaces} we mean a continuous map  $\hat F: \hat X \lra \hat N$ that lifts to a smooth map between the resolutions of the stratified spaces,
$F: \wt X \lra \wt N.$ Equivalently, these are the smooth maps $\wt X \lra \wt N$ such that 
\begin{equation*}
	F^*\CI_{\Phi}(\wt N) \subseteq \CI_{\Phi}(\wt X).
\end{equation*}
We denote these maps by
\begin{equation*}
	\CI(\hat X, \hat N) \equiv \CI_{\Phi}(\wt X, \wt N).
\end{equation*}
We say that $\hat F \in \CI(\hat X, \hat N)$ is {\em stratum-preserving} if the inverse image of a stratum is a union of strata,
and {\em strongly stratum preserving} if the codimension of a stratum of $\hat N$ coincides with the codimension of its inverse image in $\hat X.$
We say that $\hat F \in \CI(\hat X, \hat N)$ is a {\em stratified homotopy equivalence} if there is a smooth strongly stratum preserving map $\hat G \in \CI(\hat N, \hat X)$ such that
$\hat F \circ \hat G$ and $\hat G \circ \hat F$ are smoothly connected to the identity through smooth strongly stratum preserving maps.

We will assign a Hilsum-Skandalis replacement $HS(F)$ to a smooth strongly stratum preserving map. These replacements will descend to cohomology and only depend on the homotopy class of $F$ within smooth strongly stratum preserving maps.
The construction of $HS(F)$ is slightly complicated, but the main idea is that we want to `spread out' the image of $F$ to avoid precisely the sort of counterexample above. In \cite{Hilsum-Skandalis} this `spreading out' is done by embedding $N$ into a Euclidean space and averaging among nearby points. We will do something similar using the geodesic flow of a wedge metric. 

A smooth stratum-preserving map induces a map of the iterated wedge tangent bundles
\begin{equation*}
	DF: {}^w T\wt X \lra {}^w T\wt N
\end{equation*}
and hence a pull-back map
\begin{equation*}
	F^*:\CI(\wt L; \Lambda^\bullet ({}^w T^*\wt N)) \lra \CI(\wt X; \Lambda^\bullet ({}^w T^*\wt X)),
\end{equation*}
but as mentioned above, this map need not be bounded in $L^2.$
To construct a replacement, let us start with ${}^w \bbB \wt N \xlra{\pi_N} \wt N,$ the ball subbundle of  the dual bundle to ${}^w T^*\wt N,$ ${}^w T\wt N \lra \wt N,$ and the pull-back diagram
\begin{equation*}
	\xymatrix{
	F^*({}^w\bbB \wt N) \ar[r]^-{\bbB(F)} \ar[d]^{\pi_X} & {}^w \bbB \wt N \ar[d]^{\pi_N} \\
	\wt X \ar[r]^-{F} & \wt N }
\end{equation*}
Next, choose a Thom form $\cT_N$ for ${}^w \bbB \wt N \xlra{\pi_N} \wt N,$ and let $\bar\exp_N: {}^w \bbB \wt N \lra \wt N$ denote the continuous extension of the exponential map of a (complete metric conformal to an) iterated wedge metric from $\wt N^{\circ}$ to $\wt N.$
Then we set
\begin{equation*}
\begin{gathered}
	HS(F): \CI(\wt N; \Lambda^\bullet ({}^w T^*\wt N)) \lra \CI(\wt X; \Lambda^\bullet ({}^w T^*\wt X)), \\
	HS(F)(\omega) = (\pi_X)_*(\bbB(F)^*\cT_N\wedge(\bar\exp_N\circ\bbB(F))^*\omega).
\end{gathered}
\end{equation*}
This map commutes with the exterior derivative 
\begin{equation*}
	HS(F)d_N = d_X HS(F)
\end{equation*}
directly from its definition. If we endow ${}^w \bbB \wt N$ with a $\pi_N$-submersion metric compatible with $g_N$ and $F^*({}^w\bbB \wt N)$ with a $\pi_X$-submersion metric compatible with $g_X,$ which is always possible, then it is easy to see that $HS(F)$ is a bounded map with respect to the $L^2$-norms and so extends to
\begin{equation*}
	HS(F): L^2(\wt N; \Lambda^\bullet ({}^w T^*\wt N)) \lra L^2(\wt X; \Lambda^\bullet ({}^w T^*\wt X)).
\end{equation*}
Moreover, since $HS(F)$ sends forms with compact support in $\wt N^{\circ}$ to forms with compact support in $\wt X^{\circ},$ it induces maps
\begin{equation*}
	HS(F): \cD_{\max}(d_N) \lra \cD_{\max}(d_X), \quad
	HS(F): \cD_{\min}(d_N) \lra \cD_{\min}(d_X).
\end{equation*}

It turns out that $HS(F)$ defines a map on mezzoperversities. 
The key fact is that, when $\hat F$ sends the stratum $Y\subseteq \hat X$ to the stratum $Q \subseteq \hat N,$ the map $F$ restricts to the corresponding boundary hypersurfaces to a fiber bundle map
\begin{equation*}
	\xymatrix{
	\wt X  \ar@{}[r]|-*[@]{\supseteq} & M \ar[r]^F\ar[d] & P\ar[d]  \ar@{}[r]|-*[@]{\subseteq} & \wt N \\
	& Y \ar[r]^{\hat F} & Q &}
\end{equation*}
and the induced map between the link of a point in $Y$ and the link of the corresponding point in $Q$ will be a smooth strongly stratum preserving map.
The upshot is that, compatibly with $HS(F),$ there is a family of maps $HS(F\rest{Z_q}),$ $q \in Y,$ between vertical differential forms over $Y$ and vertical differential forms over $Q.$ 

If $\hat X$ and $\hat N$ are depth one non-Witt spaces, then $HS(F\rest{Z_q})$ induces a map between $\tH^{v/2}_{\dR}(M/Y)$ and $\tH^{v/2}_{\dR}(P/Q)$ (note that $v$ is the same in both spaces because $F$ is strongly stratum preserving). The origin of the flat connections in the exterior derivatives described in \S\ref{sec:CohoFiberBdles} can be used to show that they are intertwined by $HS(F\rest{Z_q}).$ 
Thus the boundary behavior of $HS(F)$ defines a `pull-back' map between mezzoperversities on $\hat N$ and mezzoperversities on $\hat X,$ which we denote
\begin{equation*}
	W_{\hat N} \mapsto F^{\sharp}W_{\hat N}.
\end{equation*}
By working with elements in the regular domain from \S\ref{sec:MetInd}, $\cD_{[W_{\hat N}]}^{\reg}(d_N),$ we can show that
\begin{equation*}
	HS(F) (\cD_{[W_{\hat N}]}(d_N)) \subseteq
	\cD_{[F^{\sharp}W_{\hat N}]}(d_X)
\end{equation*}
and hence there is a map on cohomology,
\begin{equation*}
	HS(F): \tH^\bullet _{[W_{\hat N}], \dR}(\hat N) \lra \tH^\bullet _{[F^{\sharp}W_{\hat N}], \dR}(\hat X).
\end{equation*}
An argument in \cite{Hilsum-Skandalis}, extended in \cite[Lemma 9.1]{ALMP11}, shows that this map is unchanged if $F$ is allowed to vary among smooth strongly stratum preserving maps, and that $HS(F\circ G) = HS(G) \circ HS(F)$ as maps on cohomology, for all smooth strongly stratum preserving maps $G:\hat N \lra \hat R.$
Iterating these constructions we obtain \cite[Theorem 3.7]{ALMP13.2}.

\begin{theorem}
Let $\hat X$ and $\hat N$ be pseudomanifolds, $F: \hat X \lra \hat N$ a smooth strongly stratum preserving map, and 
\begin{equation*}
	HS(F):L^2(\wt N; \Lambda^\bullet ({}^w T^*\wt N)) \lra L^2(\wt X; \Lambda^\bullet ({}^w T^*\wt X))
\end{equation*}
the Hilsum-Skandalis replacement of $F.$
Given a mezzoperversity $\bW_{\hat N}$ on $\hat N,$ $HS(F)$ induces a mezzoperversity $F^{\sharp}\bW_{\hat N}$ on $\hat X$ such that
\begin{equation*}
	HS(F)(\cD_{[\bW_{\hat N}]}(d_N)) \subseteq \cD_{[F^\sharp\bW_{\hat N}]}(d_X).
\end{equation*}
The induced map in de Rham cohomology
\begin{equation*}
	HS(F): \tH^\bullet _{[\bW_{\hat N}], \dR}(\hat N) \lra \tH^\bullet _{[F^\sharp\bW_{\hat N}], \dR}(\hat X)
\end{equation*}
is unchanged by varying $F$ smoothly through smooth strongly stratum preserving maps and is an isomorphism if $F$ is a stratified homotopy equivalence.
\end{theorem}

\subsection{Properties of the signature} \label{sec:PropSign}

Let $\hat X$ be an oriented Cheeger space and $[\bW]$ a self-dual de Rham mezzoperversity.
As mentioned above, the intersection pairing
\begin{equation*}
	B: \tH^\bullet _{[\bW], \dR}(\hat X) \times \tH^\bullet _{[\bW], \dR}(\hat X) \lra \bbR
\end{equation*}
is a non-degenerate quadratic form, and we denote its signature by
\begin{equation*}
	\sign(\hat X, \bW).
\end{equation*}
A careful inspection of the Hilsum-Skandalis replacement shows that the pull-back of a self-dual mezzoperversity is a self-dual mezzoperversity (indeed the forms $Q_k$ from \eqref{eq:Qk} are compatible with $HS(F)$) and in fact:
\begin{proposition}
If $F: \hat X \lra \hat N$ is a stratified homotopy equivalence and $\bW_N$ is a self-dual mezzoperversity on $\hat N,$ then
\begin{equation*}
	\sign(\hat X, F^{\sharp}\bW_{\hat N}) = \sign(\hat N, \bW_{\hat N})
\end{equation*}
\end{proposition}

If we endow $\hat X$ with a rigid suitably scaled iterated wedge metric $g,$ then we can identify the signature with the Fredholm index of an elliptic operator.
Assume $n = \dim \hat X$ is even.
Let 
\begin{equation*}
	L^2(\wt X; \Lambda^\bullet ({}^w T_{\bbC}^*\wt X))
\end{equation*}
denote the $L^2$-differential forms with complex coefficients, let
\begin{equation*}
\begin{gathered}
	\tau: L^2(\wt X; \Lambda^\bullet ({}^w T_{\bbC}^*\wt X)) \lra L^2(\wt X; \Lambda^\bullet ({}^w T_{\bbC}^*\wt X))\\
	\tau\rest{L^2(\wt X; \Lambda^p({}^w T_{\bbC}^*\wt X))} = i^{p(p-1)+n/2} *
\end{gathered}
\end{equation*}
and note that $\tau^2 = \Id.$
Denote the $\pm 1$ eigenspaces of $\tau$ by
\begin{equation*}
	L^2(\wt X; \Lambda^\bullet ({}^w T_{\bbC}^*\wt X))_{\pm}.
\end{equation*}
Since $\eth_{\dR}$ anticommutes with $\tau,$ it interchanges these two eigenspaces, and the restriction to either of these eigenspaces is known as the signature operator, $\eth_{\sign}.$
If $\bW$ is a self-dual mezzoperversity then 
\begin{equation*}
	\tau: \cD_{\bW}(\eth_{\dR}) \lra \cD_{\bW}(\eth_{\dR})
\end{equation*}
(with $\cD_{\bW}(\eth_{\dR}) \subseteq L^2(\wt X; \Lambda^\bullet ({}^w T_{\bbC}^*\wt X))$ after complexification) and we define 
\begin{equation*}
	\cD_{\bW}(\eth_{\sign}) = \cD_{\bW}(\eth_{\dR})\cap L^2(\wt X; \Lambda^\bullet ({}^w T_{\bbC}^*\wt X))_{+}.
\end{equation*}

\begin{proposition}
With notation as above, the operator
\begin{equation*}
	\eth_{\sign}: \cD_{\bW}(\eth_{\sign}) \subseteq L^2(\wt X; \Lambda^\bullet ({}^w T_{\bbC}^*\wt X))_{+} \lra
	L^2(\wt X; \Lambda^\bullet ({}^w T_{\bbC}^*\wt X))_{-}
\end{equation*}
is Fredholm and its index is equal to $\sign(\hat X, \bW).$
\end{proposition}

One of the most important properties of the signature of a smooth manifold is that it is unchanged under oriented bordism.
That is, if $X$ and $N$ are oriented smooth manifolds, and there is an oriented smooth manifold with boundary $\sM$ such that
\begin{equation*}
	\pa\sM = X \sqcup (-N),
\end{equation*}
then $\sign(X) = \sign(N).$

Now suppose that 
$\bW_{\!\hat X},$ $\bW_{\!\hat N}$ are self-dual mezzoperversities on the oriented Cheeger spaces $\hat X$ and $\hat N,$ respectively.
A bordism between $(\hat X, \bW_{\hat X})$ and $(\hat N, \bW_{\hat N})$ is a Cheeger space $\hat{\sM}$ with collared boundary,
\begin{equation*}
	\pa\hat \sM = \hat X \sqcup (-\hat N)
\end{equation*}
together with a mezzoperversity $\bW_{\hat{\sM}}$ that restricts to $\bW_{\hat X}$ in a collar of $\hat X$ and restricts to $\bW_{\hat N}$ in a collar of $\hat N.$
By showing that the signature operator of a Cheeger space defines a K-homology class and then applying K-homology theory, we proved in \cite[Theorem 4.7]{ALMP13.2} the bordism invariance of the signature.

\begin{proposition}
If there is a bordism between $(\hat X, \bW_{\hat X})$ and $(\hat N, \bW_{\hat N})$ then
\begin{equation*}
	\sign(\hat X, \bW_{\hat X}) = \sign(\hat N, \bW_{\hat N}).
\end{equation*}
\end{proposition}

A clever observation of Banal \cite{Banagl:LClass} (in the context of $L$-spaces and self-dual sheaves, cf. \cite[Theorem 4.8]{ALMP13.2}) is that any two mezzoperversities on a single Cheeger space are bordant. Thus this proposition shows that the signature of a Cheeger space is independent of the choice of self-dual mezzoperversity! We can restate the results so far in a more satisfying fashion:

\begin{theorem}
The signature of an oriented Cheeger space is unchanged under stratified homotopy equivalences or Cheeger space bordism.
\end{theorem}

\subsection{$L$-class and higher signatures} \label{sec:LClass}

Let $\hat X$ be a Cheeger space, $\bW$ a self-dual mezzoperversity, and let $\bbS^\ell$ denote the $\ell$-dimensional sphere.
Goresky-MacPherson \cite{Goresky-MacPherson:IH} showed that $\hat X$ can be embedded into $\bbR^m$ for $m\gg1$ and any 
\begin{equation*}
	F: \hat X \lra \bbS^{2q}
\end{equation*}
can be arranged via homotopy to be the restriction of a map defined on all of $\bbR^m$ and to be transverse to a given point $\pt \in \bbS^{2q}.$

\begin{lemma}
The inverse image of $\pt$ under $F$ is a Cheeger space and hence has a signature.
Homotopic maps produce cobordant Cheeger spaces and hence the same signature.
\end{lemma}

A trick due to Thom \cite{Thom:Espaces,Thom:Classes} uses this construction to define an $L$-class in homology.
Namely, if $4q>n+1,$ then we have
\begin{equation*}
	\tH^{2q}(\hat X;\bbQ) \cong [\hat X, \bbS^{2q}] \otimes \bbQ
\end{equation*}
and so we have a map
\begin{equation*}
	\xymatrix @R=1pt
	{\tH^{2q}(\hat X; \bbQ) \ar[r] & \bbQ \\
	[F] \ar@{|->}[r] & \sign(F^{-1}(\pt)) }
\end{equation*}
which defines an element in $H_{2q}(\hat X;\bbQ).$
If $4q \leq n+1,$ we choose any $\ell >n+1$ such that $4(q+\ell) > n + \ell +1,$
and then $\hat X \times \bbS^{2\ell}$ is a Cheeger space and
\begin{equation*}
	\tH^{2q+2\ell}(\hat X \times \bbS^{2\ell};\bbQ) \cong \tH^{2q}(\hat X;\bbQ)
\end{equation*}
so we can proceed as before.

The homology class defined in this way is called the {\em $L$-class} and denoted
\begin{equation*}
	\cL(\hat X) \in \tH_{\ev}(\hat X;\bbQ).
\end{equation*}
Now let $r: \hat X \lra B\pi_1\hat X$ be the classifying map for the universal cover of $\hat X.$
The higher signatures of $\hat X$ are defined to be the rational numbers
\begin{equation*}
	\{ \ang{\gamma, r_*\cL(\hat X)} : \gamma \in \tH^\bullet (B\pi_1\hat X; \bbQ) \}
\end{equation*}
where $\ang{\cdot, \cdot}$ denotes the pairing between cohomology and homology of $B\pi_1\hat X.$

The Novikov conjecture for $\hat X$ is the statement that the higher signatures are stratified homotopy invariant (see, e.g., \cite{Rosenberg:Survey}).
The strong Novikov conjecture for $\pi_1\hat X$ states that an `assembly map' between two K-homology groups is rationally injective.
The latter conjecture is known for many classes of groups. The analytic approach to the Novikov conjecture is to use index theory to `reduce' it to the strong Novikov conjecture. 
This reduces stratified homotopy invariance to the homotopy invariance of the index of a `higher' signature operator and, by using the Hilsum-Skandalis replacement construction, we were able to prove this analytically on both the Witt and Cheeger space settings. (For Witt spaces the index of this higher signature can also be computed topologically, at least after changing coefficients to $\bbQ,$ as a `symmetric signature.' There are now two constructions of the symmetric signature of a Witt space, \cite{Banagl:msri,Friedman-McClure}, and the latter is  stratified homotopy invariant.)  In \cite[Theorem 5.8]{ALMP13.2} we carry out the analytic approach to the Novikov conjecture for Cheeger spaces.

\begin{theorem}
Let $\hat X$ be a Cheeger space. If the strong Novikov conjecture is true for $\pi_1X,$ then the higher signatures of $\hat X$ are stratified homotopy invariants.
\end{theorem}

\subsection{Refined intersection (co)homology} \label{sec:RefIH}

In this subsection we explain a `de Rham theorem' showing that the de Rham cohomology groups defined analytically above have sheaf-theroretic descriptions. Thus we will shift gears and work more topologically. We will show that the cohomology corresponding to a mezzoperversity lies `in between' the two middle perversities. To make this precise we will first recall the sheaf-theoretic approach to intersection homology.

It is convenient to use a notation that emphasizes dimension as opposed to depth. So given $\hat X,$ an $n$-dimensional pseudomanifold, let
\begin{equation*}
	\emptyset = X_{-1} \subseteq X_0 \subseteq \ldots \subseteq X_{n-2} \subseteq X_n = \hat X
\end{equation*}
be an increasing sequence of closed subsets such that the stratum
\begin{equation*}
	S_k = X_k \setminus X_{k-1}
\end{equation*}
is either empty or a smooth manifold of dimension $k.$
The regular part of $\hat X$ is equal to $S_n = \hat X \setminus X_{n-2}.$
Let
\begin{equation*}
	\cV_k = \hat X \setminus X_{n-k}
\end{equation*}
be the increasing sequence of open subsets of $\hat X,$ starting with $\cV_2 = S_n = \hat X^{\reg}$ and satisfying
\begin{equation*}
	\cV_{k+1} = \cV_k \cup S_{n-k}.
\end{equation*}
Denote the inclusions by
\begin{equation*}
	\xymatrix{
	\cV_k \ar@{^(->}[r]^{i_k} & \cV_{k+1}  \ar@{<-^)}[r]^{j_k} & S_{n-k}.
	}
\end{equation*}

Let $\Sh(\hat X)$ denote the category of sheaves of $\bbR$-vector spaces on $\hat X.$
A constructible differential graded sheaf is a complex $\sc A$ such that $\sc A|_{S_k}$ has locally constant cohomology sheaves and has finitely generated stalk cohomology.
We denote the category of constructible differential graded sheaves by $\Sh^{\bullet}(\hat X).$
The derived category of bounded constructible complexes, $D(\hat X),$ has as objects the cochain complexes of bounded constructible sheaves but quasi-isomorphic chain complexes are now considered isomorphic.
Every object $\sc A$ in $D(\hat X)$ has associated to it a derived sheaf: the sheafification of the presheaf
\begin{equation*}
	\cU \mapsto \tH^\bullet(\sc A(\cU)),
\end{equation*}
which we denote $\bH^{\bullet}(\sc A) \in \Sh^{\bullet}(\hat X).$
The global hypercohomology of the derived sheaf is the cohomology of the global sections of an injective (more generally, any acyclic) resolution and is denoted $\bbH^{\bullet}(\hat X; \sc A).$
For a review of these concepts, see \cite{GM2, Borel, BanaglLong, Kirwan-Woolf}.

A perversity $\bar p$ is a function from $\{ 2, 3, \ldots \}$ to $\bbN_0$ such that $\bar p (c)$ and also $\bar q(c) = c-2-\bar p(c)$ are positive and nondecreasing functions of $c.$
We refer to $\bar q(c)$ as the dual perversity to $\bar p(c).$
The sheaf theoretic formulation of intersection cohomology in \cite{GM2} assigns to each perversity $\bar p$ an object in the derived category
\begin{equation*}
	\IC p \in D(\hat X)
\end{equation*}
with the property that its hypercohomology is the intersection cohomology of $\hat X$ with perversity $\bar p,$
\begin{equation*}
	\mathrm{IH}_{\bar p}^{i}(\hat X) = \bbH^{i}(\hat X; \IC p).
\end{equation*}

There are axioms characterizing $\IC p.$ For example: any constructible bounded complex of sheaves of real vector spaces, $\sc S,$ that satisfies
\begin{itemize}
\item [(a)] 
Normalization: There is an isomorphism of the restriction of $\sc S$ to the regular part $\cV_2$ of $\hat X$ and the constant rank $1$ sheaf over $\cV_2,$
\begin{equation*}
	\nu^{\bS}:\bbR_{\cV_2} \xlra{\cong} \sc S\rest{\cV_2},
\end{equation*}
\item [(b)]
Lower bound: $\tH^\ell(i_x^*\sc S) =0$ for any $x\in \hat X$ and $\ell < 0,$
\item [(c)]
$\bar p$-stalk vanishing: $\tH^\ell(i_x^*\sc S) =0$ for  any $x\in \cV_{k+1}\setminus \cV_2$ and $\ell > \bar p(k),$
\item [(d)]
$\bar p$-costalk vanishing: $\tH^\ell(i_x^!\sc S) =0$ for  any $x\in S_{n-k}$ and $\ell \leq \bar p(k)+n-k+1,$
\end{itemize}
defines $\IC p$ in $D(\hat X).$

There is also a beautiful construction of $\IC p$ in $D(\hat X)$ due to Deligne using only standard constructions of sheaf theory, namely
\begin{equation*}
	\IC p = \tau_{\leq \bar p(n)}Ri_{n*}\cdots \tau_{\leq \bar p(2)}Ri_{2*}\bbR_{\cV_2}.
\end{equation*}
We can think of this as efficiently patching together the local computation:
\begin{equation*}
	x \in S_{n-k} \implies
	H^i(\sc{IC}_{\bar p})_x \cong
	\begin{cases}
	\bbH^i(Z; \sc{IC}_{\bar p}) & i \leq \bar p(k) \\
	0 & i > \bar p(k) 
	\end{cases}
\end{equation*}
$ $

Now let us modify this theory to admit mezzoperverisites. Say that a constructible bounded complex of sheaves $\sc S$ is a {\em refined middle perversity complex of sheaves} if it satisfies the axioms:
\begin{quote}
	(a) Normalization, \quad
	(b) Lower bound, \\
	(c) $\bar n$-stalk vanishing, \quad
	(d) $\bar m$-costalk vanishing,
\end{quote}
and denote by $RP(\hat X)$ the full subcategory of $D(\hat X)$ with these objects.

The category $RP(\hat X)$ contains $\IC n$ and $\IC m.$ Any complex $\sc S \in RP(\hat X)$ has unique maps \cite[Proposition 1.6]{ABLMP}
\begin{equation*}
	\IC m \lra \sc S \lra \IC n
\end{equation*}
compatible with the normalizations on $\cV_2,$ and factoring the canonical map $\IC m \lra \IC n.$ In this sense $\sc S$ is `in between' $\IC m$ and $\IC n.$
If $\hat X$ is Witt then $\IC m \cong \IC n$ and this is the only object in $RP(\hat X).$

If $\hat X$ is not Witt then every object $\sc S \in RP(\hat X)$ is obtained from $\IC m$ by the `addition' of a suitably compatible collection, $\cL,$ of locally constant sheaves, one on each stratum where the Witt condition does not hold. We refer to this collection of locally constant sheaves as a {\em topological mezzoperveristy} \cite[Definition 2.1]{ABLMP}. Given a topological mezzoperversity $\cL = (W(S_{n-3}), W(S_{n-5}), \ldots),$ Banagl \cite{Banagl:ExtendingIH} has constructed a modified truncation function $\tau_{\leq \cL(\cdot)}$ and we can define
\begin{equation*}
	\sc{IC}_{\cL} = \tau_{\leq \cL(n)}Ri_{n*} \cdots \tau_{\leq \cL(2)} Ri_{2*} \bbR_{\cU_2}.
\end{equation*}
This satisfies the local computation
\begin{equation*}
\begin{gathered}
	x \in S_{n-k}, \; k \ev \implies
	H^i(\sc{IC}_{\cL})_x \cong
	\begin{cases}
	\bbH^i(Z; \sc{IC}_{\cL}) & i \leq \bar m(k) \\
	0 & i > \bar m(k) 
	\end{cases}\\
	x \in S_{n-k}, \; k \odd \implies
	H^i(\sc{IC}_{\cL})_x \cong
	\begin{cases}
	\bbH^i(Z; \sc{IC}_{\cL}) & i < \bar m(k) \\
	W(S_{n-k}) & i = \tfrac{k-1}2 \\
	0 & i > \bar n(k) 
	\end{cases}
\end{gathered}
\end{equation*}
and so is a refinement of Deligne's construction in which we truncate the cohomology not just at a given degree but at a given locally constant sheaf. 
We prove \cite[Theorem 2.2]{ABLMP} that every refined middle perversity complex of sheaves is given by $\sc{IC}_{\cL}$ for some topological mezzoperversity $\cL.$ \\

Now endow $\hat X$ with a rigid suitably scaled iterated wedge metric $g,$ and let $\cW$ be a de Rham mezzoperversity.
Define a presheaf on $\hat X$ by assigning to each open set $\cU \subseteq \hat X$ the vector space
\begin{equation*}
	\cD_{\cW}(\cU) = \{ \omega \in \cD_{\cW}(d) : \supp \omega \subseteq (\cU \cap \hat X^{\reg}) \}
\end{equation*}
and assigning to each inclusion $j:\cU \subseteq \cV$ the restriction map
\begin{equation*}
	j^*:\cD_{\cW}(\cV) \lra \cD_{\cW}(\cU).
\end{equation*}
The exterior derivative makes this into a complex of presheaves.
Denote the sheafification by 
\begin{equation*}
	\bL^2_{\cW}\sc\bOm \in \Sh^{\bullet}(\hat X).
\end{equation*}

(Note that, for $i\hookrightarrow \widehat X^{reg}\to \widehat X,$ we are using $i_!(D_W(d)).$)

\begin{theorem}
Let $\hat X$ be a pseudomanifold endowed with a rigid suitably scaled iterated wedge metric.
Given any de Rham mezzoperversity $\cW,$ the sheaf $\bL^2_{\cW}\sc\bOm$ defines an object in $RP(\hat X),$ and hence we have an equality in $D(\hat X),$
\begin{equation*}
	\bL^2_{\cW}\sc\bOm \cong \sc{IC}_{\cL(\cW)}
\end{equation*}
for some topological mezzoperversity $\cL(\cW).$
This establishes a one-to-one correspondence between de Rham mezzoperversities and topological mezzoperversities
\begin{equation*}
	\cW \leftrightarrow \cL(\cW).
\end{equation*}
Thus we have an equivalence between the hypercohomology of a refined middle perversity complex of sheaves and the $L^2$-de Rham cohomology with Cheeger ideal boundary conditions,
\begin{equation*}
	\tH_{\cW, \dR}^*(\hat X) = \bbH^\bullet (\hat X; \sc{IC}_{\cL(\cW)}).
\end{equation*}
\end{theorem}

On any pseudomanifold there are two canonical choices of mezzoperversity. Indeed, given a flat bundle $\cH \lra Y,$ two canonical flat subbundles are the zero bundle and the entire bundle $\cH$ itself. Let $\cW_0$ be the mezzoperversity corresponding to always choosing the zero bundle, and $\cW_{\max}$ the mezzoperversity corresponding to always choosing the entire bundle of vertical cohomology. This theorem shows that
\begin{equation*}
	\tH_{\cW_0, \dR}^*(\hat X) = \mathrm{IH}_{\bar m}^*(\hat X), \quad
	\tH_{\cW_{\max}, \dR}^*(\hat X) = \mathrm{IH}_{\bar n}^*(\hat X).
\end{equation*}
(The latter can be identified with the cohomology of $\cD_{\max}(d)$ and the former with the cohomology of $\cD_{\min}(d),$ so in these cases the isomorphisms follow from the work of Cheeger described in \S\ref{sec:Background}.)

The definition of $RP(\hat X)$ only requires a {\em topological stratification.} In this generality, in \cite{Banagl:ExtendingIH}, Banagl studied a category $SD(\hat X)$ defined just like $RP(\hat X)$ but with the extra requirement of self-duality. Spaces $\hat X$ with $SD(\hat X)$ non-empty are called $L$-spaces. We show in \cite{ABLMP} that a Thom-Mather stratified space is an $L$-space if and only if it is a Cheeger space.

\section{Stratified spaces and their resolutions} \label{sec:StratSpaces}

In this section we will review the relation between stratified spaces and manifolds with corners.

\subsection{Thom-Mather stratified spaces} \label{sec:ThomMather}

We follow \cite{Mather:Notes, Verona:StratMap}, see also \cite{Pflaum:Study}.
Our aim in this section is to discuss the definition and some basic concepts culminating in the `local cone structure' of a Thom-Mather stratified space.

All topological spaces will be `nice' meaning Hausdorff, locally compact topological space with a countable basis for its topology.
A subset $A$ of a topological space $V$ is locally closed if every point $a\in A$ has a neighborhood $\cU$ in $V$ such that $A\cap \cU$ is closed in $\cU.$
A collection $\cS$ of subsets of $V$ is locally finite if every point $v \in V$ has a neighborhood that intersects only finitely many sets in $\cS.$

\begin{definition}
A  (Thom-Mather) {\em stratified space} is a triple $(\hat X, \sS, \sT)$ in which $\hat X$ is the space, $\sS$ is a decomposition of $\hat X$ into `strata', and $\sT$ is a collection of `tubes', one per stratum.
Specifically:
\begin{itemize}
\item $\hat X$ is a `nice' topological space.
\item $\sS$ is a locally finite collection of locally closed subsets of $\hat X.$ Each $Y \in \sS$ is a topological manifold in the induced topology and is provided with a smooth structure. The family $\sS$ satisfies the frontier condition
\begin{equation}\label{eq:Frontier}
	Y, Y' \in \sS \Mand Y \cap \bar{Y'} \neq\emptyset \implies Y \subseteq \bar{Y'}.
\end{equation}
\item $\sT$ contains, for each $Y \in \sS,$ a triple $(\cT_Y, \pi_Y, \rho_Y)$ in which $\cT_Y$ is an open neighborhood of $Y$ in $\hat X,$ and
\begin{equation*}
	\pi_Y: \cT_Y \lra Y, \quad \rho_Y: \cT_Y \lra \bbR^+
\end{equation*}
are continuous functions such that $\pi_Y\rest{Y} = \mathrm{id}_Y$ and $\rho_Y^{-1}(0) = Y.$ The neighborhood $\cT_Y$ is known as a tubular neighborhood of $Y$ in $\hat X$ and the triple $(\cT_Y, \pi_Y, \rho_Y)$ as the control data of $Y$ in $\hat X.$
\end{itemize}
The control data are compatible in that if $Y, Y' \in \sS$ are distinct strata such that $\cT_{Y} \cap Y' \neq \emptyset$ then 
\begin{equation}\label{eq:A8}
	(\pi_Y, \rho_Y)\rest{\cT_Y \cap Y'}: \cT_Y\cap Y' \lra Y\times (0,\infty)
\end{equation}
is a smooth submersion, and
\begin{equation}\label{eq:A9}
	\pi_Y\circ \pi_{Y'} = \pi_Y, \quad \rho_Y \circ \pi_{Y'} = \rho_Y
\end{equation}
when both sides are defined (i.e., on $\cT_Y \cap \pi_{Y'}^{-1}(\cT_Y \cap Y')$). 
\end{definition}

A stratified space is a {\em pseudomanifold} if the codimension of each singular stratum is at least two.

\begin{example}
The {\em trivial stratification of a smooth manifold} $L$ has a single stratum equal to $L$ with control data $\cT_L =L,$ $\pi_L = \id,$ $\rho_L\equiv 0.$

If $(\hat X, \sS, \sT)$ is a stratified space and $L$ is a smooth manifold then $\hat X \times L$ is naturally a stratified space with strata $Y \times L,$ $Y \in \sS.$

If $(\hat X, \sS, \sT)$ is a stratified space and $A \subseteq \hat X$ is either an open subset or a closed subset that is a union of strata, then $A$ is naturally a stratified space with strata $Y \cap A,$ $Y \in \sS.$
\end{example}

\begin{example}[Whitney stratified spaces]
Let $L$ be a smooth manifold and let $R, S$ be smooth submanifolds of $L.$ 
We say that $(R, S)$ satisfies Whitney's condition $(\bB)$ if for every 
\begin{itemize}
\item $x \in R$ 
\item sequence $(y_k) \subseteq S$ such that $y_k \to x$ and $T_{y_k}S \to \tau \subseteq T_xL$ (as sections of a Grassmanian bundle over $L$), 
\item sequence $(x_k) \subseteq R$ such that $x_k \to x,$ $x_k \neq y_k$ and the secant lines $\overline{x_ky_k}$ (defined in a coordinate chart) converge to $\ell \subseteq T_xL$
\end{itemize}
we have $\ell \subseteq \tau.$

A Whitney stratification of a subset $A \subseteq L$ is a locally finite collection $\sS$ of pairwise disjoint smooth submanifolds of $L$ covering $A,$ such that any $Y, Y' \in \sS$ satisfy the frontier condition \eqref{eq:Frontier} and Whitney's condition $(\bB).$ Spaces admitting a Whitney stratification include 
algebraic subvarieties of $\bbR^n$ or $\bbC^n$ \cite{Whitney}, semianalytic and sub analytic sets \cite{Lojasiewicz}, and `o-minimal structures' \cite{VanDenDries-Miller}.

Whitney stratified spaces are intimately connected to Thom-Mather stratified spaces. Mather proved \cite{Mather:Notes} that every Whitney stratified space admits a Thom-Mather stratification. Moreover, Teufel \cite{Teufel} showed that every Thom-Mather stratified space can be embedded into a Euclidean space so that its image is Whitney stratified.
\end{example}

Given $Y, Y' \in \sS$ we write 
\begin{equation*}
	Y<Y' \iff Y \cap \bar {Y'} \neq \emptyset \iff Y \subseteq \bar{Y'}.
\end{equation*}
This is an order on $\sS$ and \eqref{eq:A8} shows that $Y<Y' \implies \dim Y < \dim Y'.$
Another useful consequence of \eqref{eq:Frontier} is that the closure of a stratum is a union of strata.
Indeed, this is equivalent to the frontier condition \eqref{eq:Frontier}.

The {\em depth of a stratum} is
\begin{equation*}
	\depth(Y) = \sup\{ n \in \bbN_0: \exists Y_i \in \sS \Mst Y = Y_0 < Y_1 <\ldots<Y_n \}
\end{equation*}
and the {\em depth of a stratification} is
\begin{equation*}
	\depth(\sS) = \sup\{ \depth(Y) : Y \in \sS \}.
\end{equation*}
The regular part of $\hat X$ consists of those strata of depth zero,
\begin{equation*}
	\hat X^{\reg} = \bigcup \{ Y \in \sS : \depth(Y) =0\}.
\end{equation*}

Two stratified spaces $(\hat X, \sS, \sT)$ and $(\hat X', \sS', \sT')$ are {\em equivalent} if $\hat X = \hat X'$ with the same strata (as smooth manifolds) and the control data at a stratum $Y,$
\begin{equation*}
	(\cT_Y, \pi_Y, \rho_Y), \quad (\cT_Y', \pi_Y', \rho_Y'),
\end{equation*}
coincide in a neighborhood of $Y.$

A {\em morphism between stratified spaces} $(\hat X, \sS_X, \sT_X)$ and $(\hat N, \sS_N, \sT_N)$ is a continuous map
\begin{equation*}
	f: \hat X \lra \hat N
\end{equation*}
that maps strata into strata such that, whenever $Y \in \sS_X,$ and $Q \in \sS_N$ are such that $f(Y) \subseteq Q,$
we have
\begin{equation*}
	f\rest{Y}:Y \lra Q \text{ is smooth}
\end{equation*}
and, for some $(\hat X, \sS_X, \sT_X')$ equivalent to $(\hat X, \sS_X, \sT_X),$
\begin{equation*}
	f(\cT_Y') \subseteq \cT_Q, \quad f\circ \pi_Y' = \pi_Q \circ f, \quad \rho_Y' = \rho_Q \circ f.
\end{equation*}
An {\em isomorphism of stratified spaces} is defined to be a bijective morphism whose inverse is also a morphism and hence is a homeomorphism that restricts to a diffeomorphism on each stratum.

A {\em submersive morphism} $f:\hat X \lra \hat N$ is a morphism such that, whenever $Y \in \sS_X,$ and $Q \in \sS_N,$ are such that $f(Y) \subseteq Q,$
we have
\begin{equation*}
	f\rest{Y}:Y \lra Q \text{ is a submersion.}
\end{equation*}

If $\hat X_0$ is a stratified space and $L$ is a smooth manifold, then $\hat X_0 \times L$ has a natural stratification with strata $\{ Y \times L : Y \in \sS\},$ 
and the projection onto the second factor
\begin{equation*}
	\hat X_0 \times L \xlra{\pi_L} L
\end{equation*}
is a submersive morphism.
We say that a submersive morphism $f:\hat X \lra L$ is a {\em trivial fibration} if it participates in a commutative diagram
\begin{equation*}
	\xymatrix{
	\hat X \ar[rd]_{f} \ar[rr]^{F} & & \hat X_0 \times L \ar[dl]^{\pi_L} \\
	& L &}
\end{equation*}
with $F$ an isomorphism of stratified spaces $F.$ 

A submersive morphism $f:\hat X \lra L$ onto a smooth manifold (with its trivial stratification) is a {\em locally trivial fibration} if every point $x \in L$ has a neighborhood $\cU$ such that $f\rest{f^{-1}(\cU)}:f^{-1}(\cU) \lra \cU$ is a trivial fibration.
If $L$ is connected then all of the fibers are isomorphic stratified spaces and we write
\begin{equation*}
	\hat X_0 \fib \hat X \xlra{f} L
\end{equation*}
where $\hat X_0$ is a fiber $f^{-1}(\pt).$ (Note that the `structure group' consists of isomorphisms of stratified spaces $\hat X_0 \lra \hat X_0.$)

\begin{theorem}[Thom's first isotopy lemma for stratified spaces] \label{thm:Thom1Strat}
If $f: \hat X \lra L$ is a proper submersive morphism between a stratified space $\hat X$ and a smooth manifold $L$ (with its trivial stratification), then $f$ is a locally trivial fibration.
\end{theorem}

This can be thought of as the version of Ehresmann's theorem for stratified spaces, and it indeed reduces to that theorem if $\hat X$ is smooth. A proof can be found in \cite[Theorem 2.6]{Verona:StratMap} (see also \cite[Corollary 10.2, Proposition 11.1]{Mather:Notes}, \cite[3.9.2]{Pflaum:Study}).

The main tool in the proof of Theorem \ref{thm:Thom1Strat} is the notion of a  controlled vector field on a stratified space (see \cite[\S 9]{Mather:Notes}, \cite[\S 2]{Verona:StratMap}).
A collection of smooth vector fields, one per stratum,
\begin{equation*}
	\xi = \{ \xi_Y \in \CI(Y, TY) : Y \in \sS \}
\end{equation*}
is a {\em controlled vector field} if whenever $Y, Y' \in \sS,$ $Y<Y',$ and $\zeta \in \cT_Y \cap Y'$ is sufficiently close to $Y,$
\begin{equation*}
	D\pi_Y(\xi_{Y'}(\zeta)) = \xi_Y(\pi_Y(\zeta)), \quad
	\xi_{Y'}\rho_Y = 0.
\end{equation*}
These are analogous to smooth vector fields on a smooth manifold in at least two ways:
First, the flow of a controlled vector field generates a (local) one parameter group of stratified space morphisms from an open subset $D_{\xi}\subseteq \hat X \times \bbR$ to $\hat X.$
Secondly, if $L$ is a smooth manifold and $f: \hat X \lra L$ is a submersive morphism, then every vector field on $L$ has a lift to a controlled vector field on $\hat X.$

Let us reexamine the control data of a stratified space in light of Theorem \ref{thm:Thom1Strat}.
If $Y$ is a stratum of $\hat X$ with control data $(\cT_Y, \pi_Y, \rho_Y)$ then $\cT_Y,$ as an open subset of $\hat X,$ inherits a stratification with strata 
\begin{equation*}
	\sS_{\cT_Y} = \{ \cT_Y \cap Y' : Y' \in \sS, Y \leq Y' \}.
\end{equation*}
The maps
\begin{equation*}
	(\pi_Y, \rho_Y)\rest{ \cT_Y\setminus Y}: \cT_Y \setminus Y \lra Y \times (0,\eps), \quad
	\pi_Y: \cT_Y \lra Y
\end{equation*}
are (for $\eps$ sufficiently small) proper submersive morphisms and hence locally trivial fibrations.

Let $a \in (0,\eps)$ and $K_Y = \rho_Y^{-1}(a).$
The canonical vector field on $(0,\eps)$ lifts to a controlled vector field on $\cT_Y \setminus Y$ and we can use its flow to identify, as stratified spaces, $\cT_Y\setminus Y$ with $K_Y  \times (0,\eps).$ The restriction $\psi_Y = \pi_Y\rest{\rho_Y^{-1}(a)}$ is a locally trivial fibration
\begin{equation*}
	\psi_Y: K_Y \lra Y
\end{equation*}
and the fiber over a point $q \in Y$ is a stratified space $Z_q$ called the {\em link of $\hat X$ at $Y.$} Local triviality implies that the links over points in the same connected component of $Y$ are isomorphic as stratified spaces (this explains the name `isotopy lemma').
It follows that the fiber over $q$ of the locally trivial fibration $\pi_Y$ is the cone over $Z_q,$
\begin{equation*}
	C_{\eps}(Z_q) = Z_q \times [0,\eps) / Z_q \times \{ 0\}
\end{equation*}
with its natural stratification.
Thus Thom's first isotopy lemma shows that a Thom-Mather stratified space is, locally near each stratum, a bundle of cones over simpler stratified spaces.

\subsection{Resolving a single stratum} \label{sec:ResolvingSingle}

Given $(\hat X, \sS, \sT)$ a stratified space let
\begin{equation*}
\begin{gathered}
	\sS^* = \{ Y \in \sS: \depth_{\sS}(Y) = \depth(\sS) \}, \\
	\hat X^* = \bigcup \{ Y \in \sS^* \}, \quad
	\cT_{\hat X^*} = \bigcup \{ \cT_Y : Y \in \sS^*\}.
\end{gathered}
\end{equation*}
We assume, replacing $(\hat X, \sS, \sT)$ with an equivalent stratified space if necessary, that 
\begin{equation*}
	Y, Y' \in \cS^* \implies \cT_Y \cap \cT_{Y'} = \emptyset
\end{equation*}
and we let $\pi_{\hat X^*}: \sT_{\hat X^*} \lra \hat X^*,$ $\rho_{\hat X^*}: \sT_{\hat X^*} \lra \bbR^+$ be the obvious maps.
Note that by the frontier condition $\hat X^*$ is a closed subset of $\hat X.$

As explained above, Thom's first isotopy lemma implies that $\sT_{\hat X^*}\setminus \hat X^*$ is isomorphic, as a stratified space, to $K_{\hat X^*}\times (0,\eps)$ for a stratified space participating in a locally trivial fibration
\begin{equation*}
	\psi_{\hat X^*}: K_{\hat X^*} \lra \hat X^*.
\end{equation*}
For $a \in (0,\eps),$ let 
\begin{equation*}
	\hat X^a = \hat X \setminus \{ \zeta \in \sT_{Y^*} : \rho_{\hat X^*}(\zeta) < a \}.
\end{equation*}
This is an example of a stratified space with boundary (formalized below in Definition \ref{def:StratCorners}).
Note that the particular value of $a$ is immaterial since different values yield isomorphic stratified spaces $K_{\hat X^*} \times (a,\eps).$
We will not distinguish between different values of $a,$ but denote this space by
\begin{equation*}
	[\hat X; \hat X^*]
\end{equation*}
and refer to it as the {\em radial blow-up of $\hat X$ along $\hat X^*.$}

Note that the boundary $K_{\hat X^*}$ inherits from $\hat X$ the locally trivial fibration $\psi_{\hat X^*}$ and also a {\em collar neighborhood}.
By which we will mean 
an open neighborhood of $K_{\hat X^*},$ $\cC_{K_{\hat X^*}} \subseteq [\hat X; \hat X^*],$ and a map
\begin{equation*}
	\Theta: \cC_{K_{\hat X^*}}  \lra K_{\hat X^*} \times \bbR^+
\end{equation*}
which is a homeomorphism onto its image and such that
\begin{equation*}
\begin{gathered}
	\Theta(\zeta) = (\zeta, 0) \Mforall \zeta \in K_{\hat X^*} \Mand \\
	\Theta(Y \cap \cC_{K_{\hat X^*}}) \subseteq (Y \cap K_{\hat X^*}) \times \bbR^+ \Mforall Y \in \sS\setminus \sS^*.
\end{gathered}
\end{equation*}

Finally note that 
\begin{equation*}
	\depth( [\hat X; \hat X^*] ) = \depth(\hat X) - 1,
\end{equation*}
so that we have simplified the stratification. On the other hand, we can recover $\hat X$ by collapsing the fibers of the boundary fibration 
\begin{equation*}
	\hat X = [\hat X; \hat X^*] / \{ K_{\hat X} \ni \zeta \sim \psi_{\hat X^*}(\zeta) \in \hat X^* \}.
\end{equation*}
We denote the projection map by 
\begin{equation*}
	\beta: [\hat X; \hat X^*] \lra \hat X
\end{equation*}
and refer to it as the `blow-down' map.

Hence for many purposes working on $\hat X$ is equivalent to working on $[\hat X; \hat X^*]$ together with its boundary fibration $\psi_{\hat X^*}.$
The resolution of $\hat X$ is the result of iterating this process to replace strata by boundaries. 

\subsection{Stratifications with collared corners} \label{sec:StratCorners}

In this section we summarize part of \S5 of \cite{Verona:StratMap} (where the terminology `abstract stratifications with faces' is used).
By a manifold with corners we mean a smooth manifold locally modeled on $(\bbR^+)^\ell$ with embedded boundary hypersurfaces, e.g., see \cite{Melrose:Conormal}.
(For notational convenience, we allow the empty set as a boundary hypersurface.)
A manifold with corners always has (compatible) collar neighborhoods of its boundary hypersurfaces (see e.g., \cite[Proposition 1.2]{Albin-Melrose:Resolution}). By a manifold with collared corners we mean that we have fixed a choice of collars.

\begin{definition}\label{def:StratCorners}
A  {\em stratified space with collared corners} is a quintuple $(\hat X, \sS, \sT, \sF, \sC)$ in which $\hat X$ is the space, $\sS$ is a decomposition of $\hat X$ into `strata', $\sT$ is a collection of `tubes', one per stratum, $\sF$ are the `boundary hypersurfaces' of $\hat X,$ and $\sC$ is a collection of `collars', one per boundary hypersurface.
Specifically:
\begin{itemize}
\item $\hat X$ is a `nice' topological space.
\item $\sF$ is a collection of closed subsets of $\hat X$ and $\sS$ is a locally finite collection of locally closed subsets of $\hat X.$
\item $\sC$ consists, for each $K \in \sF,$ of an open neighborhood $\cU_K\subseteq \hat X$ of $K$ and a homeomorphism
\begin{equation*}
	\Theta_K: \cU_K \lra K \times \bbR^+
\end{equation*}
such that $\Theta_K\rest{K} = \id_K,$ and 
\begin{equation*}
	\Theta_K(Y \cap \cU_K) \subseteq (Y\cap K) \times \bbR^+ \Mforevery Y \in \sS.
\end{equation*}
The pair $(\cU_K, \Theta_K)$ is known as a collar of $K.$ We denote the composition of $\Theta_K$ with the projection onto $K$ by $\theta_K$ and the composition with the projection onto $\bbR^+$ by $x_K,$ so that
\begin{equation*}
	\Theta_K(\zeta) = (\theta_K(\zeta), x_k(\zeta)).
\end{equation*}
\item Each $Y \in \sS$ is a topological manifold with boundary in the induced topology and is provided with the structure of a smooth manifold with corners whose boundary hypersurfaces are $Y \cap K$ for $K \in \sF.$ These are collared by $\Theta_{K\cap Y} = \Theta_K\rest{Y \cap \cU_K}.$ 
The family $\sS$ satisfies the frontier condition
\begin{equation*}
	Y, Y' \in \sS \Mand Y \cap \bar{Y'} \neq\emptyset \implies Y \subseteq \bar{Y'}.
\end{equation*}
\item $\sT$ contains, for each $Y \in \sS,$ a triple $(\cT_Y, \pi_Y, \rho_Y)$ in which $\cT_Y$ is an open neighborhood of $Y$ in $\hat X,$ and
\begin{equation*}
	\pi_Y: \cT_Y \lra Y, \quad \rho_Y: \cT_Y \lra \bbR^+
\end{equation*}
are continuous functions such that $\pi_Y\rest{Y} = \mathrm{id}_Y$ and $\rho_Y^{-1}(0) = Y.$ The neighborhood $\cT_Y$ is known as a tubular neighborhood of $Y$ in $\hat X$ and the triple $(\cT_Y, \pi_Y, \rho_Y)$ as the control data of $Y$ in $\hat X.$
\end{itemize}
The control data are compatible in that if $Y, Y' \in \sS$ are distinct strata such that $\cT_{Y} \cap Y' \neq \emptyset$ then 
\begin{equation*}
	(\pi_Y, \rho_Y)\rest{\cT_Y \cap Y'}: \cT_Y\cap Y' \lra Y\times (0,\infty)
\end{equation*}
is a smooth submersion, and
\begin{equation*}
	\pi_Y\circ \pi_{Y'} = \pi_Y, \quad \rho_Y \circ \pi_{Y'} = \rho_Y
\end{equation*}
when both sides are defined (i.e., on $\cT_Y \cap \pi_{Y'}^{-1}(\cT_Y \cap Y')$).
Finally the control data are compatible with the collar data in that 
for all $Y \in \sS,$ and $K \in \sF$
\begin{equation*}
	\pi_Y^{-1}(Y\cap K) =\cT_Y \cap K, \quad
	\Theta_{Y\cap K} \circ \pi_X = (\pi_X \times \id_{\bbR^+})\circ \Theta_K, \quad 
	\rho_Y = \rho_Y \circ \theta_K
\end{equation*}
whenever both sides make sense.
\end{definition}

The depth, dimension, and regular part of a stratified space with collared corners are defined as before. The interior of a stratified space with collared corners is the complement of the boundary hypersurfaces and is itself a stratified space.

\begin{example}
Any stratified space $(\hat X, \sS, \sT)$ is naturally a stratified space with collared corners.

If $(\hat X, \sS, \sT, \sF, \sC)$ is a stratified space with collared corners and $L$ is a smooth manifold with collared corners then 
$\hat X \times L$ is naturally a stratified space with collared corners. Its strata are $Y \times L,$ $Y \in \sS,$ and its boundary hypersurfaces are $K \times L^\circ,$ $K \in \sF,$ and $\hat X^\circ \times H,$ with $H$ a boundary hypersurface of $L.$

In particular $\hat X \times \bbR^+$ is a stratified space with collared corners where $\bbR^+$ has boundary $\{0\}$ with collar the obvious map $\bbR^+ \lra \{0\} \times \bbR^+.$

If $(\hat X, \sS, \sT, \sF, \sC)$ is a stratified space with collared corners and $A \subseteq \hat X$ is either an open set, a closed subset that is a union of strata, or a boundary hypersurface, then $A$ is naturally a stratified space with strata $Y \cap A,$ $Y \in \sS,$ and boundary hypersurfaces $K \cap A,$ $K \in \sF.$
\end{example}

Two stratified spaces with collared corners $(\hat X, \sS, \sT, \sF, \sC)$ and $(\hat X', \sS', \sT', \sF', \sC')$ are {\em equivalent} if 
\begin{equation*}
	\hat X = \hat X', \quad
	\sS = \sS', \quad
	\sF = \sF',
\end{equation*}
the control data coincide in a neighborhood of each stratum, and the collar data coincide in a neighborhood of each boundary hypersurface.
Up to equivalence we can assume that
\begin{equation*}
	\cT_Y \cap \cT_{Y'} \neq\emptyset \implies Y\leq Y' \Mor Y'\leq Y.
\end{equation*}

If $(\hat X, \sS_X, \sT_X, \sF_X, \sC_X)$ and $(\hat N, \sS_N, \sT_N, \sF_N, \sC_N)$  are stratified spaces with collared corners and $f:\hat X \lra \hat N$ is a continuous map, we say that it is {\em compatible with faces} if we can decompose the boundary hypersurfaces into `horizontal' and `vertical',
\begin{equation*}
	\sF_X = \sF_X^h(f) \sqcup \sF_X^v(f), \quad \sF_N = \sF_N^h(f) \sqcup \sF_N^v(f)
\end{equation*}
so that
\begin{itemize}
\item [a)] the inverse image of a boundary hypersurface in $\sF_N^h(f)$ is either empty or all of $\hat X,$
\item [b)] $f$ induces a bijection between $\sF_X^v(f)$ and $\sF_N^v(f)$ and if $K \in \sF_N^v(f)$ then, near $f^{-1}(K),$
\begin{equation*}
	\Theta_{K}\circ f = ( f\rest{f^{-1}(K)}\times \id_{\bbR^+}) \circ \Theta_{f^{-1}(K)},
\end{equation*}
\item [c)] near a boundary hypersurface $H \in \sF_X^h(f),$
\begin{equation*}
	f = (f\rest{H}) \circ \theta_H.
\end{equation*}
\end{itemize}
A continuous map $f: \hat X \lra \hat N$ is a {\em morphism between stratified spaces with collared corners,} $(\hat X, \sS_X, \sT_X, \sF_X, \sC_X)$ and $(\hat N, \sS_N, \sT_N, \sF_N, \sC_N),$ if it is compatible with faces, sends strata to strata smoothly and intertwines control data (as in \S\ref{sec:ThomMather}).
A morphism is {\em submersive} if its restriction to each stratum is a submersion. A morphism is an {\em isomorphism} if it is a bijection and a morphism whose inverse is a morphism.

For example, if $Y, Y'$ are strata of a stratified space with collared corners with $Y<Y'$ then, after potentially shrinking $\cT_Y,$ the maps
\begin{equation*}
	\pi_Y\rest{\cT_Y \cap Y'}: \cT_Y \cap Y' \lra Y, \quad
	\rho_Y\rest{\cT_Y \cap Y'}: \cT_Y \cap Y' \lra \bbR
\end{equation*}
are compatible with faces, with all of the faces being vertical for $\pi_Y$ and horizontal for $\rho_Y.$

If $K$ is a boundary hypersurface of a stratified space with collared corners and $(\cU_K, \Theta_K)$ is its collar then, after potentially shrinking $\cU_K,$ the collar map
\begin{equation*}
	\Theta_K:\cU_K \lra K \times \bbR^+
\end{equation*}
is an isomorphism of stratified spaces with collared corners.

The notions of trivial fibration and locally trivial fibration are entirely analogous to \S\ref{sec:ThomMather}.

\begin{theorem}[{Thom's first isotopy lemma for stratified spaces with collared corners \cite[Corollary 5.8]{Verona:StratMap}}]
If $f: \hat X \lra L$ is a proper submersive morphism between a stratified space with collared corners $\hat X$ and a smooth manifold with collared corners $L,$ then $f$ is a locally trivial fibration.
\end{theorem}

Proceeding as in \S\ref{sec:ThomMather}, Thom's first isotopy lemma applied to the control data of a stratified space with collared boundary implies that such a space is, locally near each stratum, a bundle of cones over simpler stratified space with collared boundary.
As anticipated above, this structure allows us to iterate the radial blow-up construction from \S\ref{sec:ResolvingSingle}. Starting from a stratified space, blowing-up the strata of greatest depth produces a stratified space with collared boundary, and iterating produces a stratified space with collared corners. For spaces of finite depth this leads to what Verona called a `total decomposition' \cite[\S 6.3.1]{Verona:StratMap} and we call a resolution \cite{ALMP11, Albin-Melrose:Resolution}.

\subsection{The resolution of a stratified space} \label{sec:TotalRes}

In \cite{ALMP11}, we used Verona's constructions to prove an equivalence between the class of Thom-Mather stratified spaces and the class of manifolds with corners and iterated fibration structures, as introduced by Melrose. This resolution is already claimed in Thom \cite{Thom:Ensembles} but was not phrased in this way. As explained above, we have found the presentation of a stratified space as a manifold with corners and iterated fibration structures very useful for carrying out analysis.

By a {\em collective boundary hypersurface} we mean a finite union of non-intersecting boundary hypersurfaces.

\begin{definition}[{Melrose \cite{Albin-Melrose:Resolution, ALMP11}}] \label{def:ifs}
A {\em manifold with corners with an iterated fibration structure} if a triple $(L, \sF, \mathit{\Phi})$ where $L$ is a smooth manifold with corners, $\sF$ consists of disjoint collective boundary hypersurfaces with every boundary hypersurface of $L$ occurring in some element of $\sF,$ and $\mathit{\Phi}$ is a collection of locally trivial fibrations, one per element $H \in \sF,$
\begin{equation*}
	Z_H \fib H \xlra{\phi_H} Y_H
\end{equation*}
where $Z_H$ and $Y_H$ are smooth manifolds with corners. These satisfy:
\begin{itemize}
\item If $H, K \in \sF$ intersect then $\dim Y_H \neq \dim Y_K$ and we write 
\begin{equation*}
	H<K \Mif \dim Y_H < \dim Y_K.
\end{equation*}
\item if $H<K$ then $Y_K$ has a collective boundary hypersurface $Y_{HK}$ participating in a fiber bundle $\phi_{HK}:Y_{HK} \lra Y_H$ such that the diagram
\begin{equation*}
	\xymatrix{
	H\cap K \ar[rr]^-{\phi_K} \ar[rd]_-{\phi_H} & & Y_{HK} \ar[ld]^-{\phi_{HK}} \ar@{}[r]|-*[@]{\subseteq} & Y_K \\
	& Y_H & }
\end{equation*}
commutes.
\end{itemize}
\end{definition}

Two useful properties of these structures are:
(a) a manifold with corners and an iterated fibration structure admits collar neighborhoods compatible with the boundary fibrations, see \cite[Proposition 3.7]{Albin-Melrose:Resolution}, and
(b) the bases of the boundary fibrations inherit iterated fibration structures \cite[Lemma 3.4]{Albin-Melrose:Resolution}.

If $L$ is a smooth manifold with corners and $H$ is a boundary hypersurface then, since $H$ is embedded, there exists a smooth function 
\begin{equation*}
	x_H:L \lra [0,1]
\end{equation*}
with $H = x_H^{-1}(\{ 0\})$ such that $dx_H$ has no zeroes on $H.$
We call any such a function a {\em boundary defining function} for $H.$

Let $L$ and $M$ be smooth manifolds with corners
and let $\{ x_{H_i} \},$ $\{\wt x_{G_j} \}$ be complete sets of defining functions for the boundary hypersurfaces of $L$ and $M$ respectively. A smooth map $f: L \lra M$ is a {\em $\mathbf{b}$-map} if for each boundary hypersurface $G_j \subseteq M,$
\begin{equation*}
\begin{gathered}
	\text{ either } f^*\wt x_{G_j} = 0 \Mor \\
	f^*\wt x_{G_j} = a_{G_j} \prod x_{H_i}^{e(i,j)}, \quad \Mwith a_{G_j} \in \CI(L, (0,\infty)) \Mand e(i,j) \in \bbN_0. 
\end{gathered}
\end{equation*}
The first alternative occurs only if $f(L) \subseteq \pa M,$ in which case we say that $f$ is a boundary $b$-map; otherwise we say that $f$ is an interior $b$-map.

A {\em morphism between manifolds with corners and boundary fibration structures}, $(L, \sF_L, \mathit{\Phi}_L)$ and $(M, \sF_M, \mathit{\Phi}_M)$ is a $b$-map 
\begin{equation*}
	f: L \lra M
\end{equation*}
that sends the collective boundary hypersurface in $\sF_L$ to collective boundary hypersurfaces in $\sF_M$ such that, 
whenever $H \in \sF_L$ and $G \in \sF_M$ are such that $f(H) \subseteq G,$ we have
\begin{equation*}
	\xymatrix{
	H \ar[d]_-{\phi_H} \ar[r]^-{f} & G\ar[d]^-{\phi_G} \\
	Y_H \ar[r]^-{\bar f} & Y_G }
\end{equation*}
for some map $\bar f.$
(Note that the maps $\bar f$ inherit the structure of morphisms between manifolds with corners and boundary fibration structures.)

A smooth action of a Lie group on a compact manifold gives rise to a manifolds with corners and boundary fibration structures, see \cite{Albin-Melrose:Resolution}, but for us the main class of examples comes from resolving a stratified space.
Let us work out in detail how a stratified space gives rise to a manifold with corners and a boundary fibration structure in the two simplest non-trivial cases:\\

If $\hat X$ has a single singular stratum $Y,$ then blowing-up $Y$ as described in \S\ref{sec:ResolvingSingle} produces a smooth manifold with boundary $[\hat X; Y]$ with a boundary fibration over $Y.$ (Note that the boundary of $[\hat X;Y]$ need not be connected, which is why the definition involves collective boundary hypersurfaces.)
We denote the blown-up space by $\wt X$ and call it the resolution of $\hat X.$\\

If $\hat X$ has two singular strata $Y_1 < Y_2$ (assumed connected) then blowing-up $Y_1$ results in $[\hat X; Y_1],$ a depth one stratified space with boundary $K_1$ and a locally trivial fibration 
\begin{equation*}
	Z_1 \fib K_1 \xlra{\psi_{Y_1}} Y_1.
\end{equation*}
The stratum $Y_1$ is a closed smooth manifold while $K$ and $Z_1$ are stratified spaces (without boundary) of depth one.
The stratum $Y_2$ of $\hat X$ lifts (or restricts) to a stratum $\cY$ of $[\hat X;Y_1].$ 
The compatibility conditions between the strata imply that $K_1 \cap \cY$ is the singular stratum of $K$ and, for each $q\in Y_1,$ $\psi_{Y_1}^{-1}(q) \cap \cY$ is the singular stratum of $\psi_{Y_1}^{-1}(q).$
The theory developed in \S\ref{sec:StratCorners} allows us to blow-up $\cY$ in $[\hat X;Y_1]$ resulting in a manifold with corners $[[\hat X;Y_1];\cY].$

We denote this space by $\wt X$ and call it the resolution of $\hat X.$
It has two boundary hypersurfaces. One of them, $K_2,$ is created by the second blow-up and hence has a locally trivial fibration over $\cY,$
\begin{equation*}
	Z_2 \fib K_2 \xlra{\psi_{\cY}} \cY.
\end{equation*}
Note that $K_2$ and $\cY$ are smooth manifolds with boundary and $Z_2$ is a smooth manifold (without boundary).
In fact we can identify $\cY$ with $[\bar Y_2;Y_1]$ and, since $Y_1$ is the unique singular stratum in $\bar Y_2,$ this is precisely $\wt Y_2,$ the resolution of $\bar Y_2.$
Thus this boundary fibration
\begin{equation*}
	Z_2 \fib K_2 \xlra{\psi_{\cY}} \wt Y_2
\end{equation*}
has as base the resolution of $Y_2$ and as fiber the link of $Y_2$ in $\hat X.$

The other boundary hypersurface of $\wt X$ is $[K_1; K_1\cap \cY].$ The set $K_1 \cap \cY$ is transverse to the fibers of the fibration $\psi_{Y_1}$ and so this fibration survives the blow-up to yield a fibration
\begin{equation*}
	[\psi_{Y_1}^{-1}(q); \psi_{Y_1}^{-1}(q) \cap \cY] \fib [K_1; K_1\cap \cY] \xlra{\phi_{Y_1}} Y_1.
\end{equation*}
As explained above, the subsets being blown-up in the fiber and the total space of this fibration are precisely the singular parts of the respective sets. Hence we can identify these sets with the corresponding resolutions,
\begin{equation*}
	\wt Z_1 \fib \wt K_1 \xlra{\phi_{Y_1}} Y_1.
\end{equation*}

Finally note that the two boundary hypersurfaces $\wt K_1,$ $K_2$ intersect
\begin{equation*}
	\wt K_1\cap K_2 = \pa \wt K_1 = \pa K_2.
\end{equation*}
The boundary fibrations restricted to this intersection satisfy
\begin{equation*}
	\phi_{Y_1}(\wt K_1\cap K_2) = Y_1, \quad \psi_{\cY}(\wt K_1 \cap K_2) = K_1 \cap \cY = \pa \wt Y_2
\end{equation*}
and hence fit into the commutative diagram
\begin{equation*}
	\xymatrix
	{ \wt K_1 \cap K_2 \ar[rr]^{\psi_{\cY}} \ar[rd]_{\phi_{Y_1}} & & K_1 \cap \cY \ar[ld]^{\psi_{Y_1}} \\
	& Y_1 & }
\end{equation*}
Thus $\wt X$ has an iterated fibration structure and we can recover $\hat X$ by collapsing the fibers of $\psi_{\cY}$ and then collapsing the fibers of $\psi_{Y_1}.$
We denote the resulting projection by
\begin{equation*}
	\beta: \wt X \lra \hat X
\end{equation*}
and refer to it as the blow-down map.\\

This procedure can be carried out inductively on any compact space with a Thom-Mather stratification as in \cite[Propositions 2.3 and 2.5]{ALMP11}, yielding the following result.

\begin{theorem}\label{thm:Resolution}$ $
\begin{itemize}
\item [a)]
Let $\hat X$ be a compact Thom-Mather stratified space. 
Iteratively blowing-up the deepest stratum produces a manifold with corners with an iterated fibration structure $\wt X,$ known as the resolution of $\hat X,$ together with a continuous map
\begin{equation*}
	\beta: \wt X \lra \hat X
\end{equation*}
such that:
\begin{itemize}
\item[i)] $\beta$ restricts to a diffeomorphism between the interior of $\wt X$ and the regular part of $\hat X,$
\item[ii)] $\beta$ determines a bijection between the collective boundary hypersurfaces $\sF(\wt X)$ and the singular strata of $\hat X,$
\begin{equation*}
	\sS(\hat X) \ni Y \mapsto \bar{\beta^{-1}(Y)} \in \sF(\wt X),
\end{equation*}
\item[iii)] the boundary fibration on $\bar{\beta^{-1}(Y)} \in \sF(\wt X)$ has base space the resolution of $\bar Y,$ $\wt Y,$ and fiber over a point $q \in Y$ equal to the resolution of the link of $q$ in $\hat X.$
\end{itemize}

\item [b)]
Let $\wt X$ be a manifold with corners with an iterated fibration structure. Choosing compatible collar neighborhoods and then iteratively collapsing the fibers of the boundary fibration with largest dimensional base produces a Thom-Mather stratified space $\hat X,$ called the blow-down of $\wt X.$

\item [c)] These procedures are mutually inverse: the blow-down of the resolution of $\hat X$ is isomorphic to $\hat X,$ and the resolution of the blow-down of $L$ is isomorphic to $L.$ 
\end{itemize}
\end{theorem}

Resolution of stratified spaces is a functor (cf. \cite[Propositions 2.6, 2.7]{ALMP11}):
If $\hat X$ and $\hat N$ are stratified spaces with resolutions $\wt X$ and $\wt N,$ then any morphism $f: \hat X\lra \hat N$ lifts to a morphism $\wt f:\wt X \lra \wt N$ and compositions lift to compositions.
Blow-down is also a functor but from manifolds with corners and iterated fibration structures to {\em isomorphism classes} of stratified spaces because we do not fix collar neighborhoods of the boundary hypersurfaces.

`Controlled' objects on a stratified space correspond to objects on the resolution that are compatible with the boundary fibration structure. For example given a controlled vector field on $\hat X,$
\begin{equation*}
	\xi = \{ \xi_Y \in \CI(Y;TY): Y \in \sS(\hat X) \},
\end{equation*}
we get a vector field on $\wt X$ by lifting $\xi_{\hat X^{\reg}}$ and taking the unique continuous extension from $\wt X^{\circ}$ (which exists by the control conditions on $\xi$). Conversely, given a vector field $\eta$ on $\wt X,$ the collection of vector fields
\begin{equation*}
	\xi = \{ \beta_*(\eta\rest{\wt X^\circ}) \} \cup \{ (\phi_H)_*(\eta\rest{H}) : H \in \sF(\wt X) \}
\end{equation*}
is a controlled vector field on $\hat X.$

Similarly the class of functions
\begin{equation*}
	\CI_{\Phi}(\wt X) = \{ f \in \CI(\wt X) : f\rest{H} \in \phi_H^\bullet \CI(Y_H) \Mforall H \in \sF(\wt X) \}
\end{equation*}
corresponds to the `controlled functions' on $\hat X$ from \cite[\S 1.2.8]{Verona:StratMap}. (See \cite[\S4]{Zucker:ReductiveLp} for a discussion of controlled vector bundles and connections and their corresponding Chern classes.)\\

Let us mention another approach to obtaining the resolution of a stratified space. After blowing-up the deepest stratum as in \S\ref{sec:ResolvingSingle}, one can make use of the collar of the boundary to define the {\em double of }$[\hat X; \hat X^*].$ This will be a Thom-Mather stratified space (without boundary) whose depth is now less than the depth of $\hat X.$ It comes equipped with an action of $\bbZ_2$ and a fundamental domain for this action recovers $[\hat X; \hat X^*].$ Inductively one can show that the double can be resolved consistently with the $\bbZ_2$ action and then restricting to a fundamental domain yields $\wt X.$ This doubling technique is developed in \cite{Brasselet-Hector-Saralegi:DeRham} where it is called {\em d\'eplissage}. This approach to resolution was carried out in \cite{ALMP11}.\\

Finally let us comment on the use of collective boundary hypersurfaces in Definition \ref{def:ifs}. A stratified space if called {\em normal} \cite[\S4]{Goresky-MacPherson:IH} if all of its links are connected. In the resolution of a normal stratified space it is not necessary to use collective boundary hypersurfaces and each boundary hypersurface is endowed with a boundary fibration. On the other hand, given a boundary fibration structure with collective boundary hypersurfaces, one can restrict the boundary fibrations to the individual boundary hypersurfaces. This corresponds to switching from the resolution of $\hat X$ to the resolution of its {\em normalization} (as in \cite[\S4.1]{Goresky-MacPherson:IH}). For the purposes of carrying out analysis on the regular part of $\hat X,$ normalization makes only a notational difference and so for simplicity we often assume that the collective boundary hypersurfaces consist of a single boundary hypersurface. This is the approach taken in \cite{ALMP11}.

\subsection{A smooth de Rham complex computing intersection cohomology}

In \cite{Brylinski}, Brylinski introduced a de Rham complex due to Goresky-MacPherson on the regular part of a stratified space that computed the intersection cohomology (cf. \cite{Pollini, Brasselet-Hector-Saralegi:DeRham, Brasselet}). This complex has a very natural description once we are working on the resolution of the stratified space. We give a brief description of this complex as an example of the usefulness of resolving a stratified space.

First, given a locally trivial fiber bundle of smooth manifolds with corners
\begin{equation*}
	F \fib H \xlra{\varphi} B,
\end{equation*}
let us say that a differential form $\omega$ on $H$ has {\em $\varphi$-vertical degree at most $p$} if
\begin{equation*}
	\Mforall V_1, \ldots, V_{p+1} \in \CI(H, TH/B), \quad
	\df i_{V_1}\ldots \df i_{V_{p+1}} \omega = 0.
\end{equation*}

Let $(\hat X, \sS, \sT)$ be a Thom-Mather stratified space, $(\wt X, \sF, \mathit{\Phi})$ its resolution and 
\begin{equation*}
	\beta:\wt X \lra \hat X
\end{equation*}
the blow-down function. 
Given $Y \in \sS$ let us denote the collective boundary hypersurface associated to $Y$ (i.e., the set $\bar{\beta^{-1}(Y)}$) by $M_Y,$ and by
\begin{equation*}
	\phi_Y: M_Y \lra \wt Y, \quad j_Y:M_Y \hookrightarrow \wt X
\end{equation*}
the boundary fibration of $M_Y$ and its inclusion into $\wt X,$ respectively.

Given $\bar p,$  a perversity function on $\hat X,$ let
\begin{multline*}
	\Omega_{\bar p}^*(\wt X)
	= \{ \omega \in \CI(\wt X; \Lambda^\bullet (T^*\wt X)) : 
	\text{ for each } Y \in \sS, \\
	j_Y^*\omega \Mand j_Y^*d\omega 
	\text{ have $\phi_Y$-vertical degree at most } \bar p(\mathrm{codim}_{\hat X}(Y)) \}.
\end{multline*}
It is immediate that $\Omega_{\bar p}^*(\wt X)$ defines a complex,
\begin{equation*}
	0 \lra
	\Omega_{\bar p}^0(\wt X) \xlra{d}
	\Omega_{\bar p}^1(\wt X) \xlra{d}
	\ldots \xlra{d}
	\Omega_{\bar p}^n(\wt X) \lra 0.
\end{equation*}

\begin{theorem}[Brylinski-Goresky-MacPherson]\label{thm:DeRhamPerversity}
Let $\hat X$ be a Thom-Mather pseudomanifold and $\bar p$ a perversity. 
The cohomology of the complex $\Omega_{\bar p}^*(\wt X)$ is isomorphic to the intersection cohomology of $\hat X$ with perversity $\bar p$ and real coefficients. 
\end{theorem}

\begin{proof}
Define a sheaf complex $\sc{\bOm}_{\bar p}$ over $\hat X$ by
\begin{equation*}
	\hat X \supseteq 
	\cU \mapsto \{ \omega \in \Omega_{\bar p}^*(\wt X) : \supp \omega \subseteq \beta^{-1}(\cU) \}.
\end{equation*}
That the sheaf $\sc{\bOm}_{\bar p}$ is quasi-isometric to $\IC p$ as complexes of sheaves of $\bbR$-vector spaces will follow after we check the axioms ($a$)-($d$) from \S\ref{sec:RefIH}. 

Adopting the notation from \S\ref{sec:RefIH}, 
first note that the existence of an appropriate partition of unity  in $\CI_{\Phi}(\wt X)$ implies that the sheaves are soft. The normalization and lower bound are automatic, so we need to check that, whenever $\cU\subseteq \hat X$ is a distinguished neighborhood of a point $x \in S_{n-k},$ we have
\begin{equation*}
	\begin{cases}
	\tH^j( \sc{\bOm}_{\bar p} (\cU) ) = 0 & \Mif j>\bar p(k) \\
	\tH^j( \sc{\bOm}_{\bar p} (\cU) ) \xlra{\cong}
	\tH^j( \sc{\bOm}_{\bar p} (\cU\setminus S_{n-k}) ) & \Mif j\leq \bar p(k)
	\end{cases}
\end{equation*}
where the map in the latter condition is induced by restriction.

If $\cU$ is as above and $\wt \cU = \beta^{-1}(\cU)$ then
\begin{equation*}
	\wt \cU \cong \bbB^{n-k} \times [0,1) \times \wt Z_x
\end{equation*}
where $\wt Z_q = \beta^{-1}(q)$ is the resolution of $Z_q,$ hence a smooth manifold with corners,
and $\sc{\bOm}_{\bar p} (\cU)$ is, directly from its definition, given by
\begin{equation*}
	\sc{\bOm}_{\bar p}(\cU) = \{ \omega \in \Omega^*(\bbB^{n-k} \times [0,1)_x \times \wt Z_x ) :
	\omega\rest{x=0}, d\omega\rest{x=0} \in 
	\Omega^*(\bbB^{n-k} \times [0,1)) \wedge \sc{\bOm}_{\bar p}(\wt Z_x)\rest{\mathrm{deg}\leq \bar p(k)} \}.
\end{equation*}
Since we are on a smooth manifold with corners, standard de Rham theory implies that the cohomology of this complex vanishes above degree $\bar p(k)$ and that, for other degrees, restriction to 
\begin{equation*}
	\wt \cU\setminus Y_{n-k} \cong \bbB^{n-k} \times (0,1) \times \wt Z_x
\end{equation*}
is an isomorphism in cohomology; indeed that the cohomology is equal to
\begin{equation*}
	\bH^j(\sc{\bOm}_{\bar p}(\cU) ) =
	\begin{cases}
	\bH^j( \sc{\bOm}_{\bar p}(\wt Z_x) ) & \Mif j \leq \bar p(k) \\
	0 & \Mif j > \bar p(k)
	\end{cases}
\end{equation*}
\end{proof}

\section{Borel-Serre compactifications and their resolutions} \label{sec:BorelSerre}

The Borel-Serre compactification and what is now known as the reductive Borel-Serre compactification of a locally symmetric space into a manifold with corners and a stratified space, respectively, were introduced in \cite{Borel-Serre:Corners} and \cite{Zucker:L2Coho}. In this section we will define another compactification of a locally symmetric space into a manifold with corners, the {\em resolved Borel-Serre compactification.} We will show that while the Borel-Serre compactification does not have an iterated fibration structure, the resolved Borel-Serre compactification does and that it can be obtained by resolving the reductive Borel-Serre compactification in the sense of Theorem \ref{thm:Resolution}.

It may be interesting to compare the resolved Borel-Serre compactification to other compactifications of locally symmetric spaces.

\subsection{Linear algebraic groups}

We review enough of the theory of linear algebraic groups to describe compactifications of locally symmetric spaces. We follow \cite{Borel-Serre:Corners, Borel-Ji:Book} to which we refer for proofs and longer explanations; see also, e.g., \cite{Goresky-Harder-MacPherson, Saper:CohoLocSymSpaces, Zucker:OnBdyCohoLocSymVar}.

A linear algebraic group is a Zariski-closed subgroup of some $GL_n(\bbC).$ If it can be defined by polynomials with $\bbQ$-coefficients we say that it is defined over $\bbQ.$ We will denote such groups by bold letters such as $\bG$. Their real points are then denoted by $\bG(\bbR)$ or the non-bold letter $G$ and form a real Lie group. The connected component of the identity will be denoted $\bG^\circ.$

A linear algebraic group $\bT$ is called an {\em algebraic torus} if it is isomorphic to a product of copies of $\bbC^*.$ If the isomorphism is defined over a field $\bbF,$ we say that it is {\em split} over $\bbF.$

An element $r \in \bG$ is {\em unipotent} if some positive power of $(r-\Id)$ is equal to the identity, and a linear algebraic group is called unipotent if all of its elements are unipotent. Every unipotent group is solvable.

The {\em radical} $\bR(\bG)$ of an algebraic group $\bG$ is the maximal connected normal solvable subgroup of $\bG,$ and the {\em unipotent radical,} $\bR_{\bU}(\bG),$ is the maximal connected unipotent normal subgroup of $\bG.$ If $\bG$ is defined over $\bbQ,$ then so are $\bR(\bG)$ and $\bR_{\bU}(\bG).$ We say that $\bG$ is {\em semisimple} if $\bR(\bG)$ is trivial and {\em reductive} if $\bR_{\bU}(\bG)$ is trivial.
We denote the real points in $\bR_{\bU}(\bG)$ by $R_UG.$

A {\em Levi subgroup} $\bL$ of $\bG$ is a maximal reductive subgroup; any two such are conjugate under $\bR_{\bU}(\bG)$ and $\bG = \bL \ltimes \bR_{\bU}(\bG).$ By a Levi subgroup $L$ of $G = \bG(\bbR)$ we will mean the real points of a Levi subgroup $\bf L$ of $\bG.$
Every compact subgroup of $G$ is the group of real points of a reductive $\bbR$-group and is hence contained in a Levi subgroup.

The maximal tori of $\bG,$ split over a field $\bbF,$ are all conjugate over $\bbF,$ and their common dimension is called the $\bbF$-rank of $\bG,$ $\mathrm{rk}_{\bbF}(G).$ A group $\bG$ is {\em split over $\bbQ$} if its $\bbQ$-rank coincides with its $\bbC$-rank. A group $\bG$ is {\em anisotropic over $\bbQ$} if it is of $\bbQ$-rank zero. 

A closed subgroup $\bP$ of a connected linear algebraic group $\bG$ is called a {\em parabolic subgroup} if $\bG/\bP$ is a projective variety or if $\bP$ contains a maximal connected solvable subgroup (known as a {\em Borel subgroup}).
By a parabolic subgroup $P$ of $G$ we will mean the real points of a parabolic subgroup $\bP$ of $\bG,$ defined over $\bbQ.$
A subgroup $H$ of $G$ is parabolic if and only if $G/H$ is compact.
We denote the set of {\em proper} parabolic subgroups of $G$ by $\cP(G).$

If $K$ is a maximal compact subgroup of $G$ and $P$ is a parabolic subgroup of $G$ then $K_P = K\cap P$ is a maximal compact subgroup of $P$ and $G = K \cdot P.$

If $\bG$ is reductive and $K$ is a maximal compact subgroup of $G,$ there exists one and only one involutive automorphism $\theta_K$ of $G$ whose fixed point set is $K$ and which is algebraic (in a suitable sense \cite[Proposition 1.6]{Borel-Serre:Corners}). This automorphism is called the {\em Cartan involution} of $G$ with respect to $K.$

If $\bN$ is a normal subgroup of $\bG$ defined over $\bbR,$ then $\theta_K(\bN(\bbR)) = \bN(\bbR).$
If $P$ is a parabolic subgroup of $G$ and $L$ a Levi subgroup of $G$ containing $K,$ then $L\cap P$ contains a unique Levi subgroup of $P$ stable under $\theta_K.$\\

We will assume from now on that $\bG$ is a connected reductive linear algebraic group defined over $\bbQ$ whose center is (defined and) anisotropic over $\bbQ.$\\

For any parabolic subgroup $\bP$ of $\bG$ defined over $\bbQ$ both $\bU_{\bP},$  the unipotent radical of $\bP,$ and $\bL_{\bP} = \bP/\bU_{\bP},$ the {\em Levi quotient} of $\bP,$ are also defined over $\bbQ.$ 
Let $\chi(\bL_{\bP})$ denote the group of algebraic rationally defined characters of $\bL_{\bP}$ and
\begin{equation*}
	\bM_{\bP} = \bigcap_{\alpha \in \chi(\bL_{\bP})}\ker \alpha^2.
\end{equation*}
Then $\bM_{\bP}$ is a reductive algebraic group defined over $\bbQ$ whose center is anisotropic over $\bbQ$ (it may fail to be connected but see \cite[Remark 2.10]{Borel-Ji:CompLocSym}).
Let $\bS_\bP$ be the maximal $\bbQ$-split torus in the center of $\bL_{\bP},$ and denote
\begin{equation*}
	U_P = \bU_{\bP}(\bbR), \quad L_P = \bL_{\bP}(\bbR), \quad M_{\bP} = \bM_{\bP}(\bbR), \quad S_{\bP} = \bS_{\bP}(\bbR), \quad 
	A_{\bP} = S_P^{\circ}.
\end{equation*}
There is a canonical isomorphism
\begin{equation*}
	A_{\bP} = (\bbR^*_+)^{\mathrm{park}_{\bG}(\bP)}
\end{equation*}
where $\mathrm{park}_{\bG}(\bP),$ the parabolic rank of $\bP$ in $\bG,$ is the length of a maximal chain
\begin{equation*}
	\bP = \bP_{\mathrm{park}_{\bG}(\bP)} \subset \bP_{\mathrm{park}_{\bG}(\bP)-1} \subset \cdots \subset \bP_1 \subset G
\end{equation*}
of proper parabolic subgroups defined over $\bbQ.$
The real points of the Levi quotient admit a decomposition
\begin{equation*}
	L_P = A_{\bP} M_{\bP} \cong A_{\bP} \times M_{\bP}.
\end{equation*}

For each choice of maximal compact subgroup $K$ of $G,$ there is a unique lift $L_{P,\theta_K}\subseteq P$ of $L_P$ to a subgroup of $P$ stable under $\theta_K.$ This lift splits the exact sequence
\begin{equation*}
	1 \lra U_P \lra P \lra L_P \lra 1.
\end{equation*}
It is possible to choose $K$ so that $L_{P, \theta_K}$ is defined over $\bbQ,$ and then so are the images $A_{\bP, \theta_K}$ and $M_{\bP, \theta_K}$ of its subgroups $A_{\bP}$ and $M_{\bP}.$
This yields the {\em rational Langlands decomposition} of $P:$
\begin{equation*}
	P = U_P A_{\bP,\theta_K} M_{\bP,\theta_K} \cong U_P \times A_{\bP,\theta_K} \times M_{\bP,\theta_K}.
\end{equation*}

The intersection of $K$ and $P,$ $K_P = K\cap P$ is a maximal compact subgroup of $P.$ Its image under the map $P \mapsto L_P$ is a maximal compact subset of $M_P$ and the lift of this image coincides with $K_P.$

Next, let us consider the parabolic subgroups of $\bG,$ defined over $\bbQ,$ contained in $\bP.$ These are the parabolic subgroups of $\bP$ and \cite[\S4.1]{Borel-Serre:Corners} are in one-to-one correspondence with the parabolic subgroups of $\bL_{\bP}$ (in fact $\bM_{\bP}$) by the projection $\bP \lra \bP/\bU_{\bP}.$ 
If $\bP, \bQ \in \cP(\bG)$ and $\bQ\subseteq \bP$ then 
\begin{equation*}
	Q_{\bP} = Q \cap M_{\bP,\theta_K} \in \cP(\bM_{\bP}).
\end{equation*}
The Langlands decompositions of $Q$ (with respect to $\theta_K$) and $Q_{\bP}$ (with respect to $\theta_{K\cap P}$),
\begin{equation*}
	Q = U_Q A_{\bQ,\theta_K} M_{\bQ,\theta_K}, \quad Q_{\bP} = U_{Q_\bP} A_{\bQ_{\bP}, \theta_{K \cap P}} M_{\bQ_\bP, \theta_{K\cap P}},
\end{equation*}
are related by
\begin{equation}\label{eq:LangRel}
	U_Q = U_P \rtimes U_{Q_P}, \quad 
	A_{\bQ,\theta_K} = A_{\bP, \theta_K} \times A_{\bQ_\bP, \theta_{K \cap P}}, \quad 
	M_{\bQ, \theta_K} = M_{\bQ_\bP, \theta_{K\cap P}}.
\end{equation}
Note that $K_Q = K \cap M_{\bQ,\theta_K} = K_{Q_P}.$ Below we will omit the Cartan involution $\theta$ in the notation for the Langlands decompositions.\\

The notion of an {\em arithmetic} subgroup of $\bG(\bbQ)$ is a sort of `substitute' for $\bG(\bbZ).$ Given a representation $\rho:\bG \lra \bG\bL_n$ defined over $\bbQ$ with discrete null space (i.e., an `almost faithful' representation), we can set
\begin{equation*}
	\bG(\bbZ)_{\rho} = \rho^{-1}(\bG\bL_n(\bbZ)).
\end{equation*}
If $\rho'$ is any other almost faithful representation then $\bG(\bbZ)_{\rho'}$ is commensurable with $\bG(\bbZ)_{\rho},$ meaning that the intersection of these two groups is of finite index in either one. We call any group $\Gamma<\bG(\bbQ)$ commensurable with $\bG(\bbZ)_{\rho}$ an arithmetic subgroup.

An arithmetic subgroup $\Gamma$ is {\em neat} if, whenever $\bH_2<\bH_1$ is a pair of algebraic $\bbQ$-subgroups of $\bG,$
\begin{equation*}
	\Gamma_{\bH_1/\bH_2} = (\Gamma\cap \bH_1(\bbQ)) / (\Gamma \cap \bH_2(\bbQ))
\end{equation*}
is torsion-free (and arithmetic). Every arithmetic group contains a neat normal subgroup of finite index.

\subsection{The partial compactification of a symmetric space to a manifold with corners}

As above, we assume that $\bG$ is a connected reductive linear algebraic group defined over $\bbQ$ whose center is (defined and) anisotropic over $\bbQ.$ Let $K$ be a maximal compact subgroup of $G$ and let
\begin{equation*}
	X = G/K
\end{equation*}
be the corresponding symmetric space. In this section we describe the Borel-Serre partial compactification \cite{Borel-Serre:Corners} of $X$ to a non-compact manifold with corners, $\bar X.$ We will also describe a `resolution' of $\bar X$ to a non-compact manifold with corners $\wt X$ with a boundary fibration structure.\\

Given $P \in \cP(G),$ the group $K_P = K \cap P$ is a maximal compact subgroup of $P.$ 
It coincides with $K \cap M_{\bP}$ and is also a maximal compact subgroup of $M_{\bP}.$
The quotient space
\begin{equation*}
	X_P = M_{\bP}/K_P = P/K_PA_{\bP}U_P
\end{equation*}
is a symmetric space of non-compact type for $M_{\bP}$ called the {\em boundary symmetric space} associated to $P.$

For each $P \in \cP(G)$ let
\begin{equation*}
	e_G(P) = U_P \times X_P.
\end{equation*}
The Borel-Serre partial compactification of $X,$ as a set, is 
\begin{equation}\label{eq:CloseX}
	\bar X = X \cup \bigsqcup_{P \in \cP(G)} e_G(P),
\end{equation}
where $e_G(P)$ is attached to $X$ using the horospherical decomposition $X = U_P \times A_{\bP} \times X_{\bP}$ (see \cite{Borel-Serre:Corners,Borel-Ji:Book}) and is a paracompact Hausdorff manifold with corners. 
Each $e_G(P)$ is the interior of a boundary face of $\bar X,$ with codimension equal to the parabolic rank of $P,$ and with closure in $\bar X$ given by
\begin{equation}\label{eq:CloseEP}
	\bar{e_G(P)} = e_G(P) \cup \bigsqcup_{\substack{Q \in \cP(G) \\ Q \subseteq P}} e_G(Q).
\end{equation}
Note that $K_Q = K \cap M_Q = K_{Q_P},$ and hence
the boundary symmetric space associated to $Q$ as a subgroup of $G$ is the same as the boundary symmetric space associated to $Q_P$ as a subgroup of $M_P,$
\begin{equation*}
	X_{Q} = M_Q/K_Q = M_{Q_P}/K_{Q_P} = X_{Q_P}.
\end{equation*}
$ $

\begin{proposition}[Boundary fibrations of the partial Borel-Serre compactification]\label{prop:PartialBS}
Let $X$ be the symmetric space of the real points of a connected reductive group $\bG$ as above 
and let $\bar X$ be its Borel-Serre partial compactification to a manifold with corners.
Each boundary face of $\bar X$ is of the form
\begin{equation*}
	\bar{e_G(P)} = U_P \times \bar{X_P}
\end{equation*}
where $P \in \cP(G),$ $U_P$ is the unipotent radical of $P,$ $X_P$ is the boundary symmetric space corresponding to $P$ and $\bar{X_P}$ is its Borel-Serre partial compactification. Let $\phi_P: \bar{e_G(P)} \lra \bar{X_P}$ denote the natural projection.

If $P, Q \in \cP(G)$ then the boundary faces $\bar{e_G(P)},$ $\bar{e_G(Q)}$ intersect if and only if 
\begin{equation*}
	R = P\cap Q \in \cP(G)
\end{equation*}
in which case $\bar{e_G(P)} \cap\bar{e_G(Q)} = \bar{e_G(R)}$ and we have a commutative diagram,
\begin{equation}\label{eq:PBSFibs}
	\xymatrix{ 
	\bar{e_G(Q)} \ar[d]_{\phi_Q} & & \bar{e_G(R)}
	\ar@{^(->}[rr] \ar@{_(->}[ll]  \ar[ld]_{\phi_Q\rest{\bar{e_G(R)}}} \ar[rd]^{\phi_P\rest{\bar{e_G(R)}}} \ar[dd]^{\phi_R}  
	& & \bar{e_G(P)} \ar[d]^{\phi_P} \\
	\bar{X_Q} & \ar@{_(->}[l] \bar{e_{M_Q}(R_Q)} \ar[rd]_{\phi_{R_Q}} & &
	\bar{e_{M_P}(R_P)} \ar@{^(->}[r] \ar[ld]^{\phi_{R_P}} & \bar{X_P} \\
	& & \bar{X_R} & & }
\end{equation}
\end{proposition}

\begin{proof}
We have described the closure $\bar{e_G(P)}$ of $e_G(P)$ in $\bar X$ in \eqref{eq:CloseEP}. The boundary symmetric space $X_P$ itself has a partial compactification as a symmetric space for $M_P,$
\begin{equation*}
	\bar{X_P} = X_P \cup \bigsqcup_{O \in \cP(M_P)} e_{M_P}(O)
	= X_P \cup \bigsqcup_{\substack{R \in \cP(G) \\ R \subseteq P}} e_{M_P}(R_P),
\end{equation*}
where $e_{M_P}(R_P) = U_{R_P} \times X_{R_P}.$
Thus to each $R \in \cP(G),$ $R \subseteq P,$ we assign a boundary face $e_G(R)$ of $\bar{e(P)}$ and a boundary face $e_{M_P}(R_P)$ of $\bar{X_P},$ given respectively by
\begin{equation*}
\begin{gathered}
	e_G(R) = U_R \times X_R = U_P \times U_{R_P} \times X_R, \\
	e_{M_P}(R_P) = U_{R_P} \times X_{R_P} = U_{R_P} \times X_R.
\end{gathered}
\end{equation*}
It follows that
\begin{multline*}
	\bar{e_G(P)} = 
	e_G(P) \cup \bigsqcup_{\substack{R \in \cP(G) \\ R \subseteq P}} e_G(R)
	= U_P \times X_P 
	\cup \bigsqcup_{\substack{R \in \cP(G) \\ R \subseteq P}} (U_P \times U_{R_P} \times X_R) \\
	= U_P \times \Big( X_P 
	\cup \bigsqcup_{\substack{R \in \cP(G) \\ R \subseteq P}} (U_{R_P} \times X_R) \Big)
	= U_P \times \Big( X_P 
	\cup \bigsqcup_{\substack{R \in \cP(G) \\ R \subseteq P}} e_{M_P}(R_P) \Big)
	= U_P \times \bar{X_P}.
\end{multline*}

Next let $P, Q \in \cP(G).$ Corollary 7.4 of \cite{Borel-Serre:Corners}, shows that
\begin{equation*}
	\bar{e_G(P)}\cap \bar{e_G(Q)} =
	\begin{cases}
	\bar{e_G(R)} & \Mif R= P\cap Q \in \cP(G) \\
	\emptyset & \text{otherwise}
	\end{cases}
\end{equation*}
From \eqref{eq:LangRel} we see that 
\begin{equation*}
	\bar{e_G(R)} = U_R \times \bar{X_R} = U_P \times U_{R_P} \times \bar{X_R}
\end{equation*}
and, since $\phi_P: U_P \times \bar{X_P} \lra X_P$ is just the projection off of $U_P,$ 
\begin{equation*}
	\phi_P(\bar{e_G(R)}) = U_{R_P} \times \bar{X_R} = \bar{e_{M_P}(R_P)}.
\end{equation*}
Moreover
\begin{equation*}
	\bar{e_G(R)} \xlra{\phi_P\rest{\bar{e_G(R)}}} \bar{e_{M_P}(R_P)} \xlra{\phi_{R_P}} \bar{X_R}
\end{equation*}
is the projection off of $U_R$ and so coincides with $\phi_R.$
This gives that the right `half' of \eqref{eq:PBSFibs} commutes, and by symmetry so does the left half.

\end{proof}

Proposition \ref{prop:PartialBS}, particularly diagram \eqref{eq:PBSFibs}, shows that the Borel-Serre partial compactification $\bar X$ does not have the structure of a manifold with corners with an iterated fibration structure (except when it is a manifold with boundary).
It is always possible to modify $\bar X$ to obtain an iterated fibration structure. 

Recall (e.g., \cite[Chapter 5]{Melrose:Corners}, \cite{Albin-Melrose:Resolution}) that if $M$ is a manifold with corners and $F\subseteq M$ is a boundary face, then 
\begin{equation*}
	[M;F],
\end{equation*}
the radial blow-up of $M$ along $F$ is a manifold with corners in which $F$ is replaced by its inward-pointing spherical normal bundle, forming a new boundary hypersurface (i.e., a boundary face of codimension one). There is a natural `blow-down map' $\beta: [M;F] \lra M$ which collapses the new boundary hypersurface back to $F.$

Consider the simple case of two maximal proper parabolic subgroups $P$ and $Q$ such that $R=P\cap Q$ is itself parabolic.
The space $[\bar X; \bar{e_G(R)}],$ replaces $\bar{e_G(R)}$ with a new boundary hypersurface
\begin{equation*}
	\bar{E_G(R)}.
\end{equation*}
The blow-down map $\beta: [\bar X; \bar{e_G(R)}] \lra \bar X$ restricts to $\bar{E_G(R)}$ to a (trivial) fiber bundle
\begin{equation*}
	\bbS_+^{\mathrm{park}(R)-1} \fib \bar{E_G(R)} \xlra{\beta\rest{\bar{E_G(R)}}} \bar{e_G(R)}
\end{equation*}
since $\bar{E_G(R)}$ is the total space of the inward pointing spherical normal bundle of $\bar{e_G(R)}$ in $\bar X.$
We can relabel this in terms of the rational Langlands decomposition of $P.$ Indeed, recall \cite{Borel-Serre:Corners}, that $\bar{e_G(R)}$ has a neighborhood of the form
\begin{equation*}
	{X(R)} = X \times_{A_R} \bar{A_R}
\end{equation*}
where $A_R$ acts on $X$ by a `geodesic action', and that we can identify $A_R \cong (\bbR^*_+)^{\mathrm{park}(R)}.$
There is an action of $\bbR^*_+$ on $A_R$ by scaling in all factors of $\bbR^*_+$ and 
\begin{equation*}
	\bbS_+^{\mathrm{park}(R)-1} \cong A_R/\bbR^*_+.
\end{equation*}
Thus the new boundary hypersurface $E_G(R)$ participates in a (trivial) fiber bundle
\begin{equation*}
	A_R/\bbR^*_+ \fib \bar{E_G(R)} \xlra{\beta\rest{\bar{E_G(R)}}}  \bar{e_G(R)}
\end{equation*}
as well as a (trivial) fiber bundle over $\bar{X_R},$
\begin{equation*}
	U_P \times (A_R/\bbR^*_+ ) \fib \bar{E_G(R)} \xlra{\Phi_P} \bar{X_R}.
\end{equation*}
By blowing-up $\bar{e_G(R)},$ we have separated the boundary hypersurfaces $\bar{e_G(P)}$ and $\bar{e_G(Q)}$ of $\bar X.$ Now at the corner of the intersection of $\beta^{-1}(\bar{e_G(P)})$ and $\bar{E_G(R)}$ we have a commutative diagram of fiber bundle maps
\begin{equation*}
	\xymatrix{
	\beta^{-1}(\bar{e_G(P)})\cap \bar{E_G(R)} \ar[rr]^-{\phi_P} \ar[rd]_-{\Phi_R} & &\bar{e_{M_P}(R_P)} \ar[ld]^-{\phi_{R_P}} \\
	& \bar{X_R} &
	}
\end{equation*}
as required in the definition of a boundary fibration structure.

In order to meet the requirements of a boundary fibration structure at every corner, we need to blow-up every boundary face of $\bar X$ of codimension greater than one. This procedure is referred to as the {\em total boundary blow-up} in \cite[\S 5.13]{Melrose:Corners}, \cite[\S2.6]{Hassell-Mazzeo-Melrose:Surgery} and denoted by a subscript `tb'.
We define the {\em resolved partial Borel-Serre compactification} of $X$ to be
\begin{equation*}
	\wt X = \bar X_{tb}.
\end{equation*}
It has a natural blow-down map $\beta_{tb}:\wt X \lra \bar X,$ and one boundary hypersurface $\bar{E_G(P)}$ for each proper parabolic subgroup of $G,$ where
\begin{equation*}
	E_G(P) = \beta_{tb}^{-1}\lrpar{e_G(P)}, \quad \bar{E_G(P)} =\bar{\beta_{tb}^{-1}\lrpar{e_G(P)}}.
\end{equation*}
(In general $\bar{E_G(P)} \neq \beta_{tb}^{-1}\lrpar{\bar{e_G(P)}}$ since the latter contains one boundary hypersurface for each parabolic subgroup of $G$ contained in $P.$)

\begin{proposition}[Boundary fibration structure of the resolved partial Borel-Serre compactification]\label{prop:ResPartialBS}
Let $X$ be the symmetric space of the real points of a connected reductive group $\bG$ as above 
and let $\wt X$ be its resolved Borel-Serre partial compactification to a manifold with corners.
Each boundary hypersurface of $\wt X$ is of the form
\begin{equation*}
	\bar{E_G(P)} = U_P \times (A_P/\bbR^*_+)_{tb} \times \wt{X_P}
\end{equation*}
where $P \in \cP(G),$ $U_P$ is the unipotent radical of $P,$ $A_P$ is the connected component of the identity in the real points of the split center of $U_P\diagdown P$ over $\bbQ,$ $X_P$ is the boundary symmetric space corresponding to $P$ and $\wt{X_P}$ is its resolved Borel-Serre partial compactification.

If $P, Q \in \cP(G)$ then $\bar{E_G(P)} \cap \bar{E_G(Q)} \neq\emptyset$ if and only if one of the parabolic groups is a subgroup of the other.

The natural projections $\Phi_P:\bar{E_G(P)} \lra \wt{X_P}$ endow $\wt X$ with an iterated fibration structure. That is, whenever $\bar{E_G(P)}\cap \bar{E_G(Q)} \neq \emptyset,$ we have (after relabeling if necessary) $\mathrm{park}(P)< \mathrm{park}(Q)$ and 
a commutative diagram of fiber bundle maps
\begin{equation*}
	\xymatrix{
	\bar{E_G(P)} \cap \bar{E_G(Q)} \ar[rr]^-{\Phi_P} \ar[rd]_-{\Phi_Q} & & \bar{E_{M_P}(Q_P)} \ar[ld]^-{\Phi_{Q_P}} \\
	& \wt{X_Q} &
	}
\end{equation*}
\end{proposition}

\begin{proof}
The fact that $\bar X$ is a Hausdorff manifold with corners implies the same for $\wt X.$

For each $k\in \bbN,$ let $\bar X_{tb>k}$ denote the space obtained from $\bar X$ by blowing-up all of the boundary faces of codimension greater than $k$ and let
\begin{equation*}
	\beta_{tb>k}:\bar X_{tb>k} \lra \bar X
\end{equation*}
denote the blow-down map.
Thus
\begin{equation*}
	\wt X = \bar X_{tb} = \bar X_{tb>1}.
\end{equation*}

Consider $P \in \cP(G)$ of parabolic rank $\ell.$ The face $\bar{e_G(P)}$ has codimension $\ell$ in $\bar X,$ so if
\begin{equation*}
	F_{G}(P) = \beta_{tb>\ell}^{-1}(e_G(P))
\end{equation*}
then $\bar{F_G(P)}$ is obtained from $\bar{e_G(P)}$ by blowing-up all of its boundary faces in order of increasing codimension. Thus we have
\begin{equation*}
	\bar{F_G(P)} = \bar{e_G(P)}_{tb} = U_P \times (\bar{X_P})_{tb} = U_P \times \wt{X_P}.
\end{equation*}
When we pass to $\bar X_{tb>\ell-1},$ the boundary face $\bar{F_{G}(P)}$ itself gets blown-up so it lifts to a boundary hypersurface
\begin{equation*}
	H_G(P) = \beta_{tb>\ell-1}^{-1}(e_G(P))
\end{equation*}
whose closure is
\begin{equation*}
	\bar{H_G(P)} = \bar{F_G(P)} \times (\bbR^*_+)^{\ell}/\bbR^*_+.
\end{equation*}
Finally when we pass to $\bar X_{tb}$ we blow-up the intersection of $\bar{H_G(P)}$ with (lifts of) boundary faces of $\bar X$ corresponding to parabolic subgroups of $X$ that contain $P,$ and these intersect $\bar{H_G(P)}$ in $\bar{F_G(P)}$ times a boundary face of $(\bbR^*_+)^{\ell}/\bbR^*_+$ so we end up with
\begin{equation*}
	\bar{E_G(P)} = \bar{F_G(P)} \times \lrpar{(\bbR^*_+)^{\ell}/\bbR^*_+}_{tb}
	= U_P \times (A_P/\bbR^*_+)_{tb} \times \wt{X_P}.
\end{equation*}

From this description we see that the boundary hypersurfaces that intersect $\bar{E_G(P)}$ come either from the blow-ups of faces $e_G(Q)$ of $Q$ that are contained in $P$ (accounting for the boundary faces of $\bar{F_G(P)}$) or from the blow-ups of faces $e_G(Q)$ of $Q$ that contain $P$ (accounting for the boundary faces of $\lrpar{(\bbR^*_+)^{\ell}/\bbR^*_+}_{tb}$).

The compatibility of the boundary fibrations over a corner now follows directly from \eqref{eq:PBSFibs}, as explained above.
\end{proof}

\subsection{The compactification of a locally symmetric space to a manifold with corners}

By Theorem 9.3 of \cite{Borel-Serre:Corners}, any arithmetic subgroup $\Gamma$ of $G$ acts properly on $\bar X$ with compact orbit space $\Gamma\diagdown \bar X,$ known as the {\em Borel-Serre compactification of $\Gamma\diagdown X.$} Similarly any arithmetic subgroup of $M_P$ acts properly on $\bar{X_P}$ with compact orbit space.

If $\Gamma$ is torsion-free it will act freely on $X,$ but not necessarily on $\bar X.$ However if $\Gamma$ is also neat, then it will act freely on $\bar X.$

For $P \in \cP$ let
\begin{equation*}
	\Gamma_P = P \cap \Gamma, \quad \Gamma_{U_P} = U_P \cap \Gamma.
\end{equation*}
These are arithmetic subgroups of $P$ and $U_P,$ respectively.
From \cite[Proposition 1.2]{Borel-Serre:Corners}, we know that the image of $\Gamma_P$ under the projection $\pi_{U_P}:P \lra U_P \diagdown P$ is contained in $M_P$ and is an arithmetic subgroup of $M_P$ which we will denote $\Gamma_{M_P}.$ These groups fit into a short exact sequence
\begin{equation*}
	0\lra \Gamma_{U_P}
	\lra \Gamma_P
	\xlra{\pi_{U_P}} \Gamma_{M_P} \lra 0.
\end{equation*}
Since $U_P$ is normal in $P,$ the product decomposition $e_G(P) = U_P \times X_P$ descends to a fiber bundle
\begin{equation*}
	\Gamma_{U_P} \diagdown U_P \fib 
	\Gamma_P \diagdown e_G(P)
	\xlra{\psi_P}
	\Gamma_{M_P} \diagdown X_P.
\end{equation*}
with fiber a compact nilmanifold. We write the total space as
\begin{equation*}
	e'_G(P) = (\Gamma_{U_P} \diagdown U_P) \times_{\Gamma_{M_P}} X_P
\end{equation*}
where $\Gamma_{M_P}$ acts on the first factor by conjugation.

The stabilizer of $e_G(P)$ under the action of $\Gamma$ on $\bar X$ is $\Gamma_P,$ and hence $e'_G(P)$ is the image of $e_G(P)$ under the projection $\bar X \lra \Gamma \diagdown \bar X.$
Thus we have
\begin{equation*}
	\Gamma\diagdown\bar{X} = \Gamma\diagdown X \cup \bigsqcup_{P \in \cP(G)/\Gamma} e'_G(P)
\end{equation*}
where $\cP(G)/\Gamma$ denotes a set of representatives of the conjugacy classes of elements of $\cP(G)$ under the action of $\Gamma.$
The image of the closure of $e_G(P)$ is equal to the closure of the image of $e_G(P)$ and is given by
\begin{equation*}
	\bar{e'_G(P)} = \Gamma_P \diagdown \bar{e_G(P)} = e'_G(P) \cup \bigsqcup_{Q \in \cP(M_P)/\Gamma_{M_P}}e'_G(Q).
\end{equation*}

\begin{proposition}[Boundary fibrations of the Borel-Serre compactification]
Let $X$ be the symmetric space of the real point of a connected reductive group $\bG$ as above, let $\Gamma$ be a neat arithmetic subgroup of $\bG(\bbQ),$ and let $\Gamma\diagdown\bar{X}$ be the Borel-Serre compactification of $\Gamma\diagdown X$ to a manifold with corners.
Each boundary face of $\Gamma\diagdown\bar{X}$ is the total space of a fiber bundle
\begin{equation*}
	\bar{e_G'(P)} = (\Gamma_{U_P}\diagdown U_P) \times_{\Gamma_{M_P}} \bar{X_P} 
	\xlra{\psi_P} \Gamma_{M_P}\diagdown \bar {X_P}
\end{equation*}
with fiber the nil manifold $\Gamma_{U_P}\diagdown U_P$ and base the Borel-Serre compactification of $\Gamma_{M_P}\diagdown  {X_P}.$

If $P, Q \in \cP(G)/\Gamma$ then the boundary faces $\bar{e'_G(P)},$ $\bar{e'_G(Q)}$ intersect if and only if, after replacing  $P$ with a $\Gamma$-conjugate subgroup,
\begin{equation*}
	R = P\cap Q \in \cP(G)
\end{equation*}
in which case $\bar{e'_G(P)} \cap\bar{e'_G(Q)} = \bar{e'_G(R)}$ and we have a commutative diagram,
\begin{equation}\label{eq:BSFibs}
	\xymatrix{ 
	\bar{e'_G(Q)} \ar[d]_{\psi_Q} & & \bar{e'_G(R)}
	\ar@{^(->}[rr] \ar@{_(->}[ll]  \ar[ld]_{\psi_Q\rest{\bar{e'_G(R)}}} \ar[rd]^{\psi_P\rest{\bar{e'_G(R)}}} \ar[dd]^{\psi_R}  
	& & \bar{e'_G(P)} \ar[d]^{\psi_P} \\
	\Gamma_{M_Q}\diagdown \bar{X_Q} & \ar@{_(->}[l] \bar{e'_{M_Q}(R_Q)} \ar[rd]_{\psi_{R_Q}} & &
	\bar{e'_{M_P}(R_P)} \ar@{^(->}[r] \ar[ld]^{\psi_{R_P}} & \Gamma_{M_P}\diagdown \bar{X_P} \\
	& & \Gamma_{M_R}\diagdown \bar{X_R} & & }
\end{equation}

\end{proposition}

\begin{proof}
If $\bP, \bQ \in \cP(G)$ and $\bQ \subseteq \bP,$ then from \eqref{eq:LangRel} we have a short exact sequence
\begin{equation*}
	1\lra U_P \lra U_Q \lra U_{Q_P} \lra 1,
\end{equation*}
and hence a fiber bundle 
\begin{equation*}
	\Gamma_{U_P} \diagdown U_P \fib
	\Gamma_{U_Q} \diagdown U_Q 
	\lra
	\Gamma_{U_{Q_P}} \diagdown U_{Q_P}.
\end{equation*}
In particular, we can identify 
\begin{equation*}
	\Gamma_{U_Q} \diagdown U_Q = (\Gamma_{U_P} \diagdown U_P) \times_{\Gamma_{U_{Q_P}}} U_{Q_P}
\end{equation*}
and hence
\begin{equation*}
	e'_G(Q) = (\Gamma_{U_Q} \diagdown U_Q) \times_{\Gamma_{M_Q}} X_Q
	= \lrpar{ (\Gamma_{U_P} \diagdown U_P) \times_{\Gamma_{U_{Q_P}}} U_{Q_P} } \times_{\Gamma_{M_Q}} X_Q.
\end{equation*}
Then using the short exact sequence
\begin{equation*}
	1 \lra \Gamma_{U_{Q_P}} \lra \Gamma_{Q_P} \lra \Gamma_{M_Q} \lra 1
\end{equation*}
we can further identify this with
\begin{equation*}
	e'_G(Q) = (\Gamma_{U_P} \diagdown U_P) \times_{\Gamma_{Q_P}} \lrpar{U_{Q_P}\times X_Q}
	= (\Gamma_{U_P} \diagdown U_P) \times_{\Gamma_{Q_P}} e_{M_Q}(Q_P).
\end{equation*}

The space $\bar{X_P}$ is a symmetric space for $M_P$ and under the action of $\Gamma_{M_P}$ on $\bar{X_{P}}$ the stabilizer of $e_{M_P}(R_P)$ is $\Gamma_{M_{R_P}},$ so we have
\begin{equation*}
	\Gamma_{M_P} \diagdown \bar{X_P} = 
	\Gamma_{M_P} \diagdown X_P \cup 
	\bigsqcup_{\substack{R \in \cP(G)/\Gamma \\ R \subseteq P}} \Gamma_{M_{R_P}}\diagdown e_{M_P}(R_P).
\end{equation*}
Thus 
\begin{multline*}
	\bar{e'_G(P)} = 
	e'_G(P) \cup \bigsqcup_{\substack{R \in \cP(G) \\ R \subseteq P}} e'_G(R)
	= (\Gamma_{U_P} \diagdown U_P) \times_{\Gamma_{M_P}} X_P 
	\cup \bigsqcup_{\substack{R \in \cP(G) \\ R \subseteq P}} 
	(\Gamma_{U_P} \diagdown U_P) \times_{\Gamma_{R_P}} e_{M_R}(R_P) \\
	= (\Gamma_{U_P} \diagdown U_P) \times_{\Gamma_{M_P}}
	\bar{X_P}.
\end{multline*}
This proves the first part of the proposition, the rest follows from Proposition \ref{prop:PartialBS} and the equivariance of the projections $\phi_P.$
\end{proof}

Just as the action of $\Gamma$ on $X$ extends to an action of $\Gamma$ on $\bar X,$ there is a unique action of $\Gamma$ on $\wt X$ that extends the action of $\Gamma$ on $X.$
Indeed, for compact spaces and groups this was proved in \cite[Proposition 6.2]{Albin-Melrose:Resolution}, and the extension to proper actions is straightforward (see, e.g., \cite{Kankaanrinta} for collar and tubular neighborhood theorems).

We define the {\em resolved Borel-Serre compactification of $\Gamma\diagdown X$} to be
\begin{equation*}
	\Gamma\diagdown \wt{X}.
\end{equation*}
Let $\bar{E_G'(P)}$ denote the image of $\bar{E_G(P)}$ under the natural projection $\wt X \lra \Gamma\diagdown \wt X$

\begin{theorem}[Boundary fibrations of the resolved Borel-Serre compactification]
Let $X$ be the symmetric space of the real point of a connected reductive group $\bG$ as above, let $\Gamma$ be a neat arithmetic subgroup of $\bG(\bbQ),$ and let $\Gamma\diagdown\wt{X}$ be the resolved Borel-Serre compactification of $\Gamma\diagdown X.$ 

There is one boundary hypersurface $\bar{E_G'(P)}$ for each $\Gamma$-conjugacy class of $\bbQ$-parabolic subgroups of $\bG.$ Each of these boundary hypersurfaces is the total space of a fiber bundle
\begin{equation*}
	\bar{E_G'(P)} = (A_P/\bbR^*_+)_{tb} \times \lrpar{(\Gamma_{U_P}\diagdown U_P) \times_{\Gamma_{M_P}} \wt{X_P} }
	\xlra{\Psi_P} \Gamma_{M_P}\diagdown \wt {X_P}
\end{equation*}
with fiber the compact manifold with corners $(A_P/\bbR^*_+)_{tb} \times \Gamma_{U_P}\diagdown U_P$ and base the resolved Borel-Serre compactification of $\Gamma_{M_P}\diagdown  {X_P}.$ Here $U_P$ denotes the unipotent radical of $P,$ $A_P$ is the connected component of the identity in the real points of the split center of $U_P\diagdown P$ over $\bbQ,$ and $X_P$ is the boundary symmetric space corresponding to $P.$

If $P, Q \in \cP(G)/\Gamma$ then $\bar{E'_G(P)} \cap \bar{E'_G(Q)} \neq\emptyset$ if and only if one of the parabolic groups is $\Gamma$-conjugate to a subgroup of the other.

The fiber bundle maps $\Psi_P: \bar{E_G'(P)} \lra \Gamma_{M_P}\diagdown \wt {X_P}$ endow $\Gamma\diagdown\wt{X}$ with an iterated fibration structure. That is, whenever $\bar{E'_G(P)}\cap \bar{E'_G(Q)} \neq \emptyset,$ we have (after relabeling if necessary) $\mathrm{park}(P)< \mathrm{park}(Q)$ and 
a commutative diagram of fiber bundle maps
\begin{equation*}
	\xymatrix{
	\bar{E'_G(P)} \cap \bar{E'_G(Q)} \ar[rr]^-{\Psi_P} \ar[rd]_-{\Psi_Q} & & \bar{E'_{M_P}(Q_P)} \ar[ld]^-{\Psi_{Q_P}} \\
	& \Gamma_Q \diagdown \wt{X_Q} &
	}
\end{equation*}

Moreover, the resolved Borel-Serre compactification is the total boundary blow-up of the Borel-Serre compactification,
\begin{equation*}
	\Gamma\diagdown \wt X = \lrpar{\Gamma\diagdown \bar{X}}_{tb}.
\end{equation*}

\end{theorem}

\begin{proof}
In terms of the rational Langlands decomposition of a parabolic subgroup
\begin{equation*}
	P = U_P A_{\bP} M_{\bP}
\end{equation*}
the face $E_G(P)$ of $\wt X$ is obtained by taking the quotient of $A_P$ by a diagonal $\bbR^*_+$-action and the quotient of $M_{\bP}$ by the action of a maximal compact subgroup.
It is this decomposition of $P$ that induces the decomposition 
\begin{equation*}
	\bar{E_G(P)} = U_P \times (A_P/\bbR^*_+)_{tb} \times \wt{X_P}.
\end{equation*}
Thus if we take a quotient of this face by $\Gamma_{U_P} = \Gamma\cap U_P$ it will act entirely on the fist factor,
\begin{equation*}
	\Gamma_{U_P}\diagdown \bar{E_G(P)}  = (\Gamma_{U_P}\diagdown U_P)  \times (A_P/\bbR^*_+)_{tb} \times \wt{X_P}.
\end{equation*}
Subsequently taking the quotient by the action of $\Gamma_{M_P}$ yields
\begin{equation*}
	\Gamma_P \diagdown \bar{E_G(P)}
	= \lrpar{ (\Gamma_{U_P}\diagdown U_P)  \times (A_P/\bbR^*_+)_{tb} } \times_{\Gamma_{M_P}} \wt{X_P}
\end{equation*}
with the group $\Gamma_{M_P}$ acting on the first factor by conjugation. However, since $A_P$ is in the center of $L_P,$ this action does not affect it and we can pull-out the factor of $A_P$ to obtain
\begin{equation*}
	\Gamma_P \diagdown \bar{E_G(P)}
	= (A_P/\bbR^*_+)_{tb} \times \lrpar{ (\Gamma_{U_P}\diagdown U_P)    \times_{\Gamma_{M_P}} \wt{X_P} }
\end{equation*}
as required.

Identifying this face with the lift of the face $\bar{e'_G(P)}$ after performing the total boundary blow-up of $\Gamma\diagdown \bar X$ can be done as in the proof of Proposition \ref{prop:ResPartialBS}.

The boundary fibration structure is inherited from that of $\wt X$ thanks to the equivariance of the maps $\Phi_P.$
\end{proof}

\subsection{The compactification of a locally symmetric space to a stratified space}
In \cite{Zucker:L2Coho}, Zucker introduced the {\em reductive Borel-Serre compactification} of a locally symmetric space.
This is the stratified space $\hat M$ obtained from $\Gamma\diagdown \bar X$ by collapsing the fibers of the fiber bundle maps
\begin{equation*}
	\Gamma_{U_P}\diagdown U_P \fib e'_G(P) \lra \Gamma_{M_P}\diagdown X_P.
\end{equation*}
Note that the natural map $\Gamma\diagdown \bar X \lra \hat M$ restricts to a diffeomorphism of the interior, $\Gamma\diagdown X,$ and the regular part of $\hat M.$

Alternatively, one obtains the same space by collapsing the fibers of the fiber bundle maps
\begin{equation*}
	\Gamma_{U_P}\diagdown U_P \fib \bar{e'_G(P)} \lra \Gamma_{M_P}\diagdown \bar{X_P}
\end{equation*}
in an order of non-decreasing codimension.
Similarly, one could start with the space $\Gamma\diagdown \wt X$ and collapse the fibers of the fiber bundle maps
\begin{equation*}
	(A_P/\bbR^*_+)_{tb} \times \Gamma_{U_P}\diagdown U_P \fib \bar{E'_G(P)} \lra \Gamma_{M_P}\diagdown \bar{X_P}
\end{equation*}
in an order of non-decreasing codimension and again end up with $\hat M.$
This is precisely the procedure outlined above for going from a manifold with corners and an iterated fibration structure to a stratified space.

\begin{corollary}
Let $X$ be the symmetric space of the real points of a connected reductive group $\bG$ as above, let $\Gamma$ be a neat arithmetic subgroup of $\bG(\bbQ),$ and let $\hat M$ be the reductive Borel-Serre compactification of $\Gamma\diagdown X.$ 

The resolution of $\hat M$ to a manifold with corners and an iterated fibration structure is $\Gamma\diagdown \wt X,$ the resolved Borel-Serre compactification of $\Gamma\diagdown X.$
\end{corollary}


\end{document}